%% file: fdistn_mnds.tex
\documentclass[12pt]{article}

\input{pre}

\begin{document}

\author{\normalsize  Rory B. B. Lucyshyn-Wright\thanks{The author gratefully acknowledges financial support in the form of an AARMS Postdoctoral Fellowship and, earlier, an NSERC Postdoctoral Fellowship.}\let\thefootnote\relax
\\
\small Mount Allison University, Sackville, New Brunswick, Canada}

\title{\large \textbf{Functional distribution monads in functional-analytic contexts}}

\date{}

\maketitle

\abstract{
We give a general categorical construction that yields several monads of measures and distributions as special cases, alongside several monads of filters.  The construction takes place within a categorical setting for generalized functional analysis, called a \textit{functional-analytic context}, formulated in terms of a given monad or algebraic theory $\T$ enriched in a closed category $\V$.  By employing the notion of \textit{commutant} for enriched algebraic theories and monads, we define the \textit{functional distribution monad} associated to a given functional-analytic context.  We establish certain general classes of examples of functional-analytic contexts in cartesian closed categories $\V$, wherein $\T$ is the theory of $R$-modules or $R$-affine spaces for a given ring or rig $R$ in $\V$, or the theory of \textit{$R$-convex spaces} for a given preordered ring $R$ in $\V$.  We prove theorems characterizing the functional distribution monads in these contexts, and on this basis we establish several specific examples of functional distribution monads.
}

\section{Introduction} \label{sec:intro}

Through work of Lawvere \cite{Law:ProbMap}, \'Swirszcz \cite{Swi}, Giry \cite{Giry} and many others it has become clear that various kinds of measures and distributions give rise to monads; see, for example, \cite{JoPl,Dob,Lu:AlgThVectInt,Kock:Dist,Lu:PhD,Ave}.  The earliest work in this regard considered monads $\MM$ for which each free $\MM$-algebra $MV$ is a space of probability measures on a space $V$, be it a measurable space \cite{Law:ProbMap,Giry} or a compact Hausdorff space \cite{Swi}, for example.  In the literature, one does not find a monad capturing \textit{arbitrary} measures on a given class of spaces, whereas one can capture measures of \textit{compact support} \cite[7.1.7]{Lu:PhD} or \textit{bounded support} \cite{Lu:AlgThVectInt}, as well as Schwartz distributions of compact support \cite[7.1.6]{Lu:PhD}\cite{QueRey:Dist}\cite[II.3.6]{MoeRey}.

In the present paper, we give a general categorical construction which yields several (new and old) monads of measures and distributions as special cases.  The general construction takes place within an abstract axiomatic setting for generalized functional analysis, called a \textit{functional-analytic context} (\ref{def:fa_ctxt}, \ref{par:func_an_ctxt_mnds}), and in such a context we define an associated monad that we call the \textit{functional distribution monad} \pref{def:cdisnt_mnd}.  By considering various particular contexts, we obtain several particular monads of measures or distributions, as well as certain monads of \textit{filters}\footnote{Formal connections between the ultrafilter monad and notions of distribution or integral are noted in \cite{Kock:MndExtQu,Lei:CodUF}.}, as instances of an abstract construction with a functional-analytic flavour.

Our starting point is the \textit{Riesz-Schwartz dualization paradigm}, which has perhaps its purest expression in \textit{cartesian closed} categories.  Given a commutative ring object $R$ in a cartesian closed category $\V$ with finite limits, one can form for each object $V$ of $\V$ a canonical `function space', namely the internal hom $[V,R]$, and we define 
\begin{equation}\label{nat_distn_mnd_for_rmod}DV = \Mod{R}([V,R],R)\end{equation}
to be the subobject of $[[V,R],R]$ described by the equations that characterize $R$-linear morphisms $\mu:[V,R] \rightarrow R$.  This construction was employed in the context of synthetic differential geometry \cite{QueRey:Dist,Kock:ProResSynthFuncAn,MoeRey}, and it yields a monad $\DD$ on $\V$.  When $\V$ and $R$ are suitably chosen, the space of all compactly supported Radon measures on a locally compact Hausdorff space $V$ is recovered as an example of one of the free $\DD$-algebras $DV$ \cite[7.1.6]{Lu:PhD}, and one can similarly capture compactly supported Schwartz distributions on a smooth manifold $V$; see \cite[7.1.6]{Lu:PhD} and \cite[II.3.6]{MoeRey}.  But in these examples, the space $DV$ is by no means locally compact (nor, respectively, a smooth manifold), and so one must embed the categories of locally compact spaces and smooth manifolds into larger categories (e.g. convergence spaces, Fr\"olicher spaces, diffeological spaces, and various toposes in synthetic differential geometry; see \S \ref{sec:background}).  

A generalization of the formula \eqref{nat_distn_mnd_for_rmod} was employed by Kock \cite{Kock:Dist} and by the author \cite{Lu:PhD}, wherein for a given commutative $\V$-enriched monad $\TT$ on any suitable closed category $\V$ we set
\begin{equation}\label{eq:nat_dist_mnd_for_t}DV = \Alg{\TT}([V,S],S)\;,\end{equation}
where $S$ is some chosen $\TT$-algebra and $[V,S]$ is the cotensor of $S$ by $V$ in the $\V$-enriched category $\Alg{\TT}$ of $\TT$-algebras.  The resulting \textit{natural distribution monad} $\DD$ (in the terminology of \cite{Lu:PhD}) was then considered alongside other given $\V$-monads that may capture other notions of distribution or measure, e.g. the \textit{abstract distribution monads} of \cite{Lu:PhD}.  The category of $\TT$-algebras supports a form of abstract functional analysis, including not only function spaces and dualization (afforded by the commutativity of $\TT$), but also completeness and density \cite{Lu:PhD,Lu:ComplClDens}, and for this reason $\TT$-algebras in this context were called \textit{linear spaces} in \cite{Lu:PhD}.

However, as we shall see, if the category of $\TT$-algebras is to be regarded as capturing an abstract form of functional analysis, and if the monad $\DD$ is to capture an associated notion of distribution or measure, then the generalization of \eqref{nat_distn_mnd_for_rmod} captured by \eqref{eq:nat_dist_mnd_for_t} is rather too direct for some purposes.  Indeed, an important further class of examples of commutative monads $\TT$ whose algebras capture a form of generalized functional analysis are those whose algebras are \textit{convex spaces} of one kind or another (see \cite{Meng,Kei:PrCpctOrd,Fri,Jac:Conv,Lu:CvxAffCmt} and \ref{def:rcvx_sp}), yet the associated natural distribution monad $\DD$ in such cases does not capture a recognizable notion of measure or distribution, regardless of the choice of $S$.  On the other hand, there is an important connection between convex spaces and \textit{probability measures} for which evidence can be found in \cite{Swi,Meng,Dob,Kei:PrCpctOrd,Fri}.  In the present paper, we will uncover a new and important facet of this connection, and this will allow us to modify the formula \eqref{eq:nat_dist_mnd_for_t} and thus capture Radon probability measures as a basic feature of the generalized functional analysis of convex spaces.

Indeed, we shall establish a different generalization of the Riesz-Schwartz paradigm \eqref{nat_distn_mnd_for_rmod}, one that is more subtle and robust than \eqref{eq:nat_dist_mnd_for_t} and allows us to capture important examples that are not available by way of \eqref{eq:nat_dist_mnd_for_t}, including spaces of Radon probability measures.  Again we shall begin with a symmetric monoidal closed category $\V$ and a given commutative $\V$-enriched monad $\TT$ on $\V$, whose associated $\V$-category of $\TT$-algebras is construed as a setting for an abstract form of functional analysis.  Again we shall assume that we are given a $\TT$-algebra $S$ to play the role of `dualizing object'.  But this time we shall assume further that $\TT$ is a \textit{$\J$-ary monad} for a given \textit{eleutheric system of arities} $\J \hookrightarrow \V$ in the sense of \cite{Lu:EnrAlgTh}, generalizing the notion of \textit{finitary $\V$-monad} defined by Kelly \cite{Ke:FL}.  For each object $V$ of $\V$, we define the \textit{object of functional distributions} on $V$ as
\begin{equation}\label{eq:func_distn}D_{\scriptscriptstyle(\J,\TT,S)}(V) = \Alg{\TTperpJ}([V,S],S)\end{equation}
wherein the role that was played by $\TT$ in \eqref{eq:nat_dist_mnd_for_t} is now played instead by an associated $\V$-monad $\TTperpJ$ called the \textit{$\J$-ary commutant} of $\TT$ with respect to $S$, introduced in \cite[10.8]{Lu:Cmtnts}.  The quadruple $(\V,\J,\TT,S)$ is called a \textbf{functional-analytic context} if it satisfies a further axiom \pref{par:func_an_ctxt_mnds}, in which case we call the resulting $\V$-monad $\DD_{\scriptscriptstyle(\J,\TT,S)}$ the \textbf{functional distribution monad} in the given context.  We can describe $\J$-ary monads $\TT$ equivalently as enriched algebraic theories $\T$ with arities $\J$ \cite[4.1, 11.8]{Lu:EnrAlgTh}, also called \textit{$\J$-theories}, and we shall often employ this viewpoint in studying functional distribution monads (\ref{def:fa_ctxt}, \ref{def:cdisnt_mnd}).

Notably, the basic Riesz-Schwartz formula for $R$-modules \eqref{nat_distn_mnd_for_rmod} can be recovered as an instance of this formula \eqref{eq:func_distn} because of the following key result (\ref{thm:cmtnt_th_lrmods}, \ref{thm:th_lrmods_sat_bal_iff_rcomm}), applicable in any suitable cartesian closed category $\V$ with a commutative (unital) ring object $R$, or more generally a commutative rig\footnote{By a \textit{rig} or \textit{semiring} we mean a set $R$ equipped with two monoid structures $(R,+,0)$ and $(R,\bullet,1)$ such that $+$ is commutative and $\bullet:R^2 \rightarrow R$ preserves $+$ and $0$ in each variable separately---i.e. a unital ``ring without negatives''.  Rig objects and their modules are discussed in \ref{def:rig_ring_module_bim_cmons} and \ref{sec:rmods}.} object:
$$
\begin{minipage}{5.2in}
\textit{In the case where $\TT$ is the $\V$-monad whose algebras are $R$-modules in $\V$, there is a natural choice of a system of arities $\J$ for which the $\J$-ary commutant $\TTperpJ$ of $\TT$ with respect to $R$ is just $\TT$ itself; i.e. $\TTperpJ \cong \TT$.  Hence we say that the $\J$-ary monad $\TT$ is \textnormal{\textbf{balanced}} with respect to $R$.
}
\end{minipage}
$$
Indeed, although $R$-modules in $\V$ are \textit{not} the algebras of an ordinary $\Set$-based Lawvere theory \cite[\S 1]{BoDay}, they are the algebras of a $\V$-enriched algebraic theory of the type studied by Borceux and Day \cite{BoDay}, where the arities are finite cardinals, or equivalently, `finite discrete' objects $1+1+...+1$ in $\V$.  Consequently, the resulting $\V$-monad $\TT$ is a $\J$-ary monad for a suitable system of arities $\J = \DFin_\V \hookrightarrow \V$ consisting of the finite discrete objects of $\V$ \pref{par:efth}.  We call $\J$-theories for this particular system of arities \textit{discretely finitary theories}, and their corresponding $\J$-ary monads we call \textit{discretely finitary $\V$-monads}.  Correspondingly, any functional-analytic context of the form $(\V,\DFin_\V,\TT,S)$ will be called a \textit{(discretely) finitary functional-analytic context}.

In this paper, we develop certain general classes of examples of finitary functional-analytic contexts $(\V,\DFin_\V,\TT,S)$ in cartesian closed categories $\V$, wherein $\TT$-algebras are, respectively, $R$-modules or $R$-affine spaces in $\V$ for a given rig $R$ in $\V$, or $R$-\textit{convex spaces} for a given preordered ring $R$ in $\V$ \pref{exa:cls_exa_ffa_ctxts}.  We prove results that characterize the functional distribution monads in these cases (\ref{thm:charn_fdistns_in_sc_rlin_ctxt}, \ref{thm:cmtnt_th_raff_sp}, \ref{thm:enr_commutant_theorem}, \ref{thm:pos_rcvx_distn_charn_thm}).

On this basis, we establish the following specific examples of the functional distribution monad $\DD_{\scriptscriptstyle(\J,\TT,S)}$ associated to a finitary functional-analytic context.  We write $\Conv$ for the category of convergence spaces \cite{BeBu}, $\Fro$ for the category of Fr\"olicher's smooth spaces \cite{Fro:SmthStr,Fro:CccAnSmthMaps,FroKr}, $\Diff$ for the category of diffeological spaces (see, e.g., \cite{BaHo,Sta}), $\LCHaus$ for the category of locally compact Hausdorff topological spaces, and $\Cah$ for the Cahiers topos \cite{Dub:ModSDG,Kock:SDG} (or indeed any topos of sheaves on a \textit{product-closed $C^\infty$-site}, \ref{par:cinfty_rings}).
\begin{center}
\begin{tabular}{ |>{\raggedright}p{7ex} |>{\raggedright}p{19.5ex} | p{3ex} |>{\raggedright}p{20.1ex} |>{\raggedright}p{23.4ex}|}
\hline
                    &               &         &                  & \tabularnewline [-1em]
$\V$                & $\Alg{\TT}$    & $S$     & $\Alg{\TTperpJ}$ & $D_{\scriptscriptstyle(\J,\TT,S)}(V)$ \tabularnewline
                    &               &         &                  & \tabularnewline [-1em]
\hhline{|=|=|=|=|=|}
$\Conv$ & $\Mod{\RR}(\V)$ \newline{\footnotesize($\RR$-module objects\newline in $\V$)} & $\RR$   & \Mod{\RR}(\V)        & {\footnotesize Radon measures of compact support ($V \in \LCHaus$)}\tabularnewline \hline
$\Conv$ & $\Mod{\RR_+}(\V)$ & $\RR_+$ & \Mod{\RR_+}(\V)      & {\footnotesize Non-negative Radon measures of compact support ($V \in \LCHaus$)}\tabularnewline \hline
$\Conv$ & $\Cvx{\RR}(\V)$ \newline {\footnotesize($\RR$-convex spaces\newline in $\V$)} & $\RR_+$ & $\Mod{\RR_+}^*(\V)$ {\footnotesize(pointed $\RR_+$-modules in $\V$, \ref{par:th_pted_right_rmods})} & {\footnotesize Radon probability measures of compact support ($V \in \LCHaus$)}\tabularnewline \hline
$\Fro$, $\Diff$, {\footnotesize or} $\Cah$ & $\Mod{\RR}(\V)$ & $\RR$ & \Mod{\RR}(\V) & {\footnotesize Schwartz distributions of compact support \newline ($V$ smooth manifold)}\tabularnewline \hline
$\Set$ & $\SLat_{\wedge\top}$ \newline {\footnotesize(meet semilattices, \ref{par:slat_filt})} & 2 & $\SLat_{\wedge\top}$ & {\footnotesize Filters on the set $V$}\tabularnewline \hline
$\Set$ & $\SLat_\wedge$ \newline {\footnotesize (binary-meet semilattices, \ref{par:slat_filt})} & 2 & $\SLat_{\wedge\top\bot}$ \newline {\footnotesize(meet semilattices with bottom element)} & {\footnotesize Proper filters on the set $V$}\tabularnewline \hline
$\Set$ & $\Set$ & 2 & $\Bool$ \newline {\footnotesize(Boolean algebras)} & {\footnotesize Ultrafilters on the set $V$}\tabularnewline \hline
\end{tabular}
\end{center}

In addition to this basic collection of examples of the functional distribution monad, we will establish further examples within a forthcoming paper.  Therein we will show that by employing the theory of convex spaces and the unit interval $[0,1]$ as $S$, one can capture \textit{arbitrary} Radon probability measures, rather than just those of compact support.  Further, we will show that certain \textit{hyperspaces} of closed and compact subsets are captured by functional distribution monads.

Up to equivalence, a system of arities in $\V$ is given by a full subcategory of $\V$ that is closed under the monoidal product and contains the unit object \cite[3.8]{Lu:EnrAlgTh}.  Hence $\V$ itself is a system of arities, for which $\V$-ary monads are simply \textit{arbitrary $\V$-monads on $\V$} \cite[11.3(2), 11.10]{Lu:EnrAlgTh}.  For this system of arities $\V$, we call $\V$-ary commutants \textit{absolute commutants} and write them as $\TT^\perp$ rather than $\TT^\perp_\V$.

With this terminology, we show in \ref{thm:cdistn_mnd_dbl_cmtnt} that the functional distribution monad $\DD_{\scriptscriptstyle(\J,\TT,S)}$ associated to a functional-analytic context $(\V,\J,\TT,S)$ may be defined succinctly as \textit{the absolute commutant of the $\J$-ary commutant of $\TT$} (w.r.t. $S$), i.e.
\begin{equation}\label{eq:fdistn_mnd_dbl_cmtnt}\DD_{\scriptscriptstyle(\J,\TT,S)} = (\TTperpJ)^\perp\;.\end{equation}
Related to this, one of our axioms for a functional-analytic context asserts precisely that the $\J$-ary \textit{double commutant} of $\TT$ should be $\TT$ itself, i.e. that
\begin{equation}\label{eq:sat}(\TTperpJ)^\perp_{\kern-.1ex\scriptscriptstyle\J} \cong \TT,\end{equation}
so we say that the $\J$-ary monad $\TT = (T,\eta,\mu)$ is \textit{saturated} with respect to $S$.  As a consequence of \eqref{eq:sat} and \eqref{eq:fdistn_mnd_dbl_cmtnt}, the $\V$-endofunctor $D_{\scriptscriptstyle(\J,\TT,S)}$ agrees with $T$ when restricted to the full subcategory $\J \hookrightarrow \V$ \pref{thm:t_jary_restn_of_d}, and we say that the monad $\TT$ is the \textit{$\J$-ary restriction} of $\DD_{\scriptscriptstyle(\J,\TT,S)}$ \pref{def:jary_restn}.  This generalizes the classical fact that the free vector space on a finite set $J$ is equally the space of all Radon measures on the finite discrete space $J$, and also the fact that the free convex space on a finite set $J$ is the space of all Radon probability measures on $J$.

Contrastingly, the natural distribution monad \eqref{eq:nat_dist_mnd_for_t} determined by $\TT$ is simply the absolute commutant $\TT^\perp$ of $\TT$ with respect to $S$ \pref{thm:charn_abs_cmtnt}, whose $\J$-ary restriction does \textit{not} in general coincide with $\TT$; however, in the special case when the $\J$-ary monad $\TT$ is \textit{balanced} (w.r.t. $S$) we have $\TTperpJ \cong \TT$ and so this issue is (or rather was) hidden.

In \S\ref{sec:background} we begin with a review of some general background material on enriched category theory, order theory, convergence spaces, smooth spaces, and $C^\infty$-rings, as well as internal rings, rigs, and modules; we also introduce some notation for working with finite products \pref{par:notn_fp}.  In \S \ref{sec:background_enr_alg} we survey certain basic elements of the study of $\J$-theories \pref{sec:enr_alg_th_sys_ar} and discretely finitary theories \pref{sec:disc_fin_enr_alg_th}, including the notions of commutation, commutativity, and commutant for $\J$-theories and $\J$-ary monads \pref{sec:cmtn} as well as basics on $\J$-algebraic symmetric monoidal closed $\V$-categories \pref{sec:alg_smc_vcats}.  In \S \ref{sec:cdistn_mnds} we define the notion of functional distribution monad, and in \S \ref{sec:exa_cdistn_mnds} we develop various specific examples on the basis of the results proved later in the paper.  In \S\ref{sec:further_fun} we develop further fundamental aspects of enriched algebra required for the remainder of the paper, concerning modules for monoids in $\J$-algebraic symmetric monoidal closed $\V$-categories \pref{sec:mod_jalg_smcvcats}, coslices of $\V$-categories of $\T$-algebras \pref{sec:coslices}, the \textit{free} discretely finitary theory generated by an ordinary Lawvere theory \pref{sec:free_enr_th}, as well as enriched-categorical aspects of the study of modules over rigs and rings in cartesian closed categories \pref{sec:rmods}.  In \S \ref{sec:cmtnt_th_lrmods} we show that for a given rig object $R$ in $\V$, the theories of left $R$-modules and right $R$-modules, respectively, are commutants of one another with respect to $R$, thus generalizing to the enriched context a result of the author in the set-based context \cite[5.14]{Lu:CvxAffCmt}.  In \S \ref{sec:aff_cvx_sp} we define the \textit{affine core} of a discretely finitary theory (generalizing from the $\V = \Set$ case Lawvere's notion \cite[\S 3]{Law:ProbsAlgTh}), and we employ this to define the notion of \textit{left $R$-affine space} (in $\V$) for a rig $R$ in $\V$ and, in particular, the notion of \textit{left $R$-convex space} for a preordered ring object $R$ in $\V$.  In \S \ref{sec:th_aff_cvx_sp_as_cmtnts} we show that for any rig object $R$ in $\V$, the theory of left $R$-affine spaces in $\V$ is the commutant (w.r.t. $R$) of the theory of \textit{pointed right $R$-modules} in $\V$, thus generalizing to the enriched context a result of the author in the set-based context \cite[7.2]{Lu:CvxAffCmt}; in particular, this result applies to left $R$-convex spaces for a preordered ring $R$ in $\V$, equivalently left $R_+$-affine spaces.  In \S \ref{sec:cmtnt_of_th_cvx_sp} we establish conditions on $R$ and $\V$ under which the commutant of the theory of left $R$-convex spaces in $\V$ is the theory of pointed right $R_+$-modules; in particular, we build on a recent result of the author in the set-based context \cite[10.20]{Lu:CvxAffCmt} in order to prove that it is sufficient to require that $\V$ have a class of generators $V$ for which the preordered ring $\V(V,R)$ is a \textit{firmly archimedean algebra over the dyadic rationals}, assuming also that the inclusion $R_+ \hookrightarrow R$ is a strong monomorphism.  This result applies not only to various \textit{concrete} categories $\V$ over $\Set$ in which the real numbers $\RR$ underlie an ordered ring object, but also to the Cahiers topos $\V = \Cah$.  Therein, the line object $R$ is a preordered ring for which the hom $\V(V,R)$ is firmly archimedean whenever $V = K \times W$ where $K$ is a closed ball in some Euclidean $n$-space and $W$ is the spectrum of a Weil algebra \pref{thm:cah}.

Elements of the present work were announced in the author's recent conference talks \cite{Lu:CT2015,Lu:CT2016}.  This paper is the first part of a forthcoming series of papers that will study the abstract functional analysis intrinsic to a given functional-analytic context $(\V,\J,\TT,S)$ and its application to the study of the functional distribution monad, the axiomatics of related monads, and the associated notion of vector-valued integration, in analogy with the author's work in connection with the natural distribution monad in \cite{Lu:PhD}.

\begin{Acknowledgement}
The author thanks the anonymous referee for helpful suggestions and remarks.  The diagram \eqref{eq:cmtnt_diagr} was suggested by the referee, as were the key elements of the discussion of the codensity monad of a $\T$-algebra in \ref{par:cmtnts_mnds}.  Also, the current title of this paper was suggested by the referee; an earlier preprint of the paper had been distributed under the title \textit{Measure and distribution monads in categorical functional analysis, I: The functional distribution monad}.
\end{Acknowledgement}

\section{General background and notation}\label{sec:background}

\begin{ParSub}[\textbf{Enriched category theory}]\label{par:enr_cat_th}
Unless otherwise specified, $\V$ will denote a given symmetric monoidal closed category with equalizers.  Throughout, $\V$ is tacitly assumed locally small; further assumptions on $\V$ will be imposed later.  We shall employ the theory of categories enriched in $\V$, as documented in \cite{EiKe,Ke:Ba,Dub}.  Given a $\V$-category $\C$ we denote by $\C_0$ the ordinary category underlying $\C$.  A morphism $f:C \rightarrow D$ in $\C$ (i.e., in $\C_0$) is equally a morphism $I \rightarrow \C(C,D)$ in $\V$, yet we shall denote the latter morphism in $\V$ by $[f]$.  There is a $\V$-category $\uV$ whose hom-objects are the internal homs $\uV(V,W)$ of the closed category $\V$ and whose underlying ordinary category $\uV_0$ may be identified with $\V$ itself.  We will distinguish terminologically between $\V$-categories (resp. $\V$-functors) and (ordinary) categories (resp. functors), but when we apply terms like \textit{limit}, \textit{colimit}, or \textit{fully faithful} to given $\V$-categories and $\V$-functors, we mean to apply the relevant $\V$-enriched notion.  In particular, by a \textit{(co)limit} in a given $\V$-category, we mean a $\V$-\textit{enriched weighted (co)limit} (called an \textit{indexed (co)limit} in \cite{Ke:Ba}).  We shall say that a weight $W:\B \rightarrow \uV$ is \textit{objectwise-countable} if the class of objects of $\B$ is a countable set; colimits for objectwise-countable weights are called \textbf{objectwise-countable colimits}.  Given a $\V$-category $\C$, a \textit{conical limit} of an ordinary functor $D:\K \rightarrow \C_0$ is given by an ordinary limit of $D$ that is preserved by each functor $\C(C,-):\C_0 \rightarrow \V$ with $C \in \ob\C$.  We can form the free $\V$-category $\K_\V$ on $\K$ as in \cite[2.5]{Ke:Ba} provided that the copower $\K(K,L) \cdot I$ in $\V$ exists for all $K,L \in \ob\K$; in this case, a conical limit of $D$ can be expressed as a certain $\V$-enriched weighted limit as in \cite[2.5]{Ke:Ba}.  A \textit{product} (resp. \textit{equalizer}, \textit{power}, ...) in a given $\V$-category $\C$ is, by definition, a conical product in $\C$.
\end{ParSub}

\begin{ParSub}[\textbf{Semilattices and filters}]\label{par:slat_filt}
By definition, a \textbf{meet semilattice} is a partially ordered set with finite meets (equivalently, binary meets and a top element).  In many texts this term is used to refer to partially ordered sets that are merely assumed to have binary meets, but these we shall instead call \textbf{binary-meet semilattices} herein.  A \textbf{homomorphism of meet (resp. binary-meet) semilattices} is a mapping between meet semilattices that preserves finite meets (resp. binary meets); we denote by $\SLat_{\wedge\top}$ and $\SLat_\wedge$ the categories of meet semilattices (resp. binary-meet semilattices) and their homomorphisms.  These categories are isomorphic to the categories of \textbf{join semilattices} and \textbf{binary-join semilattices}, respectively, where these notions are defined dually and the isomorphism is given on objects by taking the opposite order.

A subset $S$ of a meet semilattice $L$ is said to be a \textbf{filter in} $L$ if its characteristic function $\chi_S:L \rightarrow 2$ is a homomorphism of meet semilattices, where $2 = \{0,1\}$ denotes the two-element meet semilattice with top element $1$, whose preimage under $\chi_S$ is $S$.  In the case where the meet semilattice $L$ also has a bottom element $\bot$, a filter $S$ in $L$ is said to be \textbf{proper} if $\bot \notin S$, equivalently, if $\chi_S$ preserves the bottom element.

If $L$ is the meet semilattice underlying a Boolean algebra $L$, then a subset $S$ of $L$ is said to be an \textbf{ultrafilter in} $L$ if its characteristic function $\chi_S:L \rightarrow 2$ is a homomorphism of Boolean algebras.  Given a set $X$, a \textbf{filter} (resp. \textbf{proper filter}, \textbf{ultrafilter}) \textbf{on} $X$ is, by definition, a filter (resp. proper filter, ultrafilter) in the powerset $\sP(X)$ of $X$.  The Boolean algebra $\sP(X)$ may be identified with the $X$-fold power $2^X$ of the Boolean algebra $2$, so for each element $x$ of $X$, the projection map $\pi_x:2^X \rightarrow 2$ is a homomorphism of Boolean algebras and so corresponds to an ultrafilter on $X$, called the \textbf{principal ultrafilter} for $x$.  When $X$ is finite, every ultrafilter on $X$ is principal, so any homomorphism of Boolean algebras $2^X \rightarrow 2$ is of the form $\pi_x$ for a unique $x \in X$.
\end{ParSub}

\begin{ParSub}[\textbf{Convergence spaces and Radon measures}]\label{par:conv_sp_rad_meas}
A \textbf{convergence space} (see \cite{BeBu}) is a set $X$ equipped with an assignment to each point $x \in X$ a set $\K_x$ of proper filters on $X$, which are said to \textit{converge} to $x$, such that (1) the principal ultrafilter for $x$ converges to $x$, (2) if $\X,\Y \in \K_x$ then $\X \cap \Y \in \K_x$, and (3) if $\X$ and $\Y$ are proper filters such that $\Y \supseteq \X \in \K_x$, then $\Y \in \K_x$.  Given convergence spaces $X$ and $Y$, a mapping $f:X \rightarrow Y$ is \textbf{continuous} if whenever a filter $\X$ on $X$ converges to a point $x$ of $X$, the \textit{image filter} $f[\X] = \{B \subseteq Y \mid f^{-1}(B) \in \X \}$ converges to $f(x)$ in $Y$.  Convergence spaces and continuous maps form a complete and cocomplete category $\Conv$ into which the familiar category $\Top$ of topological spaces embeds as a full, reflective subcategory \cite[1.1.4(ii), 1.3.9, 1.8]{BeBu}; the objects of this reflective subcategory are called \textit{topological} convergence spaces.  Further, $\V = \Conv$ is cartesian closed, where the internal hom $\uV(X,Y)$ is the set of all continuous maps from $X$ to $Y$, equipped with the relation of \textit{continuous convergence}; explicitly, a filter $\F$ on $\V(X,Y)$ converges to $g \in \V(X,Y)$ if and only if for every $x \in X$ and every filter $\X$ converging to $x$ in $X$, the associated filter $\F[\X]$ converges to $g(x)$ in $Y$.  Here $\F[\X]$ denotes the filter on $Y$ with a basis consisting of the sets $F(A) = \{f(a) \mid f \in F, a \in A\}$ with $F \in \F$ and $A \in \X$.  Now supposing that $X$ is a locally compact Hausdorff topological space and $Y$ is a regular topological space, the convergence space $\uV(X,Y)$ is topological and carries the familiar \textit{compact-open topology} \cite[1.5.16]{BeBu}, also called the \textit{topology of compact convergence} \cite[\S 3, No. 4, D\'ef. 1]{Bou:EspFonc}.  Hence, in the special case where $Y$ is the topological ring $R = \RR$ or $\CC$, continuous $R$-linear maps $\mu:\uV(X,R) \rightarrow R$ are in bijective correspondence with $R$-valued Radon measures of compact support on $X$ \cite[Ch. IV \S 4, No. 8, Prop. 14]{Bou}.
\end{ParSub}

\begin{ParSub}[\textbf{Fr\"olicher spaces and diffeological spaces}]\label{par:fro_diff}
We shall denote by $\Mf$ the category of second-countable Hausdorff smooth manifolds, which are necessarily paracompact and will be called simply \textit{smooth manifolds}; the morphisms in $\Mf$ are arbitrary smooth maps.  We shall make use of certain cartesian closed categories $\V$ into which $\Mf$ embeds as a full subcategory, as follows.  

Firstly, a \textbf{Fr\"olicher smooth space} (or \textbf{Fr\"olicher space}) \cite{Fro:SmthStr,Fro:CccAnSmthMaps,FroKr} is a set $X$ equipped with a specified set of mappings $\C_X \subseteq \Set(\RR,X)$ called \textit{smooth curves} and a set of mappings $\F_X \subseteq \Set(X,\RR)$ called \textit{smooth functions}, such that $\C_X = \{c \in \Set(\RR,X) \mid \forall f \in \F_X \,:\, f \cdot c \in \Mf(\RR,\RR) \}$ and $\F_X = \{f \in \Set(X,\RR) \mid \forall c \in \C_X \,:\, f \cdot c \in \Mf(\RR,\RR)\}$.  Fr\"olicher spaces form a cartesian closed category $\Fro$ in which the morphisms are mappings that preserve smooth curves \cite[\S 1, 2]{Fro:SmthStr}, and there is a full embedding of $\Mf$ into $\Fro$ \cite[\S 3]{Fro:SmthStr}.

Let $\OCart$ denote the full subcategory of $\Mf$ consisting of all open subsets of the spaces $\RR^n$ $(n \gt 0)$.  A \textbf{diffeological space} (or \textit{Souriau space}) is a set $X$ equipped with an assignment to each object $U$ of $\OCart$ a set $\sP_X(U)$ of functions from $U$ to $X$, called \textit{smooth plots}, subject to certain axioms; see \cite{BaHo}.  Diffeological spaces are the objects of a category $\Diff$ in which a morphism is simply a mapping that preserves smooth plots.  $\Diff$ is equivalent to the category of \textit{concrete sheaves} on $\OCart$ with respect to the open cover topology \cite{BaHo}.  Consequently, $\Diff$ is a complete and cocomplete \textit{quasitopos} \cite{BaHo} and, in particular, is cartesian closed, and there is a full, dense embedding $\OCart \hookrightarrow \Diff$.  Every Fr\"olicher space $X$ determines a diffeological space with the same underlying set, and with smooth plots $\sP_X(U) = \Fro(U,X)$ \cite[\S 5]{Sta}.  This yields a full reflective embedding $\Fro \hookrightarrow \Diff$ \cite{Sta}.  Since the dense embedding $\OCart \hookrightarrow \Diff$ factors through $\Fro \hookrightarrow \Diff$, it follows that the embedding $\Fro \hookrightarrow \Diff$ is dense \cite[Thm. 5.13]{Ke:Ba}.  Since this embedding $\Fro \hookrightarrow \Diff$ is dense and preserves finite products, it readily follows that it preserves exponentials.
\end{ParSub}

\begin{ParSub}[\textbf{$C^\infty$-rings and the Cahiers topos}]\label{par:cinfty_rings}
The category of smooth manifolds $\Mf$ \pref{par:fro_diff} embeds as a full subcategory of the opposite of the category $\CinftyRing_{\textnormal{fg}}$ of \textbf{finitely generated $C^\infty$-rings}; see \cite{Kock:SDG,MoeRey}.  By a \textbf{product-closed $C^\infty$-site} we shall mean a full subcategory $\A$ of $\CinftyRing^\op_{\textnormal{fg}}$ such that (i) $\A$ is equipped with a subcanonical Grothendieck topology, (ii) $\A$ contains $C^\infty(M)$ for each manifold $M$, and (iii) $\A$ is closed under finite products in $\CinftyRing^\op_{\textnormal{fg}}$.  It follows that $\Mf$ embeds as a full subcategory of the topos $\Shv(\A)$ of sheaves on $\A$.  Certain toposes employed in synthetic differential geometry are of this form.  In particular, the \textbf{Cahiers topos} of Dubuc is obtained by taking $\A = \C^\op_\text{C}$ where $\C_\text{C} \hookrightarrow \CinftyRing_{\textnormal{fg}}$ consists of all the $C^\infty$-rings isomorphic to $C^\infty(M) \otimes_\infty W$ for some smooth manifold $M$ and some Weil algebra $W$ \cite[I.16]{Kock:SDG} (over $\RR$), where $\otimes_\infty$ denotes the coproduct in the category $\CinftyRing$ of $C^\infty$-rings; see \cite[III.9, Example (1)]{Kock:SDG}.  Indeed, $\C^\op_\text{C}$ satisfies (iii), since for any two objects $A = C^\infty(M) \otimes_\infty W$ and $A' = C^\infty(M') \otimes_\infty W'$ of $\C_\text{C}$, we deduce by \cite[p. 22]{MoeRey} that their coproduct in $\CinftyRing_{\textnormal{fg}}$ is $A \otimes_\infty A'$, but by \cite[I.2.5]{MoeRey} this is isomorphic to $C^\infty(M \times M') \otimes_\infty W \otimes_\infty W'$ and hence lies in $\C_\text{C}$, by \cite[I.3.21]{MoeRey}.
\end{ParSub}

\begin{ParSub}[\textbf{Notation for categories with finite products}]\label{par:notn_fp}
In working with rings, modules, and related structures in a category $\C$ with finite products, it can be convenient to use `elementwise' notation for some calculations involving the familiar algebraic operations.  For our purposes, a few notational conventions for working with finite products will suffice in this regard.

Firstly, given a finite product $X_1 \times X_2 \times ... \times X_n$ in $\C$, we shall not always use a standard notation such as $\pi_i:X_1 \times ... \times X_n \rightarrow X_i$ to denote each product projection, but rather we shall allow ourselves to use any given (distinct) symbols $x_1,x_2...,x_n$ to denote these projections $\pi_i$, and we shall write
$$
\begin{array}{lcl}
(x_1,x_2,...,x_n):X_1 \times X_2 \times ... \times X_n & \text{to mean that} & \textit{$x_i$ denotes the projection $\pi_i$}\\
                                                       & & \textit{for each $i \in \{1,2,...,n\}$.}\\
\end{array}
$$
In the case where $n = 1$, we omit the parentheses, writing 
$$
\begin{array}{lcl}
x:X & \text{to mean that} & \textit{$x$ denotes the identity morphism $1_X$ on an object $X$,}\\
\end{array}
$$
so that then $x$ plays the role of a \textit{generic element} of $X$.

If we are given a morphism $t:C \rightarrow \prod_{i \in J}X_i$ in $\C$ whose codomain is a finite product, then for each element $i$ of the indexing set $J$ we shall write $t_i:C \rightarrow X_i$ to denote the $i$-th factor of $t$.  When $J = \{1,2,...,n\}$ we write $(t_1,...,t_n):C \rightarrow X_1 \times ... \times X_n$ to denote the induced morphism $t$.  Given also a morphism $f:X_1 
\times ... \times X_n \rightarrow Y$, we shall allow ourselves to write the composite $f \cdot (t_1,...,t_n):C \rightarrow Y$ as $f(t_1,...,t_n)$.  We shall sometimes use this notation even when $n = 1$, so that if $t:C \rightarrow X$ and $f:X \rightarrow Y$ then $f(t) = f \cdot t:C \rightarrow Y$.  In the case with $n = 0$, where $f:1 \rightarrow Y$ is a \textit{constant}, we shall write the \textit{constant morphism} $C \rightarrow 1 \xrightarrow{f} Y$ as just $f$.

Let $M$ be a commutative monoid in $\C$, with monoid operations $+,0$, and suppose that $\C$ is locally small.  For each object $C$ of $\C$, the functor $\C(C,-):\C \rightarrow \Set$ preserves products and therefore sends $M$ to a commutative monoid $\C(C,M)$ (in $\Set$).  We shall write the induced addition operation on $\C(C,M)$ in infix notation as $+$ and write the zero element of $\C(C,M)$ as $0$.  Given a morphism $m = (m_1,...,m_n):C \rightarrow M^n$ in $\C$, we shall use the usual notation $\sum_{i = 1}^n m_i$ for the sum of the elements $m_1,...,m_n$ of $\C(C,M)$.

If $\V$ is a cartesian closed category, then for all morphisms of the form $f:V \rightarrow \uV(X,Y)$ and $x:V \rightarrow X$ in $\V$, we shall write $f(x)$ to denote the composite
$$f(x) = \left(V \xrightarrow{(f,x)} \uV(X,Y) \times X \xrightarrow{\Ev} Y\right)\;,$$
where $\Ev$ denotes the evaluation morphism.

As an example, we now state definitions of some basic algebraic notions that we shall need in the sequel.
\end{ParSub}

\begin{DefSub}\label{def:rig_ring_module_bim_cmons}
Let $\C$ be a category with finite products.  Given a commutative monoid $M$ in $\C$, let us denote its underlying object in $\C$ by $\ca{M}$, or even just $M$, and write its binary and unary operations as $+$ and $0$, respectively.  Given commutative monoids $M,N,P$ in $\C$, a morphism $f:\ca{M} \times \ca{N} \rightarrow \ca{P}$ in $\C$ is said to be a \textbf{bimorphism of commutative monoids} from $M,N$ to $P$ if the following equations hold, in the notation of \ref{par:notn_fp}:
\begin{enumerate}
\item[] $\begin{array}{l}
f(0,n) = 0:N \rightarrow P \;\;\text{and}\;\; f(m,0) = 0:M \rightarrow P,\;\;\text{where $n:N$ and $m:M$,}\\
\end{array}$
\item[] $\begin{array}{lcl}
f(m_1 + m_2,n) = f(m_1,n) + f(m_2,n) & : & M \times M \times N \rightarrow P,\\
                                                      &   & \text{where $(m_1,m_2,n):M \times M \times N$,}\\
\end{array}$
\item[] $\begin{array}{lcl}
f(m,n_1 + n_2) = f(m,n_1) + f(m,n_2) & : & M \times N \times N \rightarrow P,\\
                                                      &   & \text{where $(m,n_1,n_2):M \times N \times N$.}\\
\end{array}$
\end{enumerate}

A \textbf{rig} or \textbf{semiring} in $\C$ is a commutative monoid $R = (\ca{R},+,0)$ in $\C$ equipped with morphisms $1:1 \rightarrow \ca{R}$ and $\bullet:\ca{R} \times \ca{R} \rightarrow \ca{R}$ in $\C$ such that $\bullet$ is a bimorphism of commutative monoids from $R,R$ to $R$ and $(\ca{R},\bullet,1)$ is a monoid in $\C$.  We typically denote both the rig and its carrier by $R$.  A rig $R$ in $\C$ is said to be a \textbf{ring} in $\C$ if the monoid $(\ca{R},+,0)$ is an abelian group in $\C$.  $R$ is \textbf{commutative} if the monoid $(\ca{R},\bullet,1)$ is commutative.  Given a rig $R$ in $\C$, a \textbf{left} \textbf{$R$-module} (in $\C$) is a commutative monoid $M$ in $\C$ equipped with an associative, unital action $*:\ca{R} \times \ca{M} \rightarrow \ca{M}$ in $\C$ of the monoid $(\ca{R},\bullet,1)$, such that $*$ is a bimorphism of commutative monoids from $(\ca{R},+,0), M$ to $M$.
\end{DefSub}

\begin{ParSub}[\textbf{Rigs and modules of generalized elements}]\label{par:rigs_mods_gen_elts}
Let $\C$ be a locally small category with finite products, and let $C$ be an object of $\C$.  Given a rig $R$ in $\C$ \pref{def:rig_ring_module_bim_cmons}, the product-preserving functor $\C(C,-):\C \rightarrow \Set$ sends $R$ to a rig $\C(C,R)$ (in $\Set$).  Extending the notation of \ref{par:notn_fp}, we shall use the usual notation for the rig operations on $\C(C,R)$, writing the multiplication operation induced by $\bullet:R \times R \rightarrow R$ as juxtaposition, so that if $r,s \in \C(C,R)$ then $rs$ denotes the product in $\C(C,R)$.  Given a left $R$-module $M$ in $\C$ \pref{def:rig_ring_module_bim_cmons}, the product-preserving functor $\C(C,-)$ sends $M$ to a left $\C(C,R)$-module $\C(C,M)$ (in $\Set$).  Extending the notation of \ref{par:notn_fp}, we shall use the usual notation for the left $\C(C,R)$-module structure on $\C(C,M)$, so that if $r \in \C(C,R)$ and $m \in \C(C,M)$ then the left action determines an associated element $rm$ of $\C(C,M)$.
\end{ParSub}

\section{Background on enriched algebra}\label{sec:background_enr_alg}

\subsection{Enriched algebraic theories for a system of arities}\label{sec:enr_alg_th_sys_ar}

\begin{ParSubSub}\label{par:jth}
By definition, a \textbf{system of arities} in $\V$ is a symmetric strong monoidal $\V$-functor $j:\J \rightarrow \uV$ that is fully faithful \cite[3.1]{Lu:EnrAlgTh}.  For example, the full subcategory $\FinCard \hookrightarrow \Set$ consisting of the finite cardinals is a system of arities; see \cite[\S3]{Lu:EnrAlgTh} for various further examples.  A full sub-$\V$-category $\J \hookrightarrow \uV$ is a system of arities provided that it contains the unit object $I$ and is closed under $\otimes$; every system of arities is equivalent (in a suitable sense) to one of the latter form \cite[3.8]{Lu:EnrAlgTh}, and so for many purposes we can assume that given systems of arities are of this form.  Given a system of arities $j:\J \hookrightarrow \uV$, a \textbf{$\J$-theory} \cite[4.1]{Lu:EnrAlgTh} is a $\V$-category $\T$ equipped with an identity-on-objects $\V$-functor $\tau:\J^\op \rightarrow \T$ that preserves $\J$-cotensors, i.e., cotensors by objects $J$ of $\J$ (or rather, their associated objects $j(J)$ of $\V$).  Equivalently, a $\J$-theory is a $\V$-category $\T$ whose objects are precisely those of $\J$, such that for each object $J$ of $\J$ there is a specified morphism $\gamma_J:J \rightarrow \T(J,I)$ in $\V$ that exhibits $J$ as a cotensor $[J,I]$ of $I$ by $J$ in $\T$, with the further stipulation that $\gamma_I = [1_I]:I \rightarrow \T(I,I)$ \cite[5.8]{Lu:EnrAlgTh}.  The associated identity-on-objects $\V$-functor $\tau:\J^\op \rightarrow \T$ is precisely the $\V$-functor $[-,I]$ that supplies the designated $\J$-cotensors of $I$ \cite[5.8]{Lu:EnrAlgTh}.  Collectively, $\J$-theories are the objects of a category $\ThJ$ in which the morphisms are $\V$-functors $A:\T \rightarrow \U$ that commute with the associated $\V$-functors $\J^\op \rightarrow \T$ and $\J^\op \rightarrow \U$.  Note that $\J^\op$ is therefore an initial object of $\ThJ$.  A \textbf{subtheory} of a $\J$-theory $\U$ is a $\J$-theory $\T$ equipped with a morphism of theories $\T \hookrightarrow \U$ that is faithful, as a $\V$-functor, meaning that its structure morphisms are monomorphisms in $\V$.   
\end{ParSubSub}

\begin{ParSubSub}[\textbf{Correspondence between $\J$-theories and $\J$-ary monads}]\label{par:jary_mnds}
A system of arities $j:\J \hookrightarrow \uV$ is said to be \textbf{eleutheric} if every $\V$-functor $P:\J \rightarrow \uV$ has a left Kan extension along $j$ and this left Kan extension is preserved by each $\V$-functor $\uV(J,-):\uV \rightarrow \uV$ with $J \in \ob\J$ \cite[7.1, 7.3]{Lu:EnrAlgTh}.  Various examples and equivalent characterizations of eleutheric systems of arities are provided in \cite[\S 7]{Lu:EnrAlgTh}.  Throughout the remainder of \S \ref{sec:enr_alg_th_sys_ar}, we shall assume that $j:\J \hookrightarrow \uV$ is a given eleutheric system of arities.

By definition, a \textbf{$\J$-ary monad} is a $\V$-monad $\TT = (T,\eta,\mu)$ on $\uV$ such that $T$ preserves left Kan extensions along $j:\J \hookrightarrow \uV$ \cite[11.7, 11.1]{Lu:EnrAlgTh}.  Note that a $\V$-functor $T:\uV \rightarrow \uV$ preserves left Kan extensions along $j$ if and only if $T$ preserves weighted colimits with weights of the form
\begin{equation}\label{eq:wts_for_lanj}\uV(j-,V):\J^\op \rightarrow \uV\;\;\;\;\;\;(V \in \ob\V).\end{equation}
A weight $W:\B^\op \rightarrow \uV$ is said to be \textbf{$\J$-flat} if every $W$-weighted colimit in $\uV$ is preserved by each of the $\V$-functors $\uV(J,-):\uV \rightarrow \uV$ with $J \in \ob\J$ \cite[6.2]{Lu:EnrAlgTh}.  For example, each of the weights \eqref{eq:wts_for_lanj} is $\J$-flat since $j$ is eleutheric.  A $\V$-monad $(T,\eta,\mu)$ on $\uV$ is a $\J$-ary monad if and only if $T$ \textbf{conditionally preserves $\J$-flat colimits}, meaning that for any colimit $W \star D$ in $\uV$ with a $\J$-flat weight $W$, if the colimit $W \star TD$ exists in $\uV$ then the colimit $W \star D$ is preserved by $T$ \cite[12.3]{Lu:EnrAlgTh}.  Collectively, $\J$-ary monads are the objects of a category $\MndJ(\uV)$, with the usual morphisms of $\V$-monads.

There is an equivalence 
$$\ThJ \simeq \MndJ(\uV)$$
between the category of $\J$-theories and the category of $\J$-ary monads \cite[11.8]{Lu:EnrAlgTh}.  Given a $\J$-theory $\tau:\J^\op \rightarrow \T$, the $\V$-endofunctor $T:\uV \rightarrow \uV$ underlying the corresponding $\J$-ary monad $\TT$ is a left Kan extension of $\T(\tau-,I):\J \rightarrow \T$ along $j$.  Conversely, given a $\J$-ary monad $\TT$ with Kleisli $\V$-category $\uV_\TT$, the corresponding $\J$-theory is the full sub-$\V$-category of $\uV_\TT^\op$ on the objects of $\J$ \cite[11.11]{Lu:EnrAlgTh}.

There is an eleutheric system of arities with $\J = \uV$ and $j = 1_{\uV}$ \cite[7.5(3)]{Lu:EnrAlgTh}, for which a $\uV$-ary monad is simply an \textit{arbitrary} $\V$-monad on $\uV$ \cite[11.3(2), 11.10]{Lu:EnrAlgTh}, so that we have an equivalence $\Th_{\uV} \simeq \Mnd_{\VCAT}(\uV)$ between the category of $\uV$-theories (originally studied by Dubuc \cite{Dub:EnrStrSem}) and the category of $\V$-monads on $\uV$ \cite[11.10]{Lu:EnrAlgTh}.
\end{ParSubSub}

\begin{ParSubSub}[\textbf{Algebras for a $\J$-theory}]\label{par:alg_jth}
Given a $\J$-theory $\T$ and a $\V$-category $\C$, a $\T$-\textbf{algebra} in $\C$ is a $\V$-functor $A:\T \rightarrow \C$ that preserves $\J$-cotensors; we call the object $\ca{A} = A(I)$ the \textbf{carrier} of $A$.  Now assuming that $\C$ has designated $\J$-cotensors, a \textbf{normal $\T$-algebra} in $\C$ is a $\V$-functor $A:\T \rightarrow \C$ that \textit{strictly} preserves the designated $\J$-cotensors $[J,I] = J$ of $I$ in $\T$ \pref{par:jth}, i.e. sends them to the designated $\J$-cotensors $[J,\ca{A}]$ of $\ca{A} = A(I)$ in $\C$; it then follows that $A$ preserves all $\J$-cotensors and hence is a $\T$-algebra \cite[5.10]{Lu:EnrAlgTh}.  A morphism of $\J$-theories $\T \rightarrow \U$ is equivalently defined as a normal $\T$-algebra in $\U$ with carrier $I$ \cite[5.16]{Lu:EnrAlgTh}.  Note that a normal $\T$-algebra $A:\T \rightarrow \C$ is uniquely determined by its carrier $\ca{A}$ together with its components $A_{JI}:\T(J,I) \rightarrow \uV([J,\ca{A}],\ca{A})$ with $J \in \ob\J$ \cite[3.12]{Lu:Cmtnts}, and it is sometimes convenient to construe the transposes $A_J:\T(J,I) \otimes [J,\ca{A}] \rightarrow \ca{A}$ of these morphisms as constituting the $\T$-algebra structure on $\ca{A}$.  Related to this, note also that a morphism of theories $A$ is an isomorphism iff its components $A_{JI}$ $(J \in \ob\J)$ are isomorphisms.
\end{ParSubSub}

\begin{ParSubSub}[\textbf{The $\V$-category of $\T$-algebras}]\label{par:vcat_talgs}
Given $\T$-algebras $A,B:\T \rightarrow \C$, we call $\V$-natural transformations $A \rightarrow B$ $\T$-\textbf{homomorphisms}.  If the object of $\V$-natural transformations $[\T,\C](A,B)$ exists for all $\T$-algebras $A$ and $B$ in $\C$, then we obtain a $\V$-category $\Alg{\T}_\C$ whose objects are $\T$-algebras in $\C$.  Necessary and sufficient conditions for the existence of the $\V$-category $\Alg{\T}_\C$ are given in \cite[4.9, 4.13]{Lu:Cmtnts}, and in particular, $\Alg{\T}_\C$ exists for all $\T$ and $\C$ as soon as $\V$ has intersections of $(\ob\J)$-indexed families of strong subobjects\footnote{i.e., intersections of $(\ob\J)$-indexed families of strong monomorphisms with the same codomain}.  Further, for $\C = \uV$ the $\V$-category $\Alg{\T}_{\uV}$ \textit{always} exists \cite[8.9]{Lu:EnrAlgTh}, without assuming the existence of any limits in $\V$ beyond equalizers.  We often call $\T$-algebras in $\uV$ simply $\T$-\textbf{algebras}, and we write simply $\Alg{\T}$ for $\Alg{\T}_{\uV}$.

The notion of normal $\T$-algebra is `equivalent' to that of $\T$-algebra, to the extent that $\Alg{\T}_\C$ (when it exists) is equivalent to its full sub-$\V$-category 
\begin{equation}\label{eq:equiv_talgs_ntalgs}\Alg{\T}_\C^! \overset{\sim}{\lhook\joinrel\relbar\joinrel\rightarrow} \Alg{\T}_\C\end{equation}
consisting of the normal $\T$-algebras \cite[5.14]{Lu:EnrAlgTh}, provided that $\C$ has designated $\J$-cotensors; in particular, every $\T$-algebra $A$ in $\C$ is isomorphic to an associated normal $\T$-algebra, which has the same carrier and is called the \textbf{normalization} of $A$.  As we shall do throughout the sequel, we assume that the designated cotensor $[I,C]$ of an object $C$ of $\C$ is simply $C$, with structural morphism $[1_C]:I \rightarrow \C(C,C)$.

Writing $\ca{\text{$-$}}:\Alg{\T}_\C \rightarrow \C$ for the $\V$-functor given by evaluation at $I$, sending a $\T$-algebra to its carrier, it is important to note that $\ca{\text{$-$}}$ is faithful, meaning that its structure morphisms $\Alg{\T}_\C(A,B) \rightarrow \C(\ca{A},\ca{B})$ are monomorphisms; indeed, they are in fact strong monomorphisms \cite[4.8]{Lu:Cmtnts}.  In particular, a $\T$-homomorphism from $A$ to $B$ is uniquely determined by its component at $I$, so we call a morphism $f:\ca{A} \rightarrow \ca{B}$ in $\C$ a $\T$-\textbf{homomorphism} if it lies in the image of the injective map $(\Alg{\T}_\C)_0(A,B) \hookrightarrow \C_0(\ca{A},\ca{B})$.
\end{ParSubSub}

\begin{ParSubSub}[\textbf{$\J$-algebraic $\V$-categories over $\V$}]\label{par:ntalgs_emalgs}
Given a $\J$-theory $\T$, the $\V$-category $\Alg{\T}$ of $\T$-algebras in $\uV$ is equivalent to the $\V$-category $\Alg{\TT}$ of $\TT$-algebras for the associated $\J$-ary monad $\TT$ \cite[11.14]{Lu:EnrAlgTh}.  Moreover, the $\V$-category $\Alg{\T}^!$ of normal $\T$-algebras in $\uV$ is \textit{isomorphic} to $\Alg{\TT}$ \cite[11.14]{Lu:EnrAlgTh}, and in the sequel we shall identify these $\V$-categories.  We say that a $\V$-functor $G:\A \rightarrow \uV$ is \textbf{$\J$-algebraic} (resp. \textbf{strictly} $\J$-\textbf{algebraic}) if there is a $\J$-theory $\T$ and an equivalence (resp. isomorphism) $\A \simeq \Alg{\T}^!$ that commutes, up to isomorphism, with the associated $\V$-functors to $\uV$.  By \cite[12.2, 12.3]{Lu:EnrAlgTh}, $G$ is $\J$-algebraic if and only if $G$ is $\V$-monadic and the induced $\V$-monad is $\J$-ary.  Similarly, $G$ is strictly $\J$-algebraic if and only if $G$ is strictly $\V$-monadic\footnote{i.e., has a left adjoint and satisfies the second pair of equivalent conditions in \cite[II.2.1]{Dub}} and the induced $\V$-monad is $\J$-ary.  Indeed, $\Alg{\T}^! \rightarrow \uV$ is strictly $\V$-monadic \cite[8.9]{Lu:EnrAlgTh} for a $\J$-ary monad, so the same is true for any strictly $\J$-algebraic $\V$-functor $G:\A \rightarrow \uV$.  Conversely, for a $\J$-ary monad $\TT$ on $\uV$, the Eilenberg-Moore forgetful $\V$-functor $\Alg{\TT} \rightarrow \uV$ is strictly $\J$-algebraic, by \cite[11.14]{Lu:EnrAlgTh}, so the same is true of any strictly $\V$-monadic $\V$-functor that induces $\TT$.
\end{ParSubSub}

\begin{ParSubSub}[\textbf{Pointwise limits of algebras}]\label{par:lims_algs}
Let $\T$ be a $\J$-theory, and let $\C$ be a $\V$-category for which $\Alg{\T}_\C$ exists.  Given any $\V$-functor $W:\B \rightarrow \uV$, if $\C$ has $W$-weighted limits then $W$-limits in $\Alg{\T}_\C$ can be formed pointwise.  Indeed, the pointwise $W$-weighted limit of a $\V$-functor $D:\B \rightarrow \Alg{\T}_\C$ is a $\T$-algebra in $\C$ since $W$-limits commute with $\J$-cotensors in $\C$.  Consequently, when $\C$ has designated $\J$-cotensors, $W$-limits in $\Alg{\T}^!_\C$ are obtained as the normalizations of these pointwise $W$-limits.

In particular, note that if $A:\T \rightarrow \uV$ is a $\T$-algebra and $V$ is an object of $\V$, then $\uV(V,A-):\T \rightarrow \uV$ is a cotensor $[V,A]$ of $A$ by $V$ in $\Alg{\T}$ and has carrier $\uV(V,\ca{A})$.
\end{ParSubSub}

\begin{ParSubSub}[\textbf{The full theory of an object}]\label{par:full_th}
Given an object $C$ of a $\V$-category $\C$ with designated $\J$-cotensors $[J,C]$ $(J \in \ob\J)$, we can form a $\J$-theory $\C_C$ called the \textbf{full $\J$-theory of $C$ in $\C$}, with
$$\C_C(J,K) = \C([J,C],[K,C])\;\;\;\;\;\;\;\;(J,K \in \ob\J = \ob\C_C)$$
and with composition and identities as in $\C$ \cite[3.16]{Lu:Cmtnts}.  The following observations concerning $\C_C$ are drawn from \cite[3.16]{Lu:Cmtnts} and follow readily from the definitions.  There is a fully faithful $\V$-functor $i:\C_C \rightarrow \C$ given on objects by $J \mapsto [J,C]$, and $i$ is a $\C_C$-algebra in $\C$ with carrier $C$.  Any $\T$-algebra $A:\T \rightarrow \C$ equips its carrier $\ca{A}$ with a designated choice of $\J$-cotensors $[J,\ca{A}] = A(J)$, and $A$ factors through $i:\C_{|A|} \rightarrow \C$ by way of a unique morphism of $\J$-theories $A:\T \rightarrow \C_{|A|}$.  In particular, if $\C$ has designated $\J$-cotensors of each of its objects $C$, then morphisms of $\J$-theories $\T \rightarrow \C_C$ are in bijective correspondence with normal $\T$-algebras in $\C$ with carrier $C$.
\end{ParSubSub}

\subsection{Discretely finitary enriched algebraic theories}\label{sec:disc_fin_enr_alg_th}

For the present section, let $\V$ be a countably cocomplete cartesian closed category with equalizers.  It follows that $\uV$ has objectwise-countable colimits \pref{par:enr_cat_th}.  

\begin{ParSubSub}[\textbf{The system of finite discrete arities}]\label{par:efth}
Choosing a designated $n$-th copower $n \cdot 1$ of the terminal object $1$ in $\V$ for each natural number $n \in \NN$ (including $0$), let us define a $\V$-category $\DFin_\V$ whose objects are the natural numbers, where the hom-object $\DFin_\V(n,m)$ from $n$ to $m$ is defined to be the internal hom $\uV(n\cdot 1,m \cdot 1)$ in $\V$.  Composition and identities are as in $\uV$, so that we have a fully faithful $\V$-functor 
\begin{equation}\label{eqn:j}j:\DFin_\V \rightarrowtail \uV\end{equation}
given on objects by $n \mapsto n \cdot 1$.  This $\V$-functor $j$ carries the structure of an eleutheric system of arities \cite[3.7, 7.5(5)]{Lu:EnrAlgTh}, where the monoidal product of objects in $\DFin_\V$ is the usual multiplication of natural numbers.
\end{ParSubSub}

\begin{DefSubSub}\label{def:th}
A $\J$-theory for the system of arities $\J = \DFin_\V$ will be called a \textbf{discretely finitary theory} (enriched in $\V$) or, for brevity, simply a \textbf{theory}.  We shall denote by $\Th$ the category of such theories, i.e., $\Th = \Th_{\DFin_\V}$ \pref{par:jth}.
\end{DefSubSub}

Borceux and Day studied an equivalent notion of theory in \cite{BoDay}, but with different assumptions on $\V$; see \cite[4.2(6)]{Lu:EnrAlgTh} for a detailed comparison\footnote{Discretely finitary theories are also closely related to (but distinct from) Power's \textit{discrete Lawvere theories} \cite{Pow:DiscLTh}.}.  In the case where $\V = \Set$, $\DFin_{\Set}$ is the full subcategory $\FinCard \hookrightarrow \Set$ consisting of the finite cardinals, and so we recover Lawvere's notion of \textit{algebraic theory} \cite{Law:PhD}; these are often called \textbf{Lawvere theories}.

For any object $C$ of a $\V$-category $\C$ and any $n \in \NN$, a cotensor $[j(n),C]$ of $C$ by the object $j(n) = n \cdot 1$ of $\V$ is the same as a conical $n$-th power $C^n$.  Therefore, everything stated in \S\ref{sec:enr_alg_th_sys_ar} applies to discretely finitary theories when one replaces $\J$-cotensors with conical finite powers.  In particular, by \ref{par:jth} a discretely finitary theory $\T$ is equivalently given by a $\V$-category $\T$ whose objects are the natural numbers $n \in \NN$, in which each object $n$ is equipped with a family of morphisms $\pi_i:n \rightarrow 1$ $(i = 1,...,n)$ that exhibit $n$ as a (conical) $n$-th power of $1$ in the $\V$-category $\T$, with the further requirement that for $n = 1$ the automorphism $\pi_1:1 \rightarrow 1$ must be the identity.  It can sometimes be helpful to write, say, $T$ for the object $1$ of $\T$, so that objects of $\T$ are then (conical) $n$-th powers $T^n$ of $T$, all distinct, with $T = T^1$.

\begin{ParSubSub}\label{par:subth_efth}
Given a theory $\U$, a subtheory $\T$ of $\U$ is equivalently given by a family of subobjects $\iota_{n,m}:\T(n,m) \hookrightarrow \U(n,m)$ in $\V$ $(n,m \in \NN)$ such that (i) the composition morphisms $\U(n,m)\times\U(m,\ell) \rightarrow \U(n,\ell)$ for $\U$ restrict to yield morphisms $\T(n,m)\times \T(m,\ell) \rightarrow \T(n,\ell)$, (ii) each of the designated projections $\pi_i:n \rightarrow 1$ in $\U$ lies in $\T(n,1)$, in the sense that $[\pi_i]:1 \rightarrow \U(n,1)$ factors through $\iota_{n,1}$, and (iii) the canonical isomorphisms $\U(n,1)^m \rightarrow \U(n,m)$ restrict to yield morphisms $\T(n,1)^m \rightarrow \T(n,m)$.

In the classical case of $\V = \Set$, every subtheory $\T \hookrightarrow \U$ of a Lawvere theory $\U$ is isomorphic (in $\Th \slash \U$) to one for which the associated mapping $\mor\T \rightarrow \mor\U$ is simply the inclusion of a subset, in which case we also call this subset $\mor\T \subseteq \mor\U$ a \textit{subtheory} of $\U$.  A subset $\sS \subseteq \mor\U$ is a subtheory in this sense iff the subset inclusions $\sS \cap \U(n,m) \hookrightarrow \U(n,m)$ satisfy conditions (i), (ii), (iii) of the preceding paragraph.  We shall say that a subset $\G \subseteq \mor\U$ is a \textbf{generating set of operations} for $\U$ if (a) every morphism $\mu:m \rightarrow n$ in $\G$ has codomain $n = 1$, and (b) for any subtheory $\sS \subseteq \mor\U$, if $\sS$ contains $\G$ then $\sS = \mor\U$.
\end{ParSubSub}  

\begin{ParSubSub}\label{par:pres_ops}
Let $\T$ be a Lawvere theory (with $\V = \Set$), let $A$ and $B$ be $\T$-algebras in a category $\C$.  Each morphism $f:\ca{A} \rightarrow \ca{B}$ in $\C$ induces a morphism $f^n:A(n) \rightarrow B(n)$ between the $n$-th powers $A(n) = \ca{A}^n$ and $B(n) = \ca{B}^n$ for each $n \in \NN$.  Given a morphism $\omega:n \rightarrow m$ in $\T$, we shall say that $f$ \textbf{preserves} $\omega$ if $f^m \cdot A(\omega) = B(\omega) \cdot f^n$.  It is well-known that $f$ is a $\T$-homomorphism from $A$ to $B$ iff $f$ preserves every morphism $\omega$ in $\T$.  Moreover, if $\G$ is a generating set of operations for $\T$ \pref{par:subth_efth}, then it is straightforward to show that $f$ is a $\T$-homomorphism iff $f$ preserves each $\omega \in \G$, by using the fact that the set of all morphisms preserved by $f$ is a subtheory of $\T$.
\end{ParSubSub}

\begin{ParSubSub}\label{par:disc_fin_vmonad}
Since the system of arities $j:\DFin_\V \rightarrow \uV$ is eleutheric \pref{par:efth}, we have an equivalence between the category $\Th$ of discretely finitary theories enriched in $\V$ and the category $\Mnd_{\DFin_\V}(\uV)$ of $\DFin_\V$-ary monads on $\uV$, which we shall call \textbf{discretely finitary $\V$-monads}.  We shall make use of the following convenient characterization of such $\V$-monads in terms of the notion of $\J$-flat colimit \pref{par:jary_mnds} with $\J = \DFin_\V$:
\end{ParSubSub}

\begin{PropSubSub}\label{thm:nvary_monads_via_ctbl_nv_flat_colims}
Let $\TT = (T,\eta,\mu)$ be a $\V$-monad on $\uV$.  Then $\TT$ is a discretely finitary $\V$-monad if and only if $T$ preserves objectwise-countable $\DFin_\V$-flat colimits.
\end{PropSubSub}
\begin{proof}
If $\TT$ is a discretely finitary $\V$-monad, then by \ref{par:jary_mnds} we know that $T$ conditionally preserves $\DFin_\V$-flat colimits, but since $\V$ is countably cocomplete it follows that $\uV$ has objectwise-countable colimits, so this entails that $T$ preserves objectwise-countable $\DFin_\V$-flat colimits.  For the converse, note that for each object $V$ of $\V$, the weight $\uV(j-,V):\DFin_\V^\op \rightarrow \uV$ is $\DFin_\V$-flat, by \ref{par:jary_mnds}, so if $T$ preserves objectwise-countable $\DFin_\V$-flat colimits then $T$ preserves $\uV(j-,V)$-weighted colimits and hence is a discretely finitary $\V$-monad by \ref{par:jary_mnds}.
\end{proof}

By \ref{par:vcat_talgs}, the $\V$-category $\Alg{\T}$ of $\T$-algebras (in $\uV$) for a theory $\T$ necessarily exists.  We shall need to make use of certain conical colimits of $\T$-algebras, furnished by the following:

\begin{PropSubSub}\label{thm:con_colims_algs}
The $\V$-category of $\T$-algebras (resp. normal $\T$-algebras) for a discretely finitary theory $\T$ has all conical countable colimits; further, it has all conical small colimits if $\V$ is cocomplete.
\end{PropSubSub}
\begin{proof}
By \eqref{eq:equiv_talgs_ntalgs}, it suffices to treat the category $\A = \Alg{\T}^!$ of normal $\T$-algebras.  The forgetful $\V$-functor $G:\A \rightarrow \uV$ creates (conical) reflexive coequalizers, by \cite[6.3, 6.5/6.7]{Lu:EnrAlgTh}, so $\A$ has reflexive coequalizers.  By \ref{par:ntalgs_emalgs}, we know that $\A$ is strictly $\V$-monadic over $\uV$.  It follows that $\A_0$ is strictly monadic over $\V$, so since $\A_0$ has reflexive coequalizers and $\V$ has countable colimits, we can apply Linton's Argument (i.e., the argument in \cite[Cor. 2]{Lin:CoeqCatsAlgs}) to deduce that $\A_0$ has all countable colimits.  But $\A$ is a cotensored $\V$-category, so conical colimits in $\A$ are the same as ordinary colimits in $\A_0$ (e.g., by \cite[\S 3.8]{Ke:Ba}).  Similar reasoning yields the result concerning small colimits.
\end{proof}

\begin{ParSubSub}\label{par:free_on_n_gens}
Given a $\V$-category $\A$ \textbf{over} $\uV$, i.e. a $\V$-category $\A$ equipped with a $\V$-functor $G:\A \rightarrow \uV$, we will say that an object $B$ of $\A$ is \textbf{free on $n$ generators}, for a given natural number $n$, if $B$ is equipped with an isomorphism $\A(B,-) \cong (G-)^n:\A \rightarrow \uV$.
\end{ParSubSub}

\begin{ParSubSub}\label{par:yoneda}
Let $\T$ be a theory.  For each $n \in \NN$, the $\V$-functor $Y(n) = \T(n,-):\T \rightarrow \uV$ is a $\T$-algebra, and by the Yoneda Lemma we have a fully faithful $\V$-functor $Y:\T \rightarrow \Alg{\T}^\op$ and isomorphisms $\Alg{\T}(Y(n),A) \cong A(n)$ that are $\V$-natural in $n \in \T$ and $A \in \Alg{\T}$.  Hence for each $n \in \NN$, the object $Y(n)$ of $\Alg{\T}$ is free on $n$ generators, since $\A(Y(n),A) \cong \ca{A}^n$, $\V$-naturally in $A \in \Alg{\T}$.
\end{ParSubSub}

Given a $\V$-category $\A$ over $\uV$, with associated $\V$-functor $G:\A \rightarrow \uV$, we say that $\A$ is \textbf{(strictly) discretely finitary algebraic} over $\uV$ if $G$ is (strictly) $\DFin_\V$-algebraic in the sense of \ref{par:ntalgs_emalgs}.  For brevity in the present paper, we will omit the qualification ``strictly'', with the understanding that strictness is implied unless otherwise indicated.  The following provides a flexible and practically useful way of finding a theory $\T$ for which $\Alg{\T}^! \cong \A$.

\begin{PropSubSub}\label{thm:assoc_th_assoc_alg}
Let $\A$ be a discretely finitary algebraic $\V$-category over $\uV$, via $G:\A \rightarrow \uV$.  Let $\T$ be a $\V$-category with $\ob\T = \NN$, and let $E:\T \rightarrow \A^\op$ be a fully faithful $\V$-functor sending each $n \in \ob\T = \NN$ to an object $E(n)$ of $\A$ that is free on $n$ generators.  Then $\T$ is a theory, and there is an isomorphism $\A \cong \Alg{\T}^!$ that commutes with the associated $\V$-functors to $\uV$.  For each object $A$ of $\A$, the corresponding normal $\T$-algebra has carrier $GA$, and its structural morphisms are the composites
$$\T(n,1) \times (GA)^n \cong GE(n) \times \A(E(n),A) \xrightarrow{1\times G_{E(n),A}} GE(n) \times \uV(GE(n),GA) \xrightarrow{\Ev} GA$$
where the first factor is the evident isomorphism.
\end{PropSubSub}
\begin{proof}
We may identify $G:\A \rightarrow \uV$ with $\ca{\text{$-$}}:\Alg{\U}^! \rightarrow \uV$ for some theory $\U$.  Writing $I:\A \hookrightarrow \Alg{\U}$ for the inclusion, we deduce that both $I^\op E:\T \rightarrow \Alg{\U}^\op$ and $Y:\U \rightarrow \Alg{\U}^\op$ \pref{par:yoneda} are fully faithful $\V$-functors sending each $n \in \ob\T = \ob\U = \NN$ to an object of $\Alg{\U}$ that is free on $n$ generators, so there is an identity-on-objects isomorphism $H:\T \xrightarrow{\sim} \U$ and an isomorphism $I^\op E \cong YH$.  Hence w.l.o.g. we may assume that $\U = \T$ and that $I^\op E \cong Y$, so that $\A = \Alg{\T}^!$.  By \ref{par:yoneda}, the $\T$-algebra $A$ is isomorphic to $\Alg{\T}(Y-,A) = \A(E-,A):\T \rightarrow \uV$, and the components of the resulting isomorphism $\A(E-,A) \cong A$ are the isomorphisms $\A(E(n),A) \cong (GA)^n = A(n)$ associated to the free objects $E(n)$ in $\A$.  The result now follows by a straightforward verification, using the $\V$-functoriality of $G$.
\end{proof}

\begin{RemSubSub}[\textbf{Associated theories}]\label{rem:assoc_th}
There is more than one canonical choice of theory $\T$ meeting the requirements of \ref{thm:assoc_th_assoc_alg}, but obviously all such choices are isomorphic.  For example, letting $F \dashv G:\A \rightarrow \uV$, we can take $\T$ to be the $\V$-category with hom-objects $\T(n,m) = \A(Fm,Fn)$ $(n,m \in \NN)$, with composition and identities as in $\A^\op$, where we have written simply $n$ to denote the $n$-th copower $n \cdot 1$ of the terminal object $1$ of $\V$.  But since $Fm$ is free on $m$ generators we have isomorphisms $\A(Fm,Fn) \cong (GFn)^m$, so there is an isomorphic theory $\T$ with $\T(n,m) = (GFn)^m$, and in particular $\T(n,1) = GFn$.  Given an object $A$ of $\A$ we deduce by \ref{thm:assoc_th_assoc_alg} that the corresponding normal $\T$-algebra has carrier $GA$ and structural morphisms
\begin{equation}\label{eq:str_morphs_corr_alg_for_assoc_th}GFn \times (GA)^n \cong GFn \times \A(Fn,A) \xrightarrow{1\times G_{Fn,A}} GFn \times \uV(GFn,GA) \xrightarrow{\Ev} GA\;.\end{equation}
\end{RemSubSub}

\begin{RemSubSub}\label{rem:power_cones_in_assoc_th}
In the situation of \ref{thm:assoc_th_assoc_alg}, the object $E(n)$ is an $n$-th copower of $E(1)$ in $\A$, for each $n \in \NN$.  Writing $(\iota_i:E(1) \rightarrow E(n))_{i=1}^n$ for the associated cocone, the fact that $E$ is fully faithful entails that for each $i = 1,...,n$ there is a unique morphism $\pi_i:n \rightarrow 1$ in $\T$ with $E(\pi_i) = \iota_i$, and in view of the proof of \ref{thm:assoc_th_assoc_alg} it is easy to show that the morphisms $\pi_i$ are the designated $n$-th power projections in $\T$.
\end{RemSubSub}

\subsection{Commutants for theories and monads}\label{sec:cmtn}

Let $j:\J \rightarrow \uV$ be an eleutheric system of arities \pref{par:jary_mnds}.

\begin{ParSubSub}[\textbf{Commutants for theories}]\label{par:cmtnts_th}
Given a morphism of $\J$-theories $A:\T \rightarrow \U$, we can construe $A$ as a $\T$-algebra in $\U$ with carrier $I$ \pref{par:alg_jth}.  Assuming that $\Alg{\T}_\U$ exists, the \textbf{commutant} $\T^\perp_A$ of $A$ (or of $\T$ with respect to $A$) is defined in \cite[7.1]{Lu:Cmtnts} as the full $\J$-theory $(\Alg{\T}_\U)_A$ of $A$ in the $\V$-category of $\T$-algebras in $\U$.  Hence $\T^\perp_A$ is an instance of \textit{the full $\J$-theory of an object} in the sense of \ref{par:full_th}, and $\T^\perp_A$ has hom-objects
\begin{equation}\label{eq:cmt}\T^\perp_A(J,K) = \Alg{\T}_\U([J,A],[K,A])\;\;\;\;\;\;(J,K \in \ob\J),\end{equation}
with composition and identities as in $\Alg{\T}_\U$.  Here $[J,A]$ and $[K,A]$ are the cotensors of $A$ by $J$ and $K$ in $\Alg{\T}_\U$, and so $\T^\perp_A(J,K)$ is the object of $\T$-homomorphisms between these $\T$-algebras.  

We can choose each pointwise cotensor $[J,A]$ with $J \in \ob\J$ in such a way that the `carrier' $\V$-functor $G = \Ev_I:\Alg{\T}_\U \rightarrow \U$ sends $[J,A]$ to the designated cotensor $[J,I] = J$ in $\U$ \pref{par:jth}.  Then the periphery of the following diagram commutes (strictly)
\begin{equation}\label{eq:cmtnt_diagr}
\xymatrix{
\J^\op \ar[rr]^{[-,A]} \ar[dr]^{b} \ar@/_{2ex}/[ddr]_{\upsilon} & & \Alg{\T}_\U \ar@/^{2ex}/[ddl]^{G}\\
& \T^\perp_A \ar[d]^k \ar[ur]^(.4){i} & \\
& \U & 
}
\end{equation}
where $\upsilon = [-,I]:\J^\op \rightarrow \U$ is the unique morphism of $\J$-theories \pref{par:jth}.  By definition, the commutant $\T^\perp_A$ is obtained by factoring $[-,A]$ as a bijective-on-objects $\V$-functor $b$ followed by a fully faithful $\V$-functor $i$ as in this diagram.  Letting $k = G \circ i$ as in the diagram, we find that $k$ is faithful since both $i$ and $G$ are so \pref{par:vcat_talgs}.  It is now immediate that $k \circ b = \upsilon$, so the diagram commutes.  But $b = [-,I]:\J^\op \rightarrow \T^\perp_A$ is the unique morphism of $\J$-theories, so $k$ presents $\T^\perp_A$ as a subtheory of $\U$ \pref{par:jth}.

Even if $\Alg{\T}_\U$ does not exist, we can form the commutant $\T^\perp_A$ as soon as the needed objects of $\V$-natural transformations \eqref{eq:cmt} exist, in which case we say that $\T^\perp_A$ \textbf{exists} \cite[7.1]{Lu:Cmtnts}.  By \ref{par:vcat_talgs}, if $\V$ has intersections of $(\ob\J)$-indexed families of strong subobjects, then every morphism of $\J$-theories has a commutant.

Given instead an arbitrary $\T$-algebra $A:\T \rightarrow \uV$ in $\uV$, we can consider the induced morphism of $\J$-theories $A:\T \rightarrow \uV_{|A|}$ \pref{par:full_th}.  The commutant $\T^\perp_A \hookrightarrow \uV_{|A|}$ of the latter morphism \textit{always} exists \cite[10.15]{Lu:Cmtnts}, and we call it the \textbf{commutant} of $\T$ with respect the $\T$-algebra $A$.  Here $\T^\perp_A$ is precisely the full $\J$-theory of $A$ in $\Alg{\T}$ \cite[7.10]{Lu:Cmtnts}.  By \ref{par:full_th} and \ref{par:lims_algs}, we deduce that $\ca{A}$ is the carrier of a $\T^\perp_A$-algebra that we sometimes denote by $A:\T^\perp_A \rightarrow \uV$, by abuse of notation, and the latter $\T^\perp_A$-algebra is normal if the $\T$-algebra $A$ is normal.
\end{ParSubSub}

\begin{ParSubSub}[\textbf{Commutation}]\label{par:cmtn}
The notion of commutant is related to a notion of \textit{commutation}, as follows.  Given morphisms of $\J$-theories $A:\T \rightarrow \U$ and $B:\sS \rightarrow \U$, let us assume that the commutants $\T^\perp_A$ and $\sS^\perp_B$ of $A$ and $B$ both exist.  Then $B$ factors through the commutant $\T^\perp_A \hookrightarrow \U$ if and only if $A$ factors through the commutant $\sS^\perp_B \hookrightarrow \U$ \cite[7.8, 5.8]{Lu:Cmtnts}.  If these equivalent conditions hold, then we say that $A$ \textbf{commutes with} $B$, written $A \perp B$.  Hence this relation $\perp$ is symmetric, i.e. $A \perp B \;\Leftrightarrow\; B \perp A$.  This notion of commutation can defined without reference to the notion of commutant, as is done in \cite[5.12]{Lu:Cmtnts}, where it is shown that the above characterization in terms of commutants is equivalent \cite[7.8]{Lu:Cmtnts}.

It is proved in \cite[10.5]{Lu:Cmtnts} that this notion of commutation is equivalent to Kock's notion of commutation of cospans of $\V$-monads on $\uV$ \cite[4.1]{Kock:DblDln}, in the sense that if $\alpha:\TT \rightarrow \UU$ and $\beta:\SSS \rightarrow \UU$ are the morphisms of $\J$-ary monads corresponding to $A$ and $B$ under the equivalence $\ThJ \simeq \MndJ(\uV)$ \pref{par:jary_mnds}, then $A$ commutes with $B$ if and only if $\alpha$ commutes with $\beta$ in Kock's sense.  In this way, the notion of commutation is independent of the choice of the system of arities $\J$. 
\end{ParSubSub}

\begin{ParSubSub}[\textbf{Commutants for monads}]\label{par:cmtnts_mnds}
Unlike the notion of commutation \pref{par:cmtn}, the notion of \textit{commutant} depends on the choice of arities $\J$ \cite[10.12]{Lu:Cmtnts}.  If $\alpha:\TT \rightarrow \UU$ is a morphism of $\J$-ary monads for which the corresponding morphism of $\J$-theories $A:\T \rightarrow \U$ has a commutant $\T^\perp_A$, then the \textbf{$\J$-ary commutant} $\TTperpJwrt{\alpha}$ of $\alpha$ is defined as the $\J$-ary monad corresponding to $\T^\perp_A$ \cite[10.8]{Lu:Cmtnts}.  For the system of arities $1_{\uV}:\uV \rightarrow \uV$, recall that $\uV$-ary monads are simply arbitrary $\V$-monads on $\uV$ \pref{par:jary_mnds}.  Given a morphism of $\V$-monads $\alpha:\TT \rightarrow \UU$ we call the $\uV$-ary commutant of $\alpha$ (if it exists) the \textbf{absolute commutant} of $\alpha$ \cite[10.8]{Lu:Cmtnts}.

In particular, if we are given a $\J$-ary monad $\TT$ and a $\TT$-algebra $A$, then upon letting $\T$ be the corresponding $\J$-theory and $A'$ the corresponding normal $\T$-algebra in $\uV$ \pref{par:ntalgs_emalgs}, the \textbf{$\J$-ary commutant} $\TTperpJwrt{A}$ of $\TT$ with respect to $A$ may be defined as the $\J$-ary monad corresponding to the commutant $\T^\perp_{A'}$, which always exists \pref{par:cmtnts_th}.  As a special case, we obtain the notion of \textbf{absolute commutant} $\TT^\perp_A := \TT^\perp_{\kern-.1ex\scriptscriptstyle\uV A}$ of an arbitrary $\V$-monad $\TT$ on $\uV$, with respect to a $\TT$-algebra $A$, which always exists.

Hence, given a $\J$-ary monad $\TT$ and a $\TT$-algebra $A$, we can form both $\TTperpJwrt{A}$ and $\TT^\perp_A$.  Again writing $\T$ for the $\J$-theory corresponding to $\TT$ and $A':\T \rightarrow \uV$ for the normal $\T$-algebra corresponding to $A$, we claim that the \textit{codensity $\V$-monad} $\SSS$ of $A'$ exists and is isomorphic to the \textit{absolute} commutant $\TT^\perp_A$.  One can consult \cite[II]{Dub} for a definition of the codensity $\V$-monad $\SSS$, whose underlying $\V$-functor $S:\uV \rightarrow \uV$ is the right Kan extension of $A'$ along $A'$, given by 
\begin{equation}\label{eq:codensity_ends}SV = \int_{J \in \T}\uV(\uV(V,A'J),A'J)\end{equation}
$\V$-naturally in $V \in \uV$.  Note that this right Kan extension exists, since the needed ends \eqref{eq:codensity_ends} are precisely the hom-objects $\Alg{\T}([V,A'],A')$ in $\Alg{\T}$, which necessarily exists \pref{par:vcat_talgs}.  The Kleisli $\V$-category $\uV_\SSS$ for $\SSS$ has hom-objects 
$$\uV_{\SSS}(W,V) = \uV(W,SV) \cong \int_{J \in \T}\uV(\uV(V,A'J),\uV(W,A'J)) = \Alg{\T}([V,A'],[W,A'])$$
for all $V,W \in \ob\V$.  Moreover, if we let $\sS = \uV_\SSS^\op$ and let $\F$ denote the full $\uV$-theory of $A'$ in $\Alg{\T}$, then by means of \cite[II.3]{Dub} or a straightforward verification we find that these isomorphisms constitute an isomorphism of $\uV$-theories $\sS \cong \F$, recalling that $\sS$ is the $\uV$-theory corresponding to $\SSS$ \pref{par:jary_mnds}.  The $\V$-monad $\TT$ has a corresponding $\uV$-theory $\sP = \uV_\TT^\op$, with a normal $\sP$-algebra $B$ corresponding to $A$, and we know that $\Alg{\T} \simeq \Alg{\TT} \simeq \Alg{\sP}$, so $\F$ is isomorphic to the full $\uV$-theory of $B$ in $\Alg{\sP}$, i.e. $\F$ is isomorphic to the commutant $\sP^\perp_B$.  Hence we have an isomorphism of $\uV$-theories $\sS \cong \sP^\perp_B$, and by passing to the corresponding $\V$-monads we deduce that $\SSS \cong \TT^\perp_A$.  A similar description of the \textit{$\J$-ary} commutant $\TTperpJwrt{A}$ as a codensity $\V$-monad is not available; rather, Proposition \ref{thm:jary_cmt_rest_abs_cmt} below entails that $\TTperpJwrt{A}$ is the $\J$-\textit{ary restriction} of the codensity $\V$-monad $\SSS \cong \TT^\perp_A$, in the sense of \ref{def:jary_restn}.
\end{ParSubSub}

\begin{ParSubSub}[\textbf{Commutativity}]\label{par:comm}
A $\J$-theory $\T$ is \textbf{commutative} if its corresponding $\J$-ary monad $\TT$ is commutative in the sense defined by Kock \cite[3.1]{Kock:Comm}, equivalently, if the identity morphism $1_\T:\T \rightarrow \T$ commutes with itself \cite[10.6]{Lu:Cmtnts} in the sense of \cite[5.12]{Lu:Cmtnts}, cf. \ref{par:cmtn}.  The commutant of $1_\T$ is called the \textbf{centre} of $\T$ and is denoted by $Z(\T)$ \cite[7.14]{Lu:Cmtnts}, provided that it exists.  Assuming that it does, $\T$ is commutative if and only if the subtheory inclusion $Z(\T) \hookrightarrow \T$ is an isomorphism.
\end{ParSubSub}

\begin{ParSubSub}[\textbf{Saturated and balanced theories over a base}]\label{par:sat_bal}
Fixing a $\J$-theory $\U$, a \textbf{$\J$-theory over $\U$} is, by definition, an object of the slice category $\ThJ \slash \U$, i.e. a $\J$-theory $\T$ equipped with a morphism $\T \rightarrow \U$.  The \textbf{commutant} $\T^\perp$ of a $\J$-theory $\T$ over $\U$ is, by definition, the commutant of the associated morphism $\T \rightarrow \U$.  

Let us assume that the commutant of every $\J$-theory over $\U$ exists.  Then there is a functor $(-)^\perp:(\ThJ\slash\U)^\op \rightarrow \ThJ\slash\U$ given on objects by $\T \mapsto \T^\perp$, and $(-)^\perp$ is right adjoint to its formal dual \cite[8.6]{Lu:Cmtnts}.  Further, this adjunction restricts to a Galois connection on the preordered class $\SubTh(\U)$ of subtheories of $\U$, i.e. an adjunction $(-)^\perp \dashv (-)^\perp:\SubTh(\U)^\op \rightarrow \SubTh(\U)$ \cite[8.7]{Lu:Cmtnts}.

A $\J$-theory $\T$ over $\U$ is said to be \textbf{saturated} if $\T^{\perp\perp} \cong \T$ (in $\ThJ\slash\U$), and $\T$ is said to be \textbf{balanced} if $\T^\perp \cong \T$.  Clearly $\T$ is saturated if and only if $\T$ is (isomorphic to) the commutant of some $\J$-theory over $\U$, so every saturated $\J$-theory over $\U$ is necessarily a subtheory of $\U$.  In particular, any balanced $\J$-theory over $\U$ is necessarily a saturated subtheory of $\U$.  By \cite[8.10]{Lu:Cmtnts}, a subtheory $\T$ of $\U$ is commutative if and only if $\T \lt \T^\perp$ as subtheories of $\U$.  Hence a balanced $\J$-theory over $\U$ is necessarily a commutative, saturated subtheory of $\U$.

Recall that the commutant of a $\J$-theory $\T$ with respect to a given $\T$-algebra in $\uV$ necessarily exists \pref{par:cmtnts_th}.  In view of \ref{par:full_th}, it follows that if we let $\U = \uV_{V}$ be the full $\J$-theory of a given object $V$ of $\V$ with designated $\J$-cotensors, then the commutant of every $\J$-theory over $\U$ exists.  Given a $\T$-algebra $A:\T \rightarrow \uV$, we shall say that $\T$ is \textbf{saturated} (resp. \textbf{balanced}) with respect to $A$ if the induced morphism of $\J$-theories $A:\T \rightarrow \uV_{|A|}$ \pref{par:full_th} exhibits $\T$ as a saturated (resp. balanced) subtheory of $\uV_{|A|}$.

In view of the equivalence between $\J$-theories and $\J$-ary monads \pref{par:jary_mnds}, the above notions yield corresponding concepts for $\J$-ary monads.  In particular, given a $\J$-ary monad $\TT$ and a $\TT$-algebra $A$, we can ask whether $\TT$ is \textbf{saturated} (resp. \textbf{balanced}) with respect to $A$.
\end{ParSubSub}

\subsection{Algebraic symmetric monoidal closed \texorpdfstring{$\V$}{V}-categories}\label{sec:alg_smc_vcats}

Recall that a \textit{symmetric monoidal closed $\V$-category} is a $\V$-category equipped with symmetric monoidal closed structure that is entirely $\V$-functorial and $\V$-natural; see, e.g., \cite[7.1]{Lu:RelSMCI} for a definition.  Let $j:\J \hookrightarrow \uV$ be an eleutheric system of arities \pref{par:jary_mnds}.

The following is a corollary of a well-known result of Kock \cite{Kock:ClsdCatsGenCommMnds}.

\begin{PropSubSub}\label{thm:talgs_comm_th_smcvcat}
Let $\T$ be a commutative $\J$-theory such that the category of normal $\T$-algebras $\Alg{\T}^!_0$ has reflexive coequalizers.  Then $\Alg{\T}^!$ carries the structure of a symmetric monoidal closed $\V$-category.  Further, the associated $\V$-adjunction $F \dashv G:\Alg{\T}^! \rightarrow \uV$ carries the structure of a symmetric monoidal $\V$-adjunction.
\end{PropSubSub}
\begin{proof}
Letting $\TT$ denote the commutative $\V$-monad on $\uV$ associated to $\T$ \pref{par:comm}, we may identify $\Alg{\T}^!$ with the $\V$-category of $\TT$-algebras and identify the $\V$-adjunction $F \dashv G$ with the Eilenberg-Moore $\V$-adjunction for $\TT$ \pref{par:ntalgs_emalgs}.  The result now follows from \cite[5.5.4, 5.5.6]{Lu:PhD} and \cite[11.2]{Lu:RelSMCI}.
\end{proof}

\begin{DefSubSub}\label{def:bimorph}
Let $A,B,C$ be $\T$-algebras for a $\J$-theory $\T$.  We say that a morphism $f:\ca{A}\times\ca{B} \rightarrow \ca{C}$ is a $\T$-\textbf{bimorphism} from $A,B$ to $C$ if both transposes $f_1:\ca{A} \rightarrow \uV(\ca{B},\ca{C})$ and $f_2:\ca{B} \rightarrow \uV(\ca{A},\ca{C})$ of $f$ are $\T$-homomorphisms when $\uV(\ca{B},\ca{C})$ and $\uV(\ca{A},\ca{C})$ are regarded as the carriers of the cotensors $[\ca{B},C]$ and $[\ca{A},C]$ of $C$ by $\ca{B},\ca{A}$ in $\Alg{\T}$, respectively \pref{par:lims_algs}. 
\end{DefSubSub}

\begin{PropSubSub}\label{thm:desc_mon_str_in_talgs}
Let $\T$ be a commutative $\J$-theory such that the category of normal $\T$-algebras $\Alg{\T}^!_0$ has reflexive coequalizers.  Then the unit object in $\Alg{\T}^!$ is the free normal $\T$-algebra on the unit object $I$ of $\V$.  Given normal $\T$-algebras $A$ and $B$, there is a bijection between morphisms $A \otimes B \rightarrow C$ in $\Alg{\T}^!$ and $\T$-bimorphisms $\ca{A} \times \ca{B} \rightarrow \ca{C}$, natural in $C \in \Alg{\T}^!_0$, where $\otimes$ denotes the monoidal product in $\Alg{\T}^!$.
\end{PropSubSub}
\begin{proof}
In the notation of \ref{thm:talgs_comm_th_smcvcat}, the left adjoint $F$ is necessarily strong monoidal, by \cite[1.4]{Ke:Doctr}, so $F(I)$ is the unit object of $\Alg{\T}^!$.  Recall that in \ref{thm:talgs_comm_th_smcvcat} we identified $\Alg{\T}^!$ with the $\V$-category $\Alg{\TT}$ of $\TT$-algebras for the associated commutative $\V$-monad $\TT$.  By \cite[4.1]{Jac:SemWkngContr} or \cite[2.3.4]{Seal:TensMndActn} we know that there is a natural bijection between $\TT$-homomorphisms $A \otimes B \rightarrow C$ and $\TT$-\textit{bimorphisms} $f:A \times B \rightarrow C$ (\cite[\S5]{Jac:SemWkngContr} \cite[2.1]{Seal:TensMndActn}), where we now freely omit notational distinctions between $\TT$-algebras and their carriers.  But remarks in \cite[p. 101]{Kock:Dist} entail that $f$ is a $\TT$-bimorphism iff both transposes $f_1:A \rightarrow \uV(B,C)$ and $f_2:B \rightarrow \uV(A,C)$ are $\TT$-homomorphisms, where $\uV(A,C),\uV(B,C)$ are here regarded as the cotensors of $C$ by the carriers of $A,B$ in $\Alg{\TT} = \Alg{\T}^!$, respectively.
\end{proof}

\section{Functional distribution monads and functional-analytic contexts}\label{sec:cdistn_mnds}

With reference to \S\ref{sec:intro}, the $\V$-category of $\T$-algebras for a commutative $\J$-theory $\T$ can be seen as supporting an abstract form of functional analysis with respect to any suitable `dualizing object' $S$, and this leads to the following definition:

\begin{DefSub}\label{def:fa_ctxt}\emptybox
\begin{enumerate}
\item A \textbf{functional-analytic context} $(\V,\J,\T,S)$ consists of (1) a symmetric mon\-oid\-al closed category $\V$ with equalizers, (2) an eleutheric system of arities $\J \hookrightarrow \uV$, (3) a commutative $\J$-theory $\T$, and (3) a $\T$-algebra $S$ in $\uV$ such that $\T$ is saturated with respect to $S$ \pref{par:sat_bal}.
\item A \textbf{(discretely) finitary functional-analytic context} is a functional-analytic context $(\V,\J,\T,S)$ in which $\V$ is a cartesian closed category with countable colimits and $\J$ is the system of arities $\DFin_\V$ \pref{eqn:j}.  Equivalently, a finitary functional-analytic context is given by a triple $(\V,\T,S)$ in which $\V$ is a cartesian closed category with equalizers and countable colimits, $\T$ is a commutative, discretely finitary theory enriched in $\V$, and $S$ is a $\T$-algebra in $\uV$ such that $\T$ is saturated with respect to $S$ \pref{par:sat_bal}.
\item A functional-analytic context $(\V,\J,\T,S)$ is said to be \textbf{balanced} if $\T$ is balanced with respect to $S$ \pref{par:sat_bal}.
\end{enumerate}
\end{DefSub}

\begin{ParSub}[\textbf{Functional-analytic contexts via monads}]\label{par:func_an_ctxt_mnds}
In view of \ref{par:jary_mnds}, \ref{par:ntalgs_emalgs}, \ref{par:comm}, and \ref{par:sat_bal}, a functional-analytic context is equivalently given by a quadruple $(\V,\J,\TT,S)$ where $\J \hookrightarrow \uV$ is an eleutheric system of arities in a symmetric monoidal closed category $\V$ with equalizers, $\TT$ is a commutative $\J$-ary monad \pref{par:jary_mnds}, and $S$ is a $\TT$-algebra such that $\TT$ is saturated with respect to $S$ \pref{par:sat_bal}.  In particular, a finitary functional-analytic context is equivalently given by a suitable triple $(\V,\TT,S)$ in which $\TT$ is a discretely finitary $\V$-monad \pref{par:disc_fin_vmonad}.
\end{ParSub}

\begin{ExaSub}\label{exa:cls_exa_ffa_ctxts}
The following general classes of examples of finitary functional-analytic contexts will be developed in subsequent sections.  Here $\V$ is an arbitrary cartesian closed category with equalizers, countable colimits, and intersections of countable families of strong subobjects.
\begin{enumerate}
\item Given a commutative rig $R$ in $\V$ \pref{def:rig_ring_module_bim_cmons}, there is a theory $\T = \Mat_R$ such that normal $\T$-algebras may be identified with $R$-modules in $\uV$ \pref{thm:th_lrmods}.  Taking $S = R$, we obtain a balanced finitary functional-analytic context $(\V,\Mat_R,R)$ \pref{par:sc_rlin_ctxt}, which we call the \textbf{scalar $R$-linear context} in $\V$.
\item Given a commutative rig $R$ in $\V$, there is a theory $\T = \Mat_R^\aff$ such that normal $\T$-algebras are \textbf{$R$-affine spaces} in $\V$ \pref{def:raff_sp}.  Taking $S = R$, we obtain a finitary functional-analytic context $(\V,\Mat_R^\aff,R)$ \pref{par:sc_raff_ctxt}, which we call the \textbf{scalar $R$-affine context} in $\V$.
\item Given a commutative \textit{preordered ring} $R$ in $\V$ \pref{par:preord_ring}, the \textit{positive part} $R_+$ of $R$ is a rig in $\V$ \pref{par:preord_ring}.  By definition, an \textbf{$R$-convex space} in $\V$ is an $R_+$-affine space in $\V$ \pref{def:rcvx_sp}.  We define the \textbf{positive $R$-convex context} in $\V$ as the scalar $R_+$-affine context $(\V,\Mat_{R_+}^\aff,R_+)$.
\end{enumerate}
\end{ExaSub}

\begin{ParSub}[\textbf{Notation}]\label{par:notn}
Given a functional-analytic context $(\V,\J,\T,S)$, we shall write $\T^\perp$ to denote the commutant of $\T$ with respect to $S$ \pref{par:cmtnts_th}, recalling that commutants with respect to $\T$-algebras in $\uV$ necessarily exist \pref{par:cmtnts_th}.  By \ref{par:cmtnts_th}, the carrier $\ca{S}$ of $S$ is also the carrier of a $\T^\perp$-algebra, which we denote also by $S$, by abuse of notation.  By \ref{par:lims_algs}, we know that $\Alg{\T^\perp}$ is a cotensored $\V$-category, so there is a $\V$-adjunction 
\begin{equation}\label{eq:hom_cot_for_the_tperp_alg_s}[-,S] \dashv \Alg{\T^\perp}(-,S)\;\;:\;\;\Alg{\T^\perp}^\op \rightarrow \uV\end{equation}
in which the $\V$-functor $[-,S]$ sends each object $V$ of $\V$ to the cotensor $[V,S]$ of $S$ by $V$ in $\Alg{\T^\perp}$, whose carrier is the internal hom $\uV(V,\ca{S})$ in $\V$ \pref{par:lims_algs}.
\end{ParSub}

\begin{DefSub}\label{def:cdisnt_mnd}
Given a functional-analytic context $(\V,\J,\T,S)$, we define the \textbf{functional distribution monad} 
$$\DD_{\scriptscriptstyle(\T,S)} = (D_{\scriptscriptstyle(\T,S)},\delta,\kappa)$$
determined by $(\V,\J,\T,S)$ as the $\V$-monad on $\uV$ induced by the $\V$-adjunction \eqref{eq:hom_cot_for_the_tperp_alg_s}.
Hence the underlying $\V$-endofunctor $D_{\scriptscriptstyle (\T,S)}:\uV \rightarrow \uV$ is given by
$$D_{\scriptscriptstyle(\T,S)}(V) = \Alg{\T^\perp}([V,S],S)$$
$\V$-naturally in $V \in \V$, so that $D_{\scriptscriptstyle(\T,S)}(V)$ is the object of $\T^\perp$-algebra homomorphisms from $[V,S]$ to $S$.  Individual $\T^\perp$-algebra homomorphisms $\mu:[V,S] \rightarrow S$ will be called \textbf{$(\T,S)$-distributions on $V$} or \textbf{functional distributions on $V$} in the context $(\V,\J,\T,S)$.  Hence we call $D_{\scriptscriptstyle(\T,S)}(V)$ the \textbf{object of $(\T,S)$-distributions on $V$}, and so we also call $\DD_{\scriptscriptstyle(\T,S)}$ the \textbf{$(\T,S)$-distribution monad}.
\end{DefSub}

The next section (\S\ref{sec:exa_cdistn_mnds}) describes several examples of functional distribution monads.

\begin{ParSub}[\textbf{Monad-theoretic notation for functional distributions}]\label{par:mnd_notn_fdistn}
Given a functional-analytic context $(\V,\J,\TT,S)$, formulated in monad-theoretic terms as in \ref{par:func_an_ctxt_mnds}, we denote by $\DD_{\scriptscriptstyle(\J,\TT,S)}$ the functional distribution monad determined by $(\V,\J,\TT,S)$, thus explicitly indicating the system of arities $\J$ within the notation, since the datum $\TT$ alone does not serve to specify $\J$.  We denote by $\TTperpJ$ the $\J$-ary commutant of $\TT$ with respect to $S$, which necessarily exists \pref{par:cmtnts_mnds}.  As in \ref{par:notn}, we know that the carrier of $S$ also carries the structure of a $\TTperpJ$-algebra, which we denote also by $S$, by abuse of notation.  
\end{ParSub}

By taking a purely monad-theoretic view, we obtain the following characterization of the functional distribution monad as a kind of `double commutant':

\begin{ThmSub}\label{thm:cdistn_mnd_dbl_cmtnt}
Let $(\V,\J,\TT,S)$ be a functional-analytic context.  Then the functional distribution monad $\DD_{\scriptscriptstyle(\J,\TT,S)}$ is (isomorphic to) the absolute commutant of $\TTperpJ$ with respect to $S$, i.e.
$$\DD_{\scriptscriptstyle(\J,\TT,S)} \cong (\TTperpJ)^\perp\;.$$
\end{ThmSub}

In order to prove this theorem, we will require some lemmas.  Recall that there is an equivalence between $\V$-monads on $\uV$ and $\uV$-theories, for the system of arities $1_{\uV}:\uV \rightarrow \uV$ \pref{par:jary_mnds}.  Given a $\V$-monad $\UU$ on $\uV$, let $\U$ denote the associated $\uV$-theory.  Then $\U$ is a cotensored $\V$-category, so the object $I$ of $\U$ determines an associated $\V$-adjunction
\begin{equation}\label{eq:hom_cot_adj_for_i_in_vth}[-,I] \dashv \U(-,I)\;\;:\;\;\U^\op \rightarrow \uV\end{equation}
whose left adjoint supplies the cotensors of $I$ in $\U$.

\begin{LemSub}\label{lem:recovering_monad_from_vth}
Given any $\V$-monad $\UU$ on $\uV$, with associated $\uV$-theory $\U$, the $\V$-monad induced by the $\V$-adjunction \eqref{eq:hom_cot_adj_for_i_in_vth} is isomorphic to $\UU$.
\end{LemSub}
\begin{proof}
Concretely, $\U$ is the opposite $\uV_\UU^\op$ of the Kleisli $\V$-category $\uV_\UU$ \pref{par:jary_mnds}.  Letting $F \dashv G:\uV_\UU \rightarrow \uV$ denote the Kleisli $\V$-adjunction, the left adjoint $F:\uV \rightarrow \U^\op$ preserves tensors and sends $I$ to $I$.  Hence since tensors in $\U^\op$ are the same as cotensors in $\U$, we compute that $FV \cong F(V \otimes I) \cong [V,I]$ in $\U$, $\V$-naturally in $V \in \uV$.  Therefore $F$ is isomorphic to the $\V$-functor $[-,I]$ appearing in \eqref{eq:hom_cot_adj_for_i_in_vth}, so since $\UU$ is induced by the Kleisli $\V$-adjunction, the result follows.
\end{proof}

\begin{LemSub}\label{thm:charn_abs_cmtnt}
Let $\PP$ be a $\V$-monad on $\uV$, let $A$ be a $\PP$-algebra, and let $\PP^\perp$ denote the absolute commutant of $\PP$ with respect to $A$ (which necessarily exists, \ref{par:cmtnts_mnds}).  Then $\PP^\perp$ is isomorphic to the $\V$-monad induced by the `cotensor-hom' $\V$-adjunction
\begin{equation}\label{eq:hom_cot_adj_emalg}[-,A] \dashv \Alg{\PP}(-,A):\Alg{\PP}^\op \rightarrow \uV\end{equation}
for the object $A$ of $\Alg{\PP}$.
\end{LemSub}
\begin{proof}
Let $\sP$ denote the $\uV$-theory determined by $\PP$.  By definition, $\PP^\perp$ is the $\V$-monad determined by the commutant $\U := \sP^\perp$ of $\sP$ with respect to the normal $\sP$-algebra corresponding to $A$.  Letting $\A = \Alg{\PP}$, we may identify $\A$ with the isomorphic $\V$-category $\Alg{\sP}^!$ \pref{par:ntalgs_emalgs}.  By \ref{lem:recovering_monad_from_vth}, $\PP^\perp$ is isomorphic to the $\V$-monad induced by the $\V$-adjunction $[-,I] \dashv \U(-,I):\U^\op \rightarrow \uV$.  But by \ref{par:cmtnts_th} and \ref{par:full_th}, the mapping $\ob\V = \ob\U \rightarrow \ob\A$ given by $V \mapsto [V,A]$ underlies a cotensor-preserving $\V$-functor $\iota:\U \rightarrow \A$ that is fully faithful and maps $\U$ onto the full sub-$\V$-category $\A' \hookrightarrow \A$ consisting of the cotensors of $A$.  Hence $\PP^\perp$ is isomorphic to the $\V$-monad induced by the $\V$-adjunction $[-,A] \dashv \A'(-,A):(\A')^\op \rightarrow \uV$, and the result follows.
\end{proof}

\begin{proof}[Proof of Theorem \ref{thm:cdistn_mnd_dbl_cmtnt}]
Given a functional-analytic context $(\V,\J,\TT,S)$, we can take $\PP := \TTperpJ$ and invoke \ref{thm:charn_abs_cmtnt} to obtain a characterization of the absolute commutant $\PP^\perp$ with respect to $S$, from which the result follows.
\end{proof}

Let $j:\J \hookrightarrow \uV$ be an eleutheric system of arities.

\begin{DefSub}\label{def:jary_restn}
Given any $\uV$-theory $\U$, we denote by $\U|_\J$ the full sub-$\V$-category of $\U$ on the objects of $\J$.  Note that $\U|_\J$ is clearly a $\J$-theory, which we call the \textbf{restriction of $\U$ to $\J$}.  Similarly, given a $\V$-monad $\UU$ on $\uV$, with associated $\uV$-theory $\U$, we denote by $\UU|_\J$ the $\J$-ary monad determined by $\U|_\J$, and we call $\UU|_\J$ the \textbf{$\J$-ary restriction} of $\UU$.
\end{DefSub}

\begin{PropSub}\label{thm:jary_restn}
Let $\UU = (U,\eta,\mu)$ be a $\V$-monad on $\uV$, with $\J$-ary restriction $\UU|_\J = (U',\eta',\mu')$.  Then $U \circ j \cong U' \circ j:\J \rightarrow \uV$, where $j:\J \hookrightarrow \uV$ is the inclusion.
\end{PropSub}
\begin{proof}
Let $\U$ be the $\uV$-theory determined by $\UU$, and let $\upsilon:\uV^\op \rightarrow \U$ be the unique morphism of $\uV$-theories.  Then the unique morphism of $\J$-theories $\upsilon':\J^\op \rightarrow \U|_\J$ is simply a restriction of $\upsilon$.  By \ref{par:jary_mnds}, $U'$ is a left Kan extension of $(\U|_\J)(\upsilon'-,I):\J \rightarrow \uV$ along $j$, and $U \cong \U(\upsilon-,I)$.  Hence $U' \circ j \cong (\U|_\J)(\upsilon'-,I) \cong U \circ j$.
\end{proof}

\begin{PropSub}\label{thm:jary_cmt_rest_abs_cmt}
Given a $\J$-ary monad $\PP$ and a $\PP$-algebra $A$, the $\J$-ary commutant $\PPperpJwrt{A}$  is isomorphic to the $\J$-ary restriction of the absolute commutant $\PP^\perp_A$ (recalling that both $\PPperpJwrt{A}$ and $\PP^\perp_A$ necessarily exist, \ref{par:cmtnts_mnds}).
\end{PropSub}
\begin{proof}
Let $\sP$ denote the $\uV$-theory determined by $\PP$, and let $\sS$ denote the $\J$-theory determined by $\PP$.  We may identify the isomorphic $\V$-categories $\Alg{\sP}^!$, $\Alg{\PP}$, and $\Alg{\sS}^!$.  If we write $\sP^\perp$ and $\sS^\perp$ for the commutants of $\sP$ and $\sS$ (respectively) with respect to $A$, then it suffices to show that $\sP^\perp|_\J \cong \sS^\perp$.  But $\sP^\perp$ is the full $\uV$-theory of $A$ in $\Alg{\sP}^! = \Alg{\sS}^!$ \pref{par:cmtnts_th}, and it follows that $\sP^\perp|_\J$ is the full $\J$-theory of $A$ in $\Alg{\sS}^!$, i.e. $\sP^\perp|_\J = \sS^\perp$ \pref{par:cmtnts_th}.
\end{proof}

\begin{ThmSub}\label{thm:t_jary_restn_of_d}
Let $(\uV,\J,\TT,S)$ be a functional-analytic context.  Then $\TT$ is isomorphic to the $\J$-ary restriction of the functional distribution monad $\DD_{\scriptscriptstyle(\J,\TT,S)}$.   I.e. $(\DD_{\scriptscriptstyle(\J,\TT,S)})|_\J \cong \TT$.  In particular, for each object $J$ of $\J$, $D_{\scriptscriptstyle(\J,\TT,S)}(J) \cong T(J)$ by \ref{thm:jary_restn}.
\end{ThmSub}
\begin{proof}
With the notational conventions of \ref{par:mnd_notn_fdistn}, let $\PP := \TTperpJ$.  By \ref{thm:cdistn_mnd_dbl_cmtnt}, $\DD_{\scriptscriptstyle(\J,\TT,S)}$ is the absolute commutant $\PP^\perp$ of $\PP$ with respect to the $\PP$-algebra $S$.  But by \ref{thm:jary_cmt_rest_abs_cmt}, the $\J$-ary restriction $\PP^\perp|_\J$ of $\PP^\perp$ is isomorphic to the $\J$-ary commutant $\PPperpJ$ of $\PP$ with respect to $S$, so $(\DD_{\scriptscriptstyle(\J,\TT,S)})|_\J \cong \PPperpJ = (\TTperpJ)^\perp_{\kern-.1ex\scriptscriptstyle\J} \cong \TT$ since $\TT$ is saturated with respect to $S$.
\end{proof}

\section{Examples of functional distribution monads}\label{sec:exa_cdistn_mnds}

We now describe several specific examples of functional distribution monads, using results that are proved in subsequent sections.

\begin{ExaSub}[\textbf{Radon measures of compact support}]\label{exa:cs_rad_meas}
Let $\V = \Conv$ be the category of convergence spaces \pref{par:conv_sp_rad_meas}.  Let $R$ denote either the real numbers $\RR$ or the complex numbers $\CC$, considered as a commutative ring in $\V$.  The scalar $\RR$-linear context $(\V,\T,S) = (\V,\Mat_R,R)$ in $\V$ is balanced (\ref{exa:cls_exa_ffa_ctxts}, \ref{thm:th_lrmods_sat_bal_iff_rcomm}, \ref{par:sc_rlin_ctxt}), i.e. $\T \cong \T^\perp$, and the $\V$-category of normal $\T$-algebras may be identified with the $\V$-category $\Mod{R}$ of $R$-modules in $\V$ \pref{thm:th_lrmods}, also known as \textit{convergence vector spaces} \cite{BeBu}.  Hence for each convergence space $V$,
$$D_{\scriptscriptstyle(\Mat_{\scalebox{.8}{$\scriptscriptstyle R$}},R)}(V) = \Mod{R}([V,R],R)$$
is the space of all continuous $R$-linear maps $\mu:[V,R] \rightarrow R$, where $[V,R]$ is the space $\uV(V,R)$ of all continuous maps $f:V \rightarrow R$.  When $V$ is a locally compact Hausdorff topological space (considered as an object of $\V$), we deduce by \ref{par:conv_sp_rad_meas} that $(\Mat_R,R)$-distributions $\mu \in D_{\scriptscriptstyle(\Mat_{\scalebox{.8}{$\scriptscriptstyle R$}},R)}(V)$ on $V$ may be identified with compactly supported $R$-valued Radon measures on $V$.
\end{ExaSub}

\begin{ExaSub}[\textbf{Schwartz distributions of compact support, I}]\label{exa:cs_schw_distns_i}
Let $\V$ be either the category $\Fro$ of Fr\"olicher's smooth spaces, or the category $\Diff$ of diffeological spaces \pref{par:fro_diff}.  The real numbers $\RR$ constitute a commutative ring object in $\V$, so we can consider the scalar $\RR$-linear context $(\V,\Mat_\RR,\RR)$ in $\V$ (\ref{exa:cls_exa_ffa_ctxts}, \ref{par:sc_rlin_ctxt}), which is balanced \pref{thm:th_lrmods_sat_bal_iff_rcomm}.  For each object $V$ of $\V$, $D_{\scriptscriptstyle(\Mat_{\scalebox{.8}{$\scriptscriptstyle \RR$}},\RR)}(V)$ is the space of all smooth $\RR$-linear maps $\mu:[V,\RR] \rightarrow \RR$, where $[V,\RR]$ is the space $\uV(V,\RR)$ of all smooth functions $f:V \rightarrow \RR$.  Now let $V$ be a smooth manifold.  In the case that $\V = \Fro$, it follows from \cite[5.1]{FroKr} and \cite[Thm. 6]{Fro:CccAnSmthMaps} that the $(\Mat_\RR,\RR)$-distributions $\mu \in D_{\scriptscriptstyle(\Mat_{\scalebox{.8}{$\scriptscriptstyle \RR$}},\RR)}(V)$ on $V$ are precisely the Schwartz distributions of compact support on $V$.  But since the embedding $\Fro \hookrightarrow \Diff$ preserves exponentials \pref{par:fro_diff}, it follows that the same statement holds in the case where $\V = \Diff$.
\end{ExaSub}

\begin{ExaSub}[\textbf{Schwartz distributions of compact support, II}]\label{exa:cs_schw_distns_ii}
Let $\V$ be the Cahiers topos, or, more generally, the topos $\Shv(\A)$ of sheaves on any product-closed $C^\infty$-site $\A$ \pref{par:cinfty_rings}.  The $C^\infty$-ring $C^\infty(\RR)$ represents a commutative ring object $R$ in $\V$, so we can consider the scalar $R$-linear context $(\V,\Mat_R,R)$ in $\V$ (\ref{exa:cls_exa_ffa_ctxts}, \ref{par:sc_rlin_ctxt}), which is balanced \pref{thm:th_lrmods_sat_bal_iff_rcomm}.  As in \ref{exa:cs_rad_meas} and \ref{exa:cs_schw_distns_i}, $(\Mat_R,R)$-distributions on an object $V$ of $\V$ are precisely $R$-linear morphisms $[V,R] \rightarrow R$, where $[V,R]$ is $\uV(V,R)$ with the obvious $R$-module structure.  Now supposing that $V$ is a smooth manifold, considered as an object of $\V$, one can employ the argument given in \cite[II.3.6]{MoeRey} to show that $(\Mat_R,R)$-distributions on $V$ are therefore in bijective correspondence with compactly supported Schwartz distributions on $M$.  Indeed, there it is proved that this holds in the case where $\V$ is the presheaf category on $\CinftyRing^\op_{\textnormal{fg}}$, and a close scrutiny of the proof therein shows that the same argument works when $\V = [\A^\op,\Set]$, but since finite limits and exponentials in the sheaf category are formed as in the presheaf category, the more general case follows.
\end{ExaSub}

\begin{ExaSub}[\textbf{Non-negative measures of compact support}]\label{exa:nonneg_cs_meas}
As a variation on \ref{exa:cs_rad_meas}, we can take $\V = \Conv$ and consider the scalar $\RR_+$-linear context $(\V,\Mat_{\RR_+},\RR_+)$ in $\V$ \pref{par:sc_rlin_ctxt}, where $\RR_+$ is the space of non-negative real numbers, considered as a rig object in $\V$.  For each convergence space $V$,
$$D_{\scriptscriptstyle(\Mat_{\scalebox{.8}{$\scriptscriptstyle \RR_+$}}\kern-0.3ex,\,\RR_+)}(V) = \Mod{\RR_+}([V,\RR_+],\RR_+)$$
is therefore the space of all continuous $\RR_+$-linear maps $\mu:[V,\RR_+] \rightarrow \RR_+$, where $[V,\RR_+] = \uV(V,\RR_+)$ is the space of all continuous $\RR_+$-valued maps on $V$.  But it is straightforward to show that such a map $\mu$ extends uniquely to a continuous $\RR$-linear map $\hat{\mu}:[V,\RR] \rightarrow \RR$, given by $\hat{\mu}(f) = \mu(f_+) - \mu(f_-)$ where $f_+$ is defined as the pointwise supremum of $f$ and $0$ and $f_-$ is defined as $(-f)_+$, so that $f = f_+ - f_-$.  This describes a bijection between $(\Mat_{\RR_+},\RR_+)$-distributions on $V$ and continuous $\RR$-linear maps $\mu:[V,\RR] \rightarrow \RR$ with the property that $\mu(f) \gt 0$ whenever $f \gt 0$.  When $V$ is a locally compact Hausdorff topological space, we can therefore deduce by \ref{exa:cs_rad_meas} that $(\Mat_{\RR_+},\RR_+)$-distributions $\mu \in D_{\scriptscriptstyle(\Mat_{\scalebox{.8}{$\scriptscriptstyle \RR_+$}}\kern-0.3ex,\,\RR_+)}(V)$ are equivalently described as non-negative compactly supported Radon measures on $V$.
\end{ExaSub}

\begin{ExaSub}[\textbf{Radon probability measures of compact support}]\label{exa:rad_prob_meas_cpct_supp} 
Taking $\V = \Conv$ as in \ref{exa:cs_rad_meas} and \ref{exa:nonneg_cs_meas}, the space $\RR$ of real numbers carries the structure of a preordered ring in $\V$, where $\RR_+$ is the space of all non-negative reals.  Hence we can consider the positive $\RR$-convex context $(\V,\T,S) = (\V,\Mat^\aff_{\RR_+},\RR_+)$ in $\V$ \pref{exa:cls_exa_ffa_ctxts}.  Here $\T$-algebras are $\RR$-convex spaces in $\V$, which we call \textit{convergence convex spaces} \pref{exa:conv_cvx_sp}.  We prove in \ref{thm:commutant_thm_for_concrete_cats} and \ref{par:exa_conv_convex_spaces_cmtnt} that the commutant $\T^\perp$ of $\T$ with respect to $\RR_+$ is the theory of \textit{pointed $\RR_+$-modules} in $\V$ \pref{par:th_pted_right_rmods}, i.e. $\Alg{\T^\perp}^!$ is the $\V$-category $\Mod{\RR_+}^*$ whose objects are $\RR_+$-modules $M$ in $\V$ equipped with a chosen element $* \in M$.  When $\RR_+$ itself is considered as a pointed $\RR_+$-module, its designated element $*$ is the identity element $1 \in \RR_+$.  For each convergence space $V$, the cotensor $[V,\RR_+]$ in $\Mod{\RR_+}^*$ is the space $\uV(V,\RR_+)$ of all non-negative continuous real-valued functions on $V$, with the pointwise $\RR_+$-module structure and designated element $* = 1:V \rightarrow \RR_+$ the constant map with value $1$.  Therefore, 
$$D_{\scriptscriptstyle(\Mat_{\scalebox{.8}{$\scriptscriptstyle \RR_+$}}^{\scalebox{.9}{$\scriptscriptstyle\aff$}}\kern-0.3ex,\,\RR_+)}(V) = \Mod{\RR_+}^*([V,\RR_+],\RR_+)$$
is the space of all continuous maps $\mu:[V,\RR_+] \rightarrow \RR_+$ that are $\RR_+$-linear and send $1$ to $1$.  When $V$ is a locally compact Hausdorff topological space, we thus deduce by the preceding example \ref{exa:nonneg_cs_meas} that $(\Mat_{\RR_+}^\aff,\RR_+)$-distributions $\mu \in D_{\scriptscriptstyle(\Mat_{\scalebox{.8}{$\scriptscriptstyle \RR_+$}}^{\scalebox{.9}{$\scriptscriptstyle\aff$}}\kern-0.3ex,\,\RR_+)}(V)$ are equivalently described as compactly supported Radon \textit{probability} measures $\mu$ on $V$ (since Radon probability measures are precisely those non-negative real-valued Radon measures $\mu$ with the property that $\int 1 \;d\mu = 1$).
\end{ExaSub}

\begin{ExaSub}[\textbf{The filter monad}]\label{exa:filt_mnd}
The two-element set $2 = \{0,1\}$ is a distributive lattice and hence carries two associated rig structures, depending on which of the Boolean operations $\wedge,\vee$ we take as addition.  Let us now view $2$ as a commutative rig by taking $\wedge$ as the addition and $\vee$ as the multiplication.  Hence we can consider the scalar $2$-linear context $(\V,\T,S) = (\Set,\Mat_2,2)$ in $\V = \Set$ (\ref{exa:cls_exa_ffa_ctxts}, \ref{par:sc_rlin_ctxt}), which is balanced \pref{thm:th_lrmods_sat_bal_iff_rcomm}, so that $\T^\perp \cong \T$.  Modules for the commutative rig $2$ are the same as meet semilattices (with top element) \cite[2.10]{Lu:CvxAffCmt}, so the category of normal $\T$-algebras may be identified with the category $\SLat_{\wedge\top}$ of meet semilattices \cite[2.10]{Lu:CvxAffCmt}.  Hence for each set $V$
$$D_{\scriptscriptstyle(\Mat_{\scalebox{.8}{$\scriptscriptstyle 2$}},2)}(V) = \SLat_{\wedge\top}([V,2],2)$$
is the set of all homomorphisms of meet semilattices $\F:[V,2] \rightarrow 2$, where $[V,2]$ is $\Set(V,2)$ with its pointwise meet semilattice structure.  We may identify $[V,2]$ with the powerset $\sP(V)$, under the inclusion order, and it follows that $(\Mat_2,2)$-distributions $\F \in D_{\scriptscriptstyle(\Mat_{\scalebox{.8}{$\scriptscriptstyle 2$}},2)}(V)$ are equivalently described as \textit{filters} on the set $V$ \pref{par:slat_filt}.
\end{ExaSub}

\begin{ExaSub}[\textbf{The proper filter monad}]\label{exa:prop_filt_mnd}
Again viewing the two-element set $2$ as a rig, as in \ref{exa:filt_mnd}, let us now consider the scalar 2-affine context $(\Set,\Mat_2^\aff,2)$ in $\V = \Set$ (\ref{exa:cls_exa_ffa_ctxts}, \ref{par:sc_raff_ctxt}).  The category of normal $\T$-algebras for the theory $\T = \Mat_2^\aff$ is the category of $2$-affine spaces, which may be identified with the category of binary-meet semilattices $\SLat_\wedge$ \pref{par:slat_filt}, by \cite[3.3]{Lu:CvxAffCmt} (and remarks in \ref{par:slat_filt}).  Taking the commutant $\T^\perp$ of $\T$ with respect to the binary-meet semilattice $2$, the category of normal $\T^\perp$-algebras may be identified with the category $\SLat_{\scriptscriptstyle \wedge\top\bot}$ of meet semilattices with a bottom element\footnote{Here we also use the fact that $\SLat_{\scriptscriptstyle \wedge\top\bot}$ is isomorphic to the category of \textit{join} semilattices with \textit{top} element.} \cite[8.2]{Lu:CvxAffCmt}.  Hence, for each set $V$
$$D_{\scriptscriptstyle(\Mat_{\scalebox{.8}{$\scriptscriptstyle 2$}}^{\scalebox{.9}{$\scriptscriptstyle\aff$}}\kern-0.3ex,\,2)}(V) = \SLat_{\scriptscriptstyle \wedge\top\bot}([V,2],2)$$
is the set of all mappings $\F:[V,2] \rightarrow 2$ that preserve binary meets and preserve both the top element and the bottom element, where $[V,2] = \Set(V,2) \cong \sP(V)$ is the powerset of $V$.  Thus we deduce that $(\Mat_2^\aff,2)$-distributions $\F \in D_{\scriptscriptstyle(\Mat_{\scalebox{.8}{$\scriptscriptstyle 2$}}^{\scalebox{.9}{$\scriptscriptstyle\aff$}}\kern-0.3ex,\,2)}(V)$ are equivalently described as \textit{proper filters} on the set $X$ \pref{par:slat_filt}.
\end{ExaSub}

\begin{ExaSub}[\textbf{The ultrafilter monad}]\label{exa:uf_mnd}
Let $\V = \Set$, and take $\T$ to be the initial Lawvere theory $\FinCard^\op$ (\ref{par:jth}, \ref{def:th}).  The category of normal $\T$-algebras is isomorphic to $\Set$ itself (e.g., by \cite[4.2]{Lu:Cmtnts}).  The set $S := 2 = \{0,1\}$ corresponds to a $\T$-algebra $\T \rightarrow \Set$ and so determines a morphism of theories $\underline{2}:\T \rightarrow \Set_{2}$ \pref{par:full_th}, and since $\T = \FinCard^\op$ is the initial Lawvere theory, this is the unique such morphism of theories.  Hence by \cite[5.4]{Lu:CvxAffCmt}, the morphism $\underline{2}$ is \textit{central}, equivalently, its commutant is the full theory $\Set_2$.  In other words, the commutant $\T^\perp$ of $\T$ with respect to $2$ is $\Set_2$.  Hence the category of normal $\T^\perp$-algebras may be identified with the category $\Bool$ of Boolean algebras \cite[2.12]{Lu:CvxAffCmt}.  By \ref{par:cmtnts_th}, the set $2$ carries the structure of a normal $\T^\perp$-algebra, and the corresponding Boolean algebra structure on $2$ is the usual one (by \cite[2.12]{Lu:CvxAffCmt}).  The commutant of $\T^\perp$ with respect to the $\T^\perp$-algebra $2$ is precisely the subtheory $\T^{\perp\perp} \hookrightarrow \Set_2$ in which $\T^{\perp\perp}(n,m) = \Bool(2^n,2^m)$ is the set of all Boolean algebra homomorphisms $h:2^n \rightarrow 2^m$ $(n,m \in \NN)$.  Such a homomorphism $h$ is given by a family of homomorphisms of Boolean algebras $h_j:2^n \rightarrow 2$ indexed by the elements $j$ of the cardinal $m$, but by \ref{par:slat_filt} we know that for each $j$ there is a unique element $k(j) \in n$ such that $h_j$ is the $k(j)$-th projection $\pi_{k(j)}:2^n \rightarrow 2$.  It follows that there is a unique mapping $k:m \rightarrow n$ such that $h = 2^k:2^n \rightarrow 2^m$, showing that the unique morphism of Lawvere theories $\FinCard^\op \rightarrow \T^{\perp\perp}$ is fully faithful and hence is an isomorphism  of Lawvere theories $\T = \FinCard^\op \cong \T^{\perp\perp}$.  Since this isomorphism commutes with the associated morphisms to $\Set_2$, we deduce that $\FinCard^\op$ is saturated with respect to $2$ \pref{par:sat_bal}.  Also, $\FinCard^\op$ is commutative, by \cite[5.4]{Lu:CvxAffCmt}, so $(\Set,\FinCard^\op,2)$ is a finitary functional-analytic context.  For each set $V$, we now deduce that
$$D_{\scriptscriptstyle(\FinCard^\op\kern-0.3ex,\,2)}(V) = \Bool([V,2],2)$$
is the set of all homomorphisms of Boolean algebras $\U:[V,2] \rightarrow 2$ where $[V,2] = \Set(V,2)$ with the pointwise boolean algebra structure.  Identifying $[V,2]$ with the powerset $\sP(V)$ with the usual Boolean operations, we therefore deduce that $(\FinCard^\op,2)$-distributions $\U \in D_{\scriptscriptstyle(\FinCard^\op\kern-0.3ex,\,2)}(V)$ are equivalently described as \textit{ultrafilters} on the set $V$ \pref{par:slat_filt}.
\end{ExaSub}

\section{Further fundamentals of enriched algebra}\label{sec:further_fun}

In order to establish the general classes of examples of functional-analytic contexts described in \ref{exa:cls_exa_ffa_ctxts}, as well as the specific examples of functional distribution monads in \S\ref{sec:exa_cdistn_mnds}, we will need to develop further fundamental aspects of enriched algebra.

\subsection{Modules in \texorpdfstring{$\J$}{J}-algebraic symmetric monoidal closed \texorpdfstring{$\V$}{V}-categories}\label{sec:mod_jalg_smcvcats}

Our study of enriched-categorical aspects of modules and affine spaces over rigs in cartesian closed categories will be enabled by the following general result.  Let $j:\J \hookrightarrow \uV$ be an eleutheric system of arities \pref{par:jary_mnds}.

\begin{PropSubSub}\label{thm:modules_in_talg}
Let $\T$ be a commutative $\J$-theory, and suppose that the category of normal $\T$-algebras $\W = \Alg{\T}^!_0$ has reflexive coequalizers.  Let $R$ be a monoid in the closed symmetric monoidal category $\W$ \pref{thm:talgs_comm_th_smcvcat}.  Then the category of left $R$-modules in $\W$ underlies a strictly $\J$-algebraic $\V$-category over $\uV$.
\end{PropSubSub}
\begin{proof}
Let $\bar{\W} := \Alg{\T}^!$.  Since $\bar{\W}$ is a symmetric monoidal closed $\V$-category \pref{thm:talgs_comm_th_smcvcat} and $R$ is a monoid in its underlying monoidal category $\W$, the $\V$-functor $R \otimes (-):\bar{\W} \rightarrow \bar{\W}$ underlies a $\V$-monad $\SSS$.  The category of Eilenberg-Moore algebras of the ordinary monad underlying $\SSS$ is precisely the category of left $R$-modules, which is therefore the ordinary category underlying the $\V$-category $\bar{\W}^\SSS$ of $\SSS$-algebras.  Let $\Mod{R} := \bar{\W}^\SSS$, and let $P \dashv Q:\Mod{R} \rightarrow \bar{\W}$ denote the Eilenberg-Moore $\V$-adjunction for $\SSS$.  Since $\bar{\W}$ is a symmetric monoidal closed $\V$-category, the $\V$-endofunctor $S = R \otimes (-):\bar{\W} \rightarrow \bar{\W}$ has a right adjoint and hence preserves (conical) coequalizers, so the Eilenberg-Moore forgetful $\V$-functor $Q:\bar{\W}^\SSS \rightarrow \bar{\W}$ creates coequalizers.  Also, the strictly $\J$-algebraic $\V$-functor $G:\bar{\W} = \Alg{\T}^! \rightarrow \uV$ is strictly $\V$-monadic \pref{par:ntalgs_emalgs} and so creates coequalizers of $G$-contractible pairs \cite[II.2.1]{Dub}, so it follows that the composite $\V$-functor
$$U := \left(\Mod{R} \xrightarrow{Q} \bar{\W} \xrightarrow{G} \uV\right)$$
creates coequalizers of $U$-contractible pairs and hence is strictly $\V$-monadic by \cite[II.2.1]{Dub}.  $G$ has a left adjoint $F$, so a left adjoint to $U$ is obtained as the composite $L := PF$.  The $\V$-monad on $\uV$ induced by the $\V$-adjunction $L \dashv U$ is $UL = GSF$, and since $G$ is $\J$-algebraic, $G$ conditionally preserves $\J$-flat colimits \cite[12.2]{Lu:EnrAlgTh}, so since $S$ and $F$ preserve all colimits it follows that the $\V$-monad $UL$ conditionally preserves $\J$-flat colimits and hence is a $\J$-ary monad \pref{par:jary_mnds}.  Therefore $U$ is strictly $\J$-algebraic, by \ref{par:ntalgs_emalgs}.
\end{proof}

\subsection{Coslices of \texorpdfstring{$\V$}{V}-categories of algebras}\label{sec:coslices}

Letting $\V$ be a cartesian closed category with equalizers and countable colimits, we show in the present section that if $E$ is an object of a discretely finitary algebraic $\V$-category $\A$ over $\uV$, then the \textit{coslice} category $E \slash \A$ is a discretely finitary algebraic $\V$-category over $\uV$.  Let us begin by recalling some well-known facts, all of which are easily proved.

\begin{ParSubSub}\label{par:coslice}
Given an arbitrary category $\A$ with finite coproducts, any object $E$ of $\A$ carries the structure of a monoid in the cocartesian monoidal category $\A$.  The category $\Mod{E}$ of left modules for the monoid $E$ in $\A$ is isomorphic to the coslice category $E \slash \A$ under the object $E$ of $\A$, whose objects are pairs $(A,a)$ with $A \in \ob\A$ and $a:E \rightarrow A$, and we shall freely identify these categories.  The monoid $E$ determines a monad $\TT_E$ on $\A$ whose underlying endofunctor is
$$T_E = E + (-)\;:\;\A \rightarrow \A\;,$$
and it is immediate from the definitions that $\Mod{E}$ is precisely the category of Eilenberg-Moore algebras $\A^{\TT_E}$ of $\TT_E$.
\end{ParSubSub}

\begin{ParSubSub}\label{par:coslice_vcat}
Now let us instead assume that $\A$ is a cotensored $\V$-category with (conical) finite coproducts and $E$ is an object of $\A$.  By \ref{par:coslice} we have an associated monad $\TT_E$ on the underlying ordinary category $\A_0$, and $\TT_E$ clearly underlies a $\V$-monad on $\A$, which we again denote by $\TT_E$.  Hence by \ref{par:coslice} the coslice category $E \slash \A_0$ underlies a cotensored $\V$-category $\A^{\TT_E}$ that we denote by $E \slash \A$.  Given objects $(A,a)$ and $(B,b)$ of $E \slash \A$, the definition of the Eilenberg-Moore $\V$-category $\uV^{\TT_E}$ \cite[\S II.1]{Dub} gives us an expression for the hom-object $(E \slash \A)((A,a),(B,b))$ as an equalizer of a pair of morphisms expressed in terms of the $\TT_E$-algebra structures $(a,1):E + A \rightarrow A$ and $(b,1):E + B \rightarrow B$ carried by $A$ and $B$, and it follows readily that \textit{the hom-object $(E \slash \A)((A,a),(B,b))$ is the equalizer of the following pair:}
\begin{equation}\label{eq:hom_in_coslice_as_equalizer}
\xymatrix{
\A(A,B) \ar@<1.3ex>[rr]^{\A(a,B)} \ar[r]_(.6){!} & 1 \ar[r]_(.4){[b]} & \A(E,B).
}
\end{equation}

The forgetful $\V$-functor $U:E \slash \A \rightarrow \A$ is not only $\V$-monadic but has a special further property, as follows:
\end{ParSubSub}

\begin{PropSubSub}\label{thm:forg_vfunc_on_coslice_cr_conical_coeqs}
The $\V$-monadic $\V$-functor $U:E \slash \A \rightarrow \A$ creates conical\linebreak coequalizers.
\end{PropSubSub}
\begin{proof}
It is well-known that the underlying ordinary functor $U_0:E \slash \A_0 \rightarrow \A_0$ creates coequalizers \cite[17.3]{Wy:QTop}.  But since $\A$ and $E \slash \A$ are cotensored $\V$-categories, it follows that conical coequalizers in these $\V$-categories are the same as coequalizers in the underlying ordinary categories (e.g. by \cite[\S 3.8]{Ke:Ba}).
\end{proof}

\begin{CorSubSub}\label{thm:coslice_vmonadic}
Let $G:\A \rightarrow \C$ be a strictly $\V$-monadic $\V$-functor, and let $E$ be an object of $\A$.  Assume that $\A$ is cotensored and has (conical) finite coproducts.  Then the composite $\V$-functor $GU:E \slash \A \rightarrow \C$ is strictly $\V$-monadic.
\end{CorSubSub}
\begin{proof}
$GU$ has a left adjoint.  By Beck's monadicity theorem, formulated in the enriched context by Dubuc \cite[II.2.1]{Dub}, it therefore suffices to show that $GU$ creates conical coequalizers of $GU$-contractible pairs.  But by the same theorem we know that $G$ creates conical coequalizers of $G$-contractible pairs, so since $U$ creates arbitrary conical coequalizers \pref{thm:forg_vfunc_on_coslice_cr_conical_coeqs} the result follows.
\end{proof}

\begin{ThmSubSub}\label{thm:coslice_of_str_nvalg_cat_nvalg}
Let $\A$ be a discretely finitary algebraic $\V$-category over $\uV$, and let $E$ be an object of $\A$.  Then the coslice $\V$-category $E \slash \A$ is discretely finitary algebraic over $\uV$.
\end{ThmSubSub}
\begin{proof}
By \ref{par:ntalgs_emalgs} we know that the associated $\V$-adjunction $F \dashv G:\A \rightarrow \uV$ is strictly $\V$-monadic, so since $\A$ is cotensored and has finite coproducts \pref{thm:con_colims_algs} we can invoke \ref{thm:coslice_vmonadic} to deduce that the composite $\V$-functor $GU:E \slash \A \rightarrow \uV$ is strictly $\V$-monadic.  Letting $L$ denote the left adjoint to the strictly $\V$-monadic $\V$-functor $U:E \slash \A \rightarrow \A$ \pref{par:coslice_vcat}, recall that the $\V$-monad induced by the $\V$-adjunction $L \dashv U$ is $UL = E + (-):\A \rightarrow \A$ \pref{par:coslice_vcat}.  We know that the composite $\V$-adjunction $LF \dashv GU:E \slash \A \rightarrow \uV$ is strictly $\V$-monadic, and so by \ref{par:ntalgs_emalgs}, it suffices to show that the induced $\V$-monad $GULF:\uV \rightarrow \uV$ is a discretely finitary $\V$-monad, equivalently, that $GULF$ preserves objectwise-countable $\DFin_\V$-flat colimits \pref{thm:nvary_monads_via_ctbl_nv_flat_colims}.  But $G$ conditionally preserves $\DFin_\V$-flat colimits, by \cite[6.7]{Lu:EnrAlgTh}, and $\uV$ has objectwise-countable colimits \pref{sec:disc_fin_enr_alg_th}, so $G$ preserves objectwise-countable $\DFin_\V$-flat colimits.  Also, $F$ preserves all colimits, so it suffices to show that the $\V$-functor $UL = E + (-):\A \rightarrow \A$ preserves objectwise-countable $\DFin_\V$-flat colimits.  But $E + (-)$ can be written as a pointwise conical coproduct $\Delta E + 1_\A$ of the constant $\V$-functor $\Delta E:\A \rightarrow \A$ with value $E$ and the identity $\V$-functor $1_\A:\A \rightarrow \A$, so since $1_\A$ preserves colimits it suffices to show that $\Delta E$ preserves objectwise-countable $\DFin_\V$-flat colimits.

Let $\V^\flat$ denote the full sub-$\V$-category of $\uV$ on the objects $V \in \ob\V$ for which a tensor $V \otimes E$ exists in $\A$.  Then we have a $\V$-functor $(-) \otimes E:\V^\flat \rightarrow \A$.  The terminal object $1$ of $\V$ is the unit object of $\V$ and hence lies in $\V^\flat$, so we also have a constant $\V$-functor $\Delta 1:\A \rightarrow \V^\flat$ with value $1$.  The composite 
\begin{equation}\label{eq:comp_1}\xymatrix{\A \ar[rr]^{\Delta 1} & & \V^\flat \ar[rr]^{(-) \otimes E} & & \A}\end{equation}
is isomorphic to $\Delta E$.  The constant $\V$-functor $\Delta 1:\A \rightarrow \uV$ preserves objectwise-countable $\DFin_\V$-flat colimits, since it can be expressed as a composite
$$\xymatrix{\A \ar[rr]^{G} & & \uV \ar[rr]^{(-)^0 \:=\: \Delta 1} & & \uV}$$
of $\V$-functors that both preserve objectwise-countable $\DFin_\V$-flat colimits, as the latter $\V$-functor $(-)^0 \cong \uV(0,-)$ preserves all $\DFin_\V$-flat colimits \pref{par:jary_mnds}.  It follows that $\Delta 1:\A \rightarrow \V^\flat$ preserves objectwise-countable $\DFin_\V$-flat colimits, and, moreover, sends them to colimits in $\V^\flat$ that are preserved by the inclusion $\V^\flat \hookrightarrow \uV$.  Also, $(-) \otimes E:\V^\flat \rightarrow \A$ preserves any such colimit, so the composite \eqref{eq:comp_1} preserves objectwise-countable $\DFin_\V$-flat colimits.
\end{proof}

\subsection{The free enriched theory on a Lawvere theory}\label{sec:free_enr_th}

Let $\V$ be a cartesian closed category with equalizers and countable colimits, and suppose that $\V$ also has intersections of countable families of strong subobjects, so that $\Alg{\T}_\C$ exists for every discretely finitary theory $\T$ and every $\V$-category $\C$ \pref{par:vcat_talgs}.  Let $\kappa$ be an infinite cardinal such that $\V$ has coproducts of families of objects indexed by \textit{$\kappa$-small sets}, i.e., sets of cardinality less than $\kappa$.  For example, we can take $\kappa = \aleph_1$.  The full subcategory $\Set_\kappa \hookrightarrow \Set$ consisting of $\kappa$-small sets is closed under finite products, and the functor
\begin{equation}\label{eq:disc_objs_on_kappa_small_sets}(-) \cdot 1:\Set_\kappa \rightarrow \V\end{equation}
preserves finite products since $\V$ is cartesian closed\footnote{Indeed, given objects $X,Y$ of $\Set_\kappa$, the functor $(-)\times(Y \cdot 1):\V \rightarrow \V$ preserves colimits, so we have isomorphisms $(X \times Y) \cdot 1 \xrightarrow{\sim} X \cdot (Y \cdot 1) \xrightarrow{\sim} (X \cdot 1) \times (Y \cdot 1)$ whose composite is the relevant comparison morphism.  Also $1 \cdot 1 \cong 1$.}.

Letting $\T$ be a $\kappa$-small Lawvere theory (enriched in $\Set$), we can form the free $\V$-category $\T_\V$ on $\T$ \pref{par:enr_cat_th}, which has the same objects as $\T$, with each hom-set $\T_\V(n,m) = \T(n,m) \cdot 1$
obtained as the $\T(n,m)$-fold copower of the terminal object $1$ in $\V$, where $n,m \in \NN$.

\begin{PropSubSub}\label{thm:free_enr_theory}
$\T_\V$ is a theory.  Further,
\begin{enumerate}
\item $\T_\V$ is the free $\V$-enriched theory on $\T$, i.e.
\begin{equation}\label{eq:free_vth_adj}\Th(\T_\V,\U) \;\;\;\cong\;\;\; \Th^{\scriptscriptstyle(\Set)}(\T,\U_0)\end{equation}
naturally in $\U \in \Th$, where $\Th^{\scriptscriptstyle(\Set)}$ denotes the category of Lawvere theories.
\item Let $\C$ be a $\V$-category with conical finite powers.  Then the ordinary category underlying $\Alg{\T_\V}_\C$ is isomorphic to the category of $\T$-algebras in $\C_0$, i.e. \newline $(\Alg{\T_\V}_\C)_0 \cong \Alg{\T}_{\C_0}$.
\item When $\C$ has designated conical finite powers, $(\Alg{\T_\V}^!_\C)_0 \cong \Alg{\T}^!_{\C_0}$.
\end{enumerate}
\end{PropSubSub}
\begin{proof}
We have isomorphisms of categories
\begin{equation}\label{eq:univ_freevcat}\VCAT(\T_\V,\C) \cong \CAT(\T,\C_0)\end{equation}
natural in $\C \in \VCAT$ \cite[\S 2.5]{Ke:Ba}.  In particular, the unit of this representation is an identity-on-objects functor $E:\T \rightarrow (\T_\V)_0$.  Given an object $n$ of $\T$, the $n$-th power projections $\pi_i:n \rightarrow 1$ in $\T$ are thus sent to morphisms $E\pi_i:n \rightarrow 1$ in $\T_\V$.  In order to show that $\T_\V$ is a theory it suffices to show that for each object $m$ of $\T$, the morphisms $\T_\V(m,E\pi_i):\T_\V(m,n) \rightarrow \T_\V(m,1)$ present $\T_\V(m,n)$ as an $n$-th power in $\V$.  But these are simply the morphisms
\begin{equation}\label{eqn:power_cone_1}\T(m,\pi_i) \cdot 1:\T(m,n) \cdot 1\rightarrow \T(m,1) \cdot 1\;.\end{equation}
The functor $(-) \cdot 1$ of \eqref{eq:disc_objs_on_kappa_small_sets} preserves finite powers, so the family \eqref{eqn:power_cone_1} is an $n$-th power cone as needed.

Suppose $\C$ has conical finite powers.  Given a functor $A:\T \rightarrow \C_0$ with corresponding $\V$-functor $A^\sharp:\T_\V \rightarrow \C$, we claim that $A$ is a $\T$-algebra if and only if $A^\sharp$ is a $\T_\V$-algebra.  Indeed, in view of \cite[5.9]{Lu:EnrAlgTh}, since $\T_\V$ and $\C$ have conical finite powers, $A^\sharp$ is a $\T_\V$-algebra as soon as its underlying ordinary functor preserves finite powers of $1$, but the designated $n$-th power projections $\pi_i:n \rightarrow 1$ in $\T$ are sent by $A = A^\sharp E$ to the same morphisms that one obtains by applying $A^\sharp$ to the designated $n$-th power projections $E\pi_i$ in $\T_\V$.  Hence the isomorphism \eqref{eq:univ_freevcat} restricts to yield the isomorphism required for 2, and it is now easy to see that this isomorphism restricts further to yield 3.  Since morphisms of theories are certain normal algebras \pref{par:alg_jth}, the isomorphism in 3 restricts further to yield the isomorphism needed in 1 when we take $\C = \U$, and the naturality in $\U$ follows from the naturality of \eqref{eq:univ_freevcat}.
\end{proof}

Specializing \ref{thm:free_enr_theory} to the case of $\C = \uV$ we obtain the following.

\begin{ThmSubSub}\label{thm:vcat_algs_for_a_lth}
Let $\T$ be a Lawvere theory (enriched in $\Set$), and suppose that $\V$ has coproducts of $(\mor\T)$-indexed families.  Then the category $\Alg{\T}_{\V}$ of all $\T$-algebras in $\V$ underlies a (non-strictly) discretely finitary algebraic $\V$-category over $\uV$, and $\Alg{\T}^!_\V$ underlies a (strictly) discretely finitary algebraic $\V$-category over $\uV$.
\end{ThmSubSub}
\begin{proof}
Since $\V$ has an initial object, it follows that $\V$ has coproducts of families indexed by sets of cardinality less than or equal to that of $\mor\T$.  Hence we can invoke \ref{thm:free_enr_theory} to deduce that $(\Alg{\T_\V})_0 \cong \Alg{\T}_\V$ and $(\Alg{\T_\V}^!)_0 \cong \Alg{\T}_\V^!$. 
\end{proof}

\begin{ExaSubSub}\label{exa:th_cmon_ab}
Given a rig $R$ in $\Set$, the category of left $R$-modules is isomorphic to the category of normal $\T$-algebras for a Lawvere theory $\T$, namely the category $\T = \Mat_R$ of $R$-matrices; see \cite[2.8]{Lu:CvxAffCmt}, for example.  In particular, commutative monoids can be described equivalently as left $\NN$-modules and hence as normal $\T$-algebras where $\T = \Mat_\NN$.  Moreover, commutative monoids \textit{in $\V$} are equivalently described as normal $\T$-algebras in $\V$, so we can invoke \ref{thm:vcat_algs_for_a_lth} to deduce the following:
\end{ExaSubSub}

\begin{CorSubSub}\label{thm:comm_mnds_ab_grps_in_v}
The category of commutative monoids in $\V$ underlies a discretely finitary algebraic $\V$-category $\CMon(\V)$ over $\uV$.
\end{CorSubSub}

\begin{PropSubSub}\label{thm:free_vth_on_commth_is_comm}
Suppose that the Lawvere theory $\T$ is commutative.  Then the free $\V$-enriched theory $\T_\V$ on $\T$ is commutative.  Hence, by \ref{thm:vcat_algs_for_a_lth}, \ref{thm:talgs_comm_th_smcvcat}, and \ref{thm:con_colims_algs}, the category of normal $\T$-algebras in $\V$ carries the structure of a symmetric monoidal closed $\V$-category $\Alg{\T_\V}^!$.
\end{PropSubSub}
\begin{proof}
Letting $\sS := \T_\V$, it suffices to show that the identity morphism on $\sS$ factors through the centre $\iota:Z(\sS) \hookrightarrow \sS$ \pref{par:comm}.  In view of the adjunction \eqref{eq:free_vth_adj}, it suffices to show that the unit morphism $E:\T \rightarrow (\T_\V)_0 = \sS_0$ factors through $\iota_0:Z(\sS)_0 \rightarrow \sS_0$.  The identity $1_\sS:\sS \rightarrow \sS$ is a normal $\sS$-algebra in $\sS$, and $Z(\sS)$ is by definition the full theory of $1_\sS$ in $\Alg{\sS}_\sS$ (\ref{par:comm}, \ref{par:cmtnts_th}).  It follows that the underlying Lawvere theory $Z(\sS)_0$ is the full theory of $1_\sS$ in $(\Alg{\sS}_\sS)_0$.  But by \ref{thm:free_enr_theory}, we have an isomorphism $(\Alg{\sS}_{\sS})_0 \cong \Alg{\T}_{\sS_0}$ under which the $\sS$-algebra $1_\sS$ corresponds to the $\T$-algebra $E$, and this isomorphism commutes with the associated functors valued in $\sS_0$.  Hence $Z(\sS)_0$ is isomorphic to the full theory of $E$ in $\Alg{\T}_{\sS_0}$, i.e. $Z(\sS)_0$ is isomorphic to the commutant $\T^\perp_E$ of $E$ \pref{par:cmtnts_th}.  Moreover, considering $Z(\sS)_0$ as a theory over $\sS_0$, via $\iota_0$, we find that $Z(\sS)_0 \cong \T^\perp_E$ as theories over $\sS_0$.  But $\T$ is commutative and hence $E:\T \rightarrow \sS_0$ commutes with itself, by \cite[5.15]{Lu:Cmtnts}, so $E$ factors through its own commutant $\T^\perp_E \hookrightarrow \sS_0$ and hence factors through $\iota_0:Z(\sS)_0 \rightarrow \sS_0$.
\end{proof}

\begin{ExaSubSub}\label{exa:cmonv_abv_comm_nvalg}
The Lawvere theory $\Mat_\NN$ of commutative monoids is commutative; e.g., see \cite[4.6]{Lu:CvxAffCmt}.  Hence the associated $\V$-enriched theory $(\Mat_\NN)_\V$ is commutative, by \ref{thm:free_vth_on_commth_is_comm}.  Therefore, with reference to \ref{thm:comm_mnds_ab_grps_in_v}, we obtain the following:
\end{ExaSubSub}

\begin{CorSubSub}\label{thm:cmonv_abv_smcvcats}
$\CMon(\V)$ is a symmetric monoidal closed $\V$-category.
\end{CorSubSub}

\begin{ParSubSub}[\textbf{Cotensors of $\T_\V$-algebras}]\label{par:cot_tvalgs}
Let $A:\T \rightarrow \V$ be a $\T$-algebra in $\V$.  Under the bijection between $\T$-algebras in $\V$ and $\T_\V$-algebras in $\uV$ \pref{thm:free_enr_theory}, $A$ corresponds to a $\T_\V$-algebra $A^\sharp:\T_\V \rightarrow \uV$.  Given an object $V$ of $\V$, we can form the (pointwise) cotensor $[V,A^\sharp] = \uV(V,A^\sharp-)$ in $\Alg{\T_\V}$ \pref{par:lims_algs}, whose corresponding $\T$-algebra in $\V$ is the functor $\uV(V,A-):\T \rightarrow \V$.
\end{ParSubSub}

\begin{PropSubSub}\label{thm:charn_tvbimorphs}
Let $\G \subseteq \mor\T$ be a generating set of operations for the Lawvere theory $\T$ \pref{par:subth_efth}, and let $A,B,C$ be $\T$-algebras in $\V$, which we can view equivalently as $\T_\V$-algebras in $\uV$ (\ref{thm:free_enr_theory}, \ref{thm:vcat_algs_for_a_lth}).  Then a morphism $f:\ca{A} \times \ca{B} \rightarrow \ca{C}$ in $\V$ is a $\T_\V$-bimorphism \pref{def:bimorph} if and only if the following diagrams commute for every $\omega:n \rightarrow 1$ in $\G$
\begin{equation}\label{eq:diags_char_tvbimorphs}
\xymatrix@!C=10ex @R=5ex{
\ca{A}^n \times \ca{B} \ar[d]_{A\omega \times 1} \ar[r]^{\chi_1} & (\ca{A}\times\ca{B})^n \ar[r]^(.6){f^n} & \ca{C}^n \ar[d]^{C\omega} & \ca{A}\times\ca{B}^n \ar[d]_{1 \times B\omega} \ar[r]^{\chi_2} & (\ca{A}\times\ca{B})^n \ar[r]^(.6){f^n} & \ca{C}^n \ar[d]^{C\omega}\\
\ca{A} \times \ca{B} \ar[rr]_f & & \ca{C} & \ca{A}\times\ca{B} \ar[rr]_f & & \ca{C}
}
\end{equation}
where $\chi_1$ and $\chi_2$ are induced by the families $(\pi_i \times 1)_{i=1}^n$, $(1 \times \pi_i)_{i=1}^n$, respectively, and we have written $\ca{A}^n$ for the $n$-th power $A(n)$ of $\ca{A}$ determined by the $\T$-algebra $A$.
\end{PropSubSub}
\begin{proof}
Identifying $\Alg{\T_\V}_0$ with $\Alg{\T}_\V$, $f$ is a $\T_\V$-bimorphism iff each of its transposes $f_1:\ca{A} \rightarrow \uV(\ca{B},\ca{C})$ and $f_2:\ca{B} \rightarrow \uV(\ca{A},\ca{C})$ is a $\T$-homomorphism, where $\uV(\ca{A},\ca{C})$ and $\uV(\ca{B},\ca{C})$ are regarded as the carriers of the cotensors 
$$\uV(\ca{A},C-),\;\uV(\ca{B},C-)\;:\;\T \rightarrow \V$$
described in \ref{par:cot_tvalgs}.  The characterization of $\T$-homomorphisms given in \ref{par:pres_ops} tells us that $f_1$ and $f_2$ are $\T$-homomorphisms iff certain equations hold for every morphism $\omega:n \rightarrow 1$ in the generating set $\G$.  By way of exponential transposition, these equations are readily shown to be equivalent to the commutativity of the diagrams \eqref{eq:diags_char_tvbimorphs}.
\end{proof}

\begin{CorSubSub}\label{thm:bim_cmons}
Let $\T = \Mat_\NN$ denote the Lawvere theory of commutative monoids, so that $\CMon(\V)$ is the $\V$-category of normal $\T_\V$-algebras \pref{thm:comm_mnds_ab_grps_in_v}.  Given commutative monoids $M,N,P$ in $\V$, a morphism $f:\ca{M} \times \ca{N} \rightarrow \ca{P}$ in $\V$ is a $\T_\V$-bimorphism from $M,N$ to $P$ if and only if $f$ is a bimorphism of commutative monoids \pref{def:rig_ring_module_bim_cmons}.
\end{CorSubSub}
\begin{proof}
Since $\{+:2 \rightarrow 1,\;0:0 \rightarrow 1\} \subseteq \mor\T$ is a generating set of operations for $\T$, this follows directly from \ref{thm:charn_tvbimorphs}.
\end{proof}

\subsection{Modules for rigs in cartesian closed categories}\label{sec:rmods}

As in \ref{sec:free_enr_th}, let $\V$ be a cartesian closed category with equalizers, countable colimits, and intersections of countable families of strong subobjects.

\begin{PropSubSub}\label{thm:rigs_as_monoids_in_cmonv}
Rigs in $\V$ \pref{def:rig_ring_module_bim_cmons} are equivalently defined as monoids in the monoidal category $\CMon(\V)_0$ \pref{thm:cmonv_abv_smcvcats}.  Given a rig $R$ in $\V$, left $R$-modules in $\V$ \pref{def:rig_ring_module_bim_cmons} are equivalently defined as left $R$-modules for the monoid $R$ in the monoidal category $\CMon(\V)_0$.
\end{PropSubSub}
\begin{proof}
This follows from \ref{exa:cmonv_abv_comm_nvalg}, \ref{thm:desc_mon_str_in_talgs}, \ref{thm:bim_cmons}, \ref{def:rig_ring_module_bim_cmons}.
\end{proof}

\begin{CorSubSub}\label{thm:rmod_alg}
Given a rig $R$ in $\V$, the category of left $R$-modules in $\V$ underlies a $\V$-category $\Mod{R}$ that is discretely finitary algebraic over $\uV$.
\end{CorSubSub}
\begin{proof}
This follows from \ref{thm:modules_in_talg} and \ref{thm:rigs_as_monoids_in_cmonv}.
\end{proof}

\begin{RemSubSub}\label{rem:opposite_ring_right_rmods}
In view of \ref{thm:rigs_as_monoids_in_cmonv}, since the monoidal category $\CMon(\V)_0$ is symmetric, we can take the \textbf{opposite} $R^\op$ of any rig $R$ in $\V$.  Left $R^\op$-modules can be described equivalently as \textbf{right $R$-modules}.
\end{RemSubSub}

\begin{ParSubSub}[\textbf{Matrix multiplication for an internal rig}]\label{par:matrix_mult}
Given a rig $R$ in $\V$ and natural numbers $n,m$, we call the power $R^{m \times n}$ the \textbf{object of $m \times n$-matrices} over $R$.  Letting $n,m,\ell \in \NN$, if we are given an object $V$ of $\V$ and morphisms $b:V \rightarrow R^{\ell \times m}$ and $a:V \rightarrow  R^{m \times n}$ in $\V$, then we write $ba:V \rightarrow R^{\ell \times n}$ to denote the morphism defined by the following equations
$$(ba)_{ki} = \sum_{j=1}^m b_{kj}a_{ji}\;:\;V \rightarrow R$$
with $k \in \{1,...,\ell\}$ and $i \in \{1,...,n\}$, where we have employed the notation of \ref{par:notn_fp}, \ref{par:rigs_mods_gen_elts}.  In the case where $V = R^{\ell \times m} \times R^{m \times n}$ and $b,a$ are the product projections, the associated morphism $ba$ will be called the \textbf{matrix multiplication morphism} and written as
\begin{equation}\label{eq:matr_mult_mor}\bullet_{nm\ell}:R^{\ell \times m} \times R^{m \times n} \rightarrow R^{\ell \times n}\;.\end{equation}
\end{ParSubSub}

\begin{ParSubSub}[\textbf{The theory of left $R$-modules}]\label{par:th_lrmods}
Given a rig $R$ in $\V$, we deduce by \ref{thm:rmod_alg} that $\Mod{R}$ is isomorphic to the $\V$-category of normal $\T$-algebras for a theory $\T$.  An elementary computation shows that $R$ itself is the free left $R$-module on the terminal object $1$ of $\V$.  Hence, letting $F \dashv G:\Mod{R} \rightarrow \uV$ denote the associated discretely finitary algebraic $\V$-adjunction, we find that the left adjoint $\V$-functor $F$ sends the conical $n$-th copower $n \cdot 1$ of $1$ in $\uV$ to a conical $n$-th copower $F(n \cdot 1)$ of $R$ in $\Mod{R}$ $(n \in \NN)$.  But $\Mod{R}_0$ has finite biproducts, as one readily verifies, so $F(n \cdot 1)$ is an $n$-th power $R^n$ of $R$ in $\Mod{R}_0$, and $R^n$ is moreover a conical $n$-th power in $\Mod{R}$ since $\Mod{R}$ has conical finite powers \pref{par:lims_algs}.

Hence, by \ref{thm:assoc_th_assoc_alg}, the associated theory $\T$ may be defined by
$$\T(n,m) = \Mod{R}(R^m,R^n)\;\;\;\;\;\;(n,m \in \NN = \ob\T)$$
with composition and identities as in $\Mod{R}^\op$.  Since $R$ is a free $R$-module on one generator \pref{par:free_on_n_gens} and $R^m$ (resp. $R^n$) is a conical $m$-th copower (resp. $n$-th power) in $\Mod{R}$, we have a composite isomorphism
$$\Lambda_{nm} = \left(R^{m \times n} \xrightarrow{\sim} \Mod{R}(R,R)^{m \times n} \xrightarrow{\sim} \Mod{R}(R^m,R^n) = \T(n,m)\right)$$
in $\V$.  Explicitly, the composite $R^{m \times n} \xrightarrow{\Lambda_{nm}} \Mod{R}(R^m,R^n) \hookrightarrow \uV(R^m,R^n)$ is the transpose of the matrix multiplication morphism $\bullet_{nm1}:R^m \times R^{m \times n} \rightarrow R^n$ \pref{par:matrix_mult}, where we have identified $R^m$ with the object of row vectors $R^{1 \times m}$.

Hence there is a unique $\V$-category $\Mat_R$ with hom-objects $\Mat_R(n,m) = R^{m \times n}$ such that the isomorphisms $\Lambda_{nm}$ constitute an identity-on-objects isomorphism of $\V$-categories $\Lambda:\Mat_R \xrightarrow{\sim} \T$.  A routine computation shows that the resulting composition morphisms in $\Mat_R$ are the matrix multiplication morphisms \eqref{eq:matr_mult_mor}, and the identity morphism on $n$ in $\Mat_R$ is the evident \textit{identity matrix} $I_n:1 \rightarrow R^{n \times n}$.  We call $\Mat_R$ the \textbf{$\V$-category of $R$-matrices}.
\end{ParSubSub}

\begin{PropSubSub}\label{thm:th_lrmods}
Given a rig $R$ in $\V$, the $\V$-category of $R$-matrices $\Mat_R$ is a theory, and the $\V$-category of normal $\Mat_R$-algebras is isomorphic to the $\V$-category of left $R$-modules $\Mod{R}$.  Given a left $R$-module $M$, the corresponding normal $\Mat_R$-algebra $\underline{M}$ has the same carrier as $M$, and its structure morphisms
$$\underline{M}_n:\Mat_R(n,1) \times M^n \rightarrow M$$
in $\V$ $(n \in \NN)$ are given by
$$\underline{M}_n(r,m) = \sum_{i = 1}^n r_im_i\;:\;R^n \times M^n \rightarrow M\;\;\;\;\text{(where $(r,m):R^n \times M^n$)}$$
in the notation of \ref{par:notn_fp}, \ref{par:rigs_mods_gen_elts}.  For each $n \in \NN$, the designated $n$-th power cone $(\pi_i:n \rightarrow 1)_{i=1}^n$ in the theory $\Mat_R$ consists of the standard basis row-vectors $b_1,...,b_n \in (\Mat_R)_0(n,1) = \V(1,R^{1\times n}) \cong \V(1,R)^n$, with $(b_i)_j = 1$ if $i = j$ and $(b_i)_j = 0$ otherwise.
\end{PropSubSub}
\begin{proof}
The first two claims are immediate from \ref{par:th_lrmods} since $\Mat_R \cong \T$.  The third claim follows from \ref{thm:assoc_th_assoc_alg}.  In view of the definition of $\T$ given in \ref{par:th_lrmods}, it follows from \ref{rem:power_cones_in_assoc_th} that the designated $n$-th power projections $\pi_i:n \rightarrow 1$ in $\T$ are precisely the morphisms $\iota_i:R^1 = R \rightarrow R^n$ $(i = 1,...,n)$ that present $R^n$ as an $n$-th copower of $R$ in $\Mod{R}$.  Under the isomorphism $\Lambda:\Mat_R \xrightarrow{\sim} \T$, the morphisms $\pi_i = \iota_i$ in $\T$ correspond to the standard basis row-vectors $b_i$.
\end{proof}

\section{The commutant of the theory of left \texorpdfstring{$R$}{R}-modules in \texorpdfstring{$\V$}{V}}\label{sec:cmtnt_th_lrmods}

For the remainder of the paper, let $\V$ be a cartesian closed category with equalizers, countable colimits, and intersections of countable families of strong subobjects.

\begin{ParSub}\label{par:cmtnt_th_lrmods}
Let $R$ be a rig in $\V$.  In \ref{par:th_lrmods} and \ref{thm:th_lrmods} we saw that left $R$-modules in $\V$ are the same as normal $\Mat_R$-algebras, where the theory $\Mat_R$ is the $\V$-category of $R$-matrices.  Since $R$ itself is a left $R$-module, we have an associated normal $\Mat_R$-algebra $R:\Mat_R \rightarrow \uV$, and the corresponding morphism of theories $R:\Mat_R \rightarrow \uV_R$ \pref{par:full_th} presents $\Mat_R$ as a theory over $\uV_R$.  Hence by \ref{par:cmtnts_th} we can take the commutant $\Mat_R^\perp$ of $\Mat_R$ over $\uV_R$, which is a subtheory of $\uV_R$ whose hom-objects are the subobjects
$$\Mat_R^\perp(m,n) = \Alg{\Mat_R}(R^m,R^n) = \Mod{R}(R^m,R^n) \hookrightarrow \uV(R^m,R^n)\;.$$
But we had observed in \ref{par:th_lrmods} that $\Mat_R$ is isomorphic to a theory $\T$ with $\T(n,m) = \Mod{R}(R^m,R^n)$, and we now deduce that in fact $\T^\op = \Mat_R^\perp$, so that
$$\Mat_R^\perp \cong (\Mat_R)^\op$$
as $\V$-categories.   But in fact there is an isomorphism $(\Mat_R)^\op \cong \Mat_{R^\op}$ given by \textit{transposition} of matrices, where $R^\op$ is the opposite of $R$ \pref{rem:opposite_ring_right_rmods}.  Indeed, for all $n,m \in \NN$ we have canonical isomorphisms
$$t_{nm}\;:\;\Mat_{R^\op}(n,m) = R^{m \times n} \rightarrow R^{n \times m} = \Mat_R(m,n),$$
and for each morphism of the form $a:V \rightarrow R^{m \times n}$ in $\V$ we shall write $a^\intercal:V \rightarrow R^{n \times m}$ to denote the composite $t_{nm} \cdot a$.  An easy computation shows that these isomorphisms $t_{nm}$ constitute an identity-on-objects $\V$-functor $t:\Mat_{R^\op} \rightarrow (\Mat_R)^\op$.  Hence
\begin{equation}\label{eq:cmtnt_th_lrmods_th_rrmods}\Mat_R^\perp \cong (\Mat_R)^\op \cong \Mat_{R^\op}\end{equation}
as $\V$-categories.  The composite isomorphism
\begin{equation}\label{eq:lambda_prime}\Theta\;:\;\Mat_{R^\op} \xrightarrow{\sim} \Mat_R^\perp\end{equation}
in \eqref{eq:cmtnt_th_lrmods_th_rrmods} is an identity-on-objects $\V$-functor, and its structure morphisms are the composite isomorphisms
\begin{equation}\label{eq:str_morphs_lambda_prime}\Theta_{nm} = \left(\Mat_{R^\op}(n,m) = R^{m \times n} \xrightarrow{t_{nm}} R^{n \times m} \xrightarrow{\Lambda_{mn}} \Mod{R}(R^n,R^m)\right)\end{equation}
$(n,m \in \NN)$ where $\Lambda_{mn}$ is the isomorphism described in \ref{par:th_lrmods}.  Explicitly, the composite 
$$R^{m \times n} \xrightarrow{\Theta_{nm}} \Mod{R}(R^n,R^m) \hookrightarrow \uV(R^n,R^m)$$
is the exponential transpose of the morphism $R^n \times R^{m \times n} \rightarrow R^m$ denoted by
$$xa^\intercal\;:\;R^n \times R^{m \times n} \rightarrow R^m\;\;\;\;\text{where $(x,a):R^n \times R^{m \times n}$}$$
in the notation of \ref{par:notn_fp} and \ref{par:matrix_mult}, with the identifications $R^n = R^{1 \times n}$ and $R^m = R^{1 \times m}$.

Further, $\Theta$ sends the designated $n$-th power projections $b_i:n \rightarrow 1$ in $\Mat_{R^\op}$ (i.e., the standard basis row-vectors $b_i \in \V(1,R^{1 \times n})$, \ref{thm:th_lrmods}) to the projections $\pi_i:R^n \rightarrow R$, which are equally the designated projections $\pi_i:n \rightarrow 1$ in $\Mat_R^\perp = (\Mod{R})_R$ \pref{par:cmtnts_th}.  Hence $\Theta$ is an isomorphism of theories.

$\Mat_{R^\op}$ is the theory of \textit{right} $R$-modules \pref{rem:opposite_ring_right_rmods}, and the right $R$-module structure carried by $R$ induces a morphism of theories $\Mat_{R^\op} \rightarrow \uV_R$ by means of which $\Mat_{R^\op}$ can be considered as a theory over $\uV_R$.  Further, using \eqref{eq:str_morphs_lambda_prime} and \ref{thm:th_lrmods}, we readily compute that $\Theta$ is an isomorphism of theories over $\uV_R$.  Therefore, $\Mat_{R^\op}$ is the commutant of $\Mat_R$ over $\uV_R$.  But since this result holds for any rig $R$ in $\V$, we can also invoke it with respect the rig $R^\op$ to deduce the following:
\end{ParSub}

\begin{ThmSub}\label{thm:cmtnt_th_lrmods}
Let $R$ be a rig in $\V$.  Then the theory of left $R$-modules $\Mat_R$ and the theory of right $R$-modules $\Mat_{R^\op}$ are commutants of one another over the full theory $\uV_R$ of $R$ in $\uV$.
\end{ThmSub}

\begin{CorSub}\label{thm:th_lrmods_sat_bal_iff_rcomm}
Let $R$ be a rig in $\V$.  Then $\Mat_R$ is saturated with respect to $R$ and hence is a saturated subtheory of the full theory $\uV_R$ of $R$ in $\uV$.  Further, if $R$ is commutative, then the subtheory $\Mat_R \hookrightarrow \uV_R$ is balanced, and $\Mat_R$ is commutative.
\end{CorSub}
\begin{proof}
The first claim follows from \ref{thm:cmtnt_th_lrmods} and \ref{par:sat_bal}.  Further, if $R$ is commutative then $R^\op = R$ and hence it follows from \ref{thm:cmtnt_th_lrmods} that $\Mat_R \hookrightarrow \uV_R$ is balanced, so $\Mat_R$ is commutative by \ref{par:sat_bal}.
\end{proof}

\begin{ParSub}\label{par:sc_rlin_ctxt}
Letting $R$ be a commutative rig $R$ in $\V$, we deduce by \ref{thm:th_lrmods_sat_bal_iff_rcomm} that $(\V,\Mat_R,R)$ is a balanced finitary functional-analytic context \pref{def:fa_ctxt}.  We call $(\V,\Mat_R,R)$ the \textbf{scalar $R$-linear context} in $\V$.
\end{ParSub}

\begin{CorSub}\label{thm:charn_fdistns_in_sc_rlin_ctxt}
Let $R$ be a commutative rig in $\V$, and let $V$ be an object of $\V$.  Then functional distributions on $V$ in the scalar $R$-linear context $(\V,\Mat_R,R)$ are the same as $R$-linear morphisms $\mu:[V,R] \rightarrow R$ in $\V$, and moreover $D_{\scriptscriptstyle(\Mat_{\scalebox{.8}{$\scriptscriptstyle R$}},R)}(V) = \Mod{R}([V,R],R)$.
\end{CorSub}

\section{Affine and convex spaces in cartesian closed categories}\label{sec:aff_cvx_sp}

\begin{ParSub}[\textbf{The affine core of a theory}]\label{par:aff_core}
Let $\T$ be a theory (enriched in $\V$, \ref{def:th}).  Recalling that $n \in \ob\T = \NN$ is an $n$-th power of $1$ in $\T$, let us write $\delta_n:1 \rightarrow n$ to denote the diagonal morphism in $\T$.  For all $n,m \in \NN$, let us denote by $\iota_{nm}:\T^\aff(n,m) \hookrightarrow \T(n,m)$ the equalizer of the morphisms $\phi_{nm} = \T(\delta_n,m):\T(n,m) \rightarrow \T(1,m)$ and $\psi_{nm} = \left(\T(n,m) \xrightarrow{!} 1 \xrightarrow{[\delta_{m}]} \T(1,m)\right)$.
\end{ParSub}

\begin{PropSub}\label{thm:aff_core}
The equalizers $\iota_{nm}$ in \ref{par:aff_core} are the structure morphisms of a subtheory inclusion $\T^\aff \hookrightarrow \T$.
\end{PropSub}
\begin{proof}
It is straightforward to check that the $\iota_{nm}$ satisfy conditions (i), (ii), (iii) in the characterization of subtheories given in \ref{par:subth_efth}.
\end{proof}

\begin{DefSub}\label{def:aff_core}
Given a theory $\T$, we call the subtheory $\T^\aff \hookrightarrow \T$ of \ref{thm:aff_core} the \textbf{affine core} of $\T$.
\end{DefSub}

\begin{ExaSub}[\textbf{The theory of left $R$-affine spaces}]\label{exa:th_lraffsp}
Letting $R$ be a rig in $\V$, recall that the $\V$-category of $R$-matrices $\T = \Mat_R$ is the theory of left $R$-modules \pref{thm:th_lrmods}.  In this example, the diagonal morphisms $\delta_n \in \T_0(1,n) = \V(1,R^{n \times 1}) = \V(1,R^n)$ are the column vectors $(1,....,1):1 \rightarrow R^n$ induced by $n$ copies of the unit $1:1 \rightarrow R$ of $R$.  Consequently, the subobjects $\Mat_R^\aff(n,m) \hookrightarrow \Mat_R(n,m) = R^{m \times n}$ are the equalizers of the morphisms $\phi_{nm},\psi_{nm}:R^{m \times n} \rightarrow R^m$ characterized by the equations
$$
\begin{array}{rclcccr}
(\phi_{nm}(a))_j & = & \displaystyle{\sum_{i = 1}^n a_{ji}} & : & R^{m \times n} \rightarrow R & &\\
(\psi_{nm}(a))_j & = & 1                     & : & R^{m \times n} \rightarrow R & & \text{where $a:R^{m \times n}$}\\
\end{array}
$$
in the notation of \ref{par:notn_fp} and \ref{par:rigs_mods_gen_elts}, with $j = 1,...,m$, so that $\Mat^\aff_R(n,m)$ deserves to be construed as the object of all $R$-matrices in which each row sums to 1.  In particular, let us denote by
$$R^{n,\aff} \hookrightarrow R^n$$
the subobject $\Mat^\aff_R(n,1) \hookrightarrow R^{1 \times n} = R^n$ of row vectors with sum $1$.
\end{ExaSub}

\begin{DefSub}\label{def:raff_sp}
Given a rig $R$ in $\V$, we call normal $\Mat_R^\aff$-algebras \textbf{(left) $R$-affine spaces} (in $\V$), and we define $\Aff{R} = \Alg{\Mat_R^\aff}^!$.  Morphisms in $\Aff{R}$ will be called \textbf{(left) $R$-affine morphisms}.  Given a left $R$-affine space $A$, whose carrier we shall write also as $A$, we denote the associated structure morphisms $R^{n,\aff} \times A^n \rightarrow A$ $(n \in \NN)$ by 
\begin{equation}\label{eq:aff_comb}\sum_{i = 1}^n c_ix_i\;\;\;\;\;\text{where $(c,x):R^{n,\aff} \times A^n$}\end{equation}
in the notation of \ref{par:notn_fp}.  The indicated summation expression is called a \textbf{(left) $R$-affine combination}.
\end{DefSub}

\begin{RemSub}\label{rem:rmod_raff}
Given a rig $R$ in $\V$, the subtheory inclusion $\Mat_R^\aff \hookrightarrow \Mat_R$ induces a $\V$-functor $\Mod{R} = \Alg{\Mat_R}^! \rightarrow \Alg{\Mat_R^\aff}^! = \Aff{R}$, so that every left $R$-module $M$ carries the structure of a left $R$-affine space.
\end{RemSub}

\begin{PropSub}\label{thm:raffsp_comm_alg}
If $R$ is a commutative rig in $\V$, then the theory $\Mat_R^\aff$ of $R$-affine spaces is commutative.
\end{PropSub}
\begin{proof}
By definition $\Mat_R^\aff$ is a subtheory of the theory $\Mat_R$, which is commutative \pref{thm:th_lrmods_sat_bal_iff_rcomm}, so we can apply \cite[5.16]{Lu:Cmtnts}.
\end{proof}

\begin{ParSub}\label{par:preord_ring}
A \textbf{preordered ring} is a ring $R$ equipped with a preorder such that\footnote{Preordered rings can be described equivalently as monoids in a certain symmetric monoidal category of \textit{preordered abelian groups} \cite[3.4]{Lu:CvxAffCmt}.} (i) if $r \lt r'$ and $s \lt s'$ in $R$ then $r + s \lt r' + s'$, (ii) if $0 \lt r$ and $0 \lt s$ in $R$ then $0 \lt rs$, and (iii) $0 \lt 1$.  On the other hand, a preordered ring $R$ can be defined equivalently as a ring $R$ equipped with a given sub-rig $R_+ \hookrightarrow R$ (see \cite[3.4]{Lu:CvxAffCmt}).  The corresponding preorder on $R$ is then given by $r \lt s \Leftrightarrow s - r \in R_+$, so that $R_+ \subseteq R$ then consists of all elements $s \in R$ such that $0 \lt s$.  Either of these two equivalent definitions yields a description of preordered rings as the models of a finite limit sketch (see, e.g., \cite{BaWe}), so we obtain a notion of preordered ring in an arbitrary category with finite limits, such as $\V$.  For simplicity we will employ the second definition:  By definition, a \textbf{preordered ring in} $\V$ is a ring $R$ in $\V$ equipped with a sub-rig $R_+ \hookrightarrow R$, which we call the \textbf{positive part} of $R$.  We shall say that a preordered ring $R$ in $\V$ is \textbf{strong} if the inclusion $R_+ \hookrightarrow R$ is a strong monomorphism in $\V$.
\end{ParSub}

\begin{DefSub}\label{def:rcvx_sp}
Given a preordered ring $R$ in $\V$, a \textbf{(left) $R$-convex space} (in $\V$) is, by definition, a left $R_+$-affine space (in $\V$).  We define $\Cvx{R} := \Aff{R_+} = \Alg{\Mat_{R_+}^\aff}^!$.
\end{DefSub}

For example, when $\V = \Set$ and $R = \RR$ is the ordered ring of real numbers, $\RR$-convex spaces are usually called \text{convex spaces}; see \cite{Lu:CvxAffCmt}.

\begin{ExaSub}[\textbf{Convergence convex spaces}]\label{exa:conv_cvx_sp}
Letting $\V = \Conv$ be the category of convergence spaces, the real numbers constitute a preordered ring $\RR$ in $\Conv$, where $\RR_+ \hookrightarrow \RR$ is the subspace consisting of the non-negative reals.  By way of definition, a \textbf{convergence convex space} is an $\RR$-convex space in $\Conv$.  For example, if $M$ is an $\RR$-module in $\Conv$, i.e. a \textit{convergence vector space} \cite{BeBu}, then $M$ carries the structure of an $\RR_+$-module in $\Conv$, so $M$ carries the structure of a convergence convex space by \ref{rem:rmod_raff}.  Hence any convex subset $A$ of $M$ carries the structure of a convergence convex space.
\end{ExaSub}

\section{Theories of affine and convex spaces in \texorpdfstring{$\V$}{V} as commutants}\label{sec:th_aff_cvx_sp_as_cmtnts}

In the present section we show that the theory $\Mat_R^\aff$ of left $R$-affine spaces for a rig $R$ in $\V$ can be described as a commutant, over $\uV_R$, of the theory of \textit{pointed right $R$-modules}.  We define the latter notion by way of the following more general scheme:

\begin{ParSub}[\textbf{Pointed $\T$-algebras}]\label{par:pted_talgs}
Given a theory $\T$ (enriched in $\V$, \ref{def:th}), a \textbf{pointed normal $\T$-algebra} is, by definition, a pair $(A,*)$ consisting of a normal $\T$-algebra $A$ and a morphism $*:1 \rightarrow \ca{A}$ in $\V$, or equivalently, a morphism $*:F1 \rightarrow A$ where $F1$ denotes the free normal $\T$-algebra on one generator.  Letting $\A = \Alg{\T}^!$, pointed normal $\T$-algebras are therefore the objects of the coslice $\V$-category $F1 \slash \A$ \pref{par:coslice_vcat}, which we write also as $\A^* = (\Alg{\T}^!)^*$.  By \ref{thm:coslice_of_str_nvalg_cat_nvalg}, this $\V$-category $\A^*$ is discretely finitary algebraic over $\uV$, so we may identify $\A^*$ with the $\V$-category of normal $\T^*$-algebras for a theory $\T^*$, i.e.
$$(\Alg{\T}^!)^* = \Alg{\T^*}^!\;.$$

Let $F \dashv G:\A \rightarrow \uV$ and $L \dashv U:\A^* \rightarrow \A$ be the associated $\V$-adjunctions \pref{par:coslice_vcat}, so that the composite $\V$-adjunction $LF \dashv GU:\A^* \rightarrow \uV$ exhibits $\A^*$ as discretely finitary algebraic over $\uV$ \pref{thm:coslice_of_str_nvalg_cat_nvalg}.  The left adjoint $L$ sends each object $A$ of $\A$ to the object $LA = (F1 + A,\iota_1)$ of $\A^*$, where $\iota_1:F1 \rightarrow F1 + A$ is the coproduct injection.  Therefore $LF$ sends each object $V$ of $\uV$ to the object $(F1 + FV,\iota_1)$ of $\A^*$, which we may identify with the object $(F(1 + V),F\iota_1)$, where the latter instance of $\iota_1$ denotes the coproduct injection $\iota_1:1 \rightarrow 1 + V$.

Hence by \ref{rem:assoc_th}, the associated theory $\T^*$ has hom-objects
$$\T^*(n,m) = (GULFn)^m = (GF(1 + n))^m = |F(1 + n)|^m\;.$$
Given any pointed normal $\T$-algebra $(A,*)$, the corresponding normal $\T^*$-algebra has carrier $\ca{A}$, and by using \ref{rem:assoc_th} and \eqref{eq:str_morphs_corr_alg_for_assoc_th} we can show straightforwardly that its associated structural morphisms
\begin{equation}\label{eq:str_morphs_assoc_tstar_alg}\T^*(n,1) = |F(1 + n)| \rightarrow \uV(|A|^n,|A|)\;\;\;\;(n \in \NN)\end{equation}
correspond by exponential transposition to the composites
$$|A|^n \xrightarrow{(* \cdot !,1)} |A| \times |A|^n = |A|^{1 + n} \xrightarrow{\sim} \A(F(1+n),A) \hookrightarrow \uV(|F(1+n)|,|A|),$$
where $*:1 \rightarrow \ca{A}$ is the `point' carried by $A$.
\end{ParSub}

Now let $R$ be an arbitrary rig in $\V$.

\begin{ParSub}[\textbf{The theory of pointed right $R$-modules}]\label{par:th_pted_right_rmods}
Recalling that right $R$-modules are the same as left $R^\op$-modules, we can identify the $\V$-category of right $R$-modules $\A = \Mod{R^\op}$ with the $\V$-category of normal $\Mat_{R^\op}$-algebras, by \ref{thm:th_lrmods}.  A \textbf{pointed right $R$-module} is, by definition, a pointed normal $\Mat_{R^\op}$-algebra \pref{par:pted_talgs}, i.e. a right $R$-module $M$ equipped with a designated morphism $*:1 \rightarrow \ca{M}$ in $\V$, equivalently, a morphism $*:R \rightarrow M$ in $\Mod{R^\op}$ (since $R$ is the free right $R$-module on one generator, \ref{par:th_lrmods}).  By \ref{par:pted_talgs}, the $\V$-category of pointed right $R$-modules $\Mod{R^\op}^* = R \slash \Mod{R^\op}$ may be identified with the $\V$-category of normal $\Mat_{R^\op}^*$-algebras for a theory $\Mat_{R^\op}^* = (\Mat_{R^\op})^*$, where
$$\Mat_{R^\op}^*(n,m) = (R^{1 + n})^m\;.$$
We can consider $R$ itself as a pointed right $R$-module $(R,1)$, where $1:1 \rightarrow R$ is the identity element, so by \ref{par:full_th} the corresponding normal $\Mat_{R^\op}^*$-algebra $R:\Mat_{R^\op}^* \rightarrow \uV$ determines a corresponding morphism of theories
\begin{equation}\label{eq:can_mor_matropstar_to_vr}\Phi^R:\Mat_{R^\op}^* \rightarrow \uV_R\end{equation}
into the full theory $\uV_R$ of $R$ in $\uV$.  Let us denote the structural morphisms of $\Phi^R$ by
\begin{equation}\label{eq:str_morphs_of_morph_to_vr_det_by_r}\Phi^R_n := \Phi^R_{n1}\;:\;\Mat_{R^\op}^*(n,1) = R^{1+n} \longrightarrow \uV(R^n,R)\;\;\;\;\;\;(n \in \NN)\end{equation}
and denote the exponential transpose of $\Phi^R_n$ by
\begin{equation}\label{eq:phi_bar}\bar{\Phi}^R_n:R^n \times R^{1+n} \rightarrow R\;.\end{equation}
Using \ref{par:pted_talgs}, we readily deduce that $\Phi^R_n$ and $\bar{\Phi}^R_n$ are characterized by the equations
\begin{equation}\label{eq:formula_for_phirn}(\Phi^R_n(u))(x) = \bar{\Phi}^R_n(x,u) = u_0 + \sum_{i=1}^n x_iu_i\;\;\;\;\;\;\text{where $(x,u):R^n \times R^{1+n}$}\end{equation}
in the notation of \ref{par:notn_fp} and \ref{par:rigs_mods_gen_elts}; here we consider $R^n$ and $R^{1+n}$ as products $\prod_{i = 1}^n R$ and $\prod_{i = 0}^nR$, respectively, so that $x = (x_1,...,x_n)$ and $u = (u_0,u_1,...,u_n)$ are the first and second projections of the binary product $R^n \times R^{1+n}$.
\end{ParSub}

$R$ itself is a left $R$-affine space and so determines a morphism of theories $\Mat_R^\aff \rightarrow \uV_R$ \pref{par:full_th} by means of which we shall regard $\Mat_R^\aff$ as a theory over $\uV_R$.

\begin{ThmSub}\label{thm:th_raffsp_as_cmtnt}
The theory of left $R$-affine spaces for a rig $R$ in $\V$ is the commutant of the theory of pointed right $R$-modules, where both are considered as theories over the full theory $\uV_R$ of $R$ in $\uV$ \pref{par:full_th}.  Symbolically, $\Mat_R^\aff \cong (\Mat_{R^\op}^*)^\perp$ over $\uV_R$.
\end{ThmSub}
\begin{proof}
The proof is analogous to the argument in \cite[7.2]{Lu:CvxAffCmt} in the $\Set$-based case.  By \ref{thm:cmtnt_th_lrmods} we know that $\Mat_R \cong \Mat_{R^\op}^\perp$ over $\uV_R$, so it suffices to show that
$$(\Mat_{R^\op}^\perp)^\aff = (\Mat_{R^\op}^*)^\perp$$
as theories over $\uV_R$, recalling that the left-hand side denotes the affine core of $\Mat_{R^\op}^\perp$ \pref{def:aff_core}.  Let us denote the category $\Alg{\Mat_{R^\op}}^! = \Mod{R^\op}$ by $\A$.  Recalling that $\Mat_{R^\op}^\perp(n,m) = \A(R^n,R^m)$ for all $n,m \in \NN$, note that $(\Mat_{R^\op}^\perp)^\aff$ is by definition the subtheory of $\Mat_{R^\op}^\perp$ whose hom-objects $(\Mat_{R^\op}^\perp)^\aff(n,m)$ are the equalizers of the pairs
\begin{equation}\label{eq:pair_of_which_the_homs_of_matropperpaff_are_equalizers}
\xymatrix{
\A(R^n,R^m) \ar@<1.3ex>[rr]^{\A(\delta_n,R^m)} \ar[r]_(.6){!} & 1 \ar[r]_(.4){[\delta_m]} & \A(R^1,R^m)
}
\end{equation}
where $\delta_n = (1,...,1):R^1 = R \rightarrow R^n$ is the diagonal morphism.  As in \ref{par:th_pted_right_rmods}, we can consider $R$ itself as an object of $\A^* = R \slash \A$, namely the pair $(R,1_R:R \rightarrow R)$, and the $n$-th power of this object in $\A^*$ is $(R^n,\delta_n)$.  Recall that $(\Mat_{R^\op}^*)^\perp$ is by definition the subtheory of $\uV_R$ whose hom-objects are
$$(\Mat_{R^\op}^*)^\perp(n,m) = \A^*(R^n,R^m) = (R \slash \A)((R^n,\delta_n),(R^m,\delta_m)),$$
but by \ref{par:coslice_vcat} \eqref{eq:hom_in_coslice_as_equalizer} these are precisely the equalizers of the pairs \eqref{eq:pair_of_which_the_homs_of_matropperpaff_are_equalizers} and the result follows.
\end{proof}

\begin{ParSub}\label{par:sc_raff_ctxt}
Let $R$ be a commutative rig $R$ in $\V$.  Then the theory of $R$-affine spaces $\Mat_R^\aff$ is saturated with respect to $R$, by \ref{thm:th_raffsp_as_cmtnt} and \ref{par:sat_bal}.  $\Mat_R^\aff$ is also commutative \pref{thm:raffsp_comm_alg}, so $(\V,\Mat^\aff_R,R)$ is a finitary functional-analytic context, which we call the \textbf{scalar $R$-affine context} in $\V$.
\end{ParSub}

\section{The commutant of the theory of convex spaces in \texorpdfstring{$\V$}{V}}\label{sec:cmtnt_of_th_cvx_sp}

In the previous section, we showed that for a rig $S$ in $\V$, the theory $\Mat_S^\aff$ of left $S$-affine spaces for a rig $S$ in $\V$ is the commutant of the theory $\Mat_{S^\op}^*$ of pointed right $S$-modules, over $\uV_S$.  However, in view of Example \ref{exa:prop_filt_mnd}, it is not in general true that the commutant of $\Mat_S^\aff$ over $\uV_S$ is $\Mat_{S^\op}^*$.  In the present section, we will prove that for a wide class of rigs $S$ in categories $\V$, these theories $\Mat_S^\aff$ and $\Mat_{S^\op}^*$ \textit{are} in fact commutants of one another over $\uV_S$.  In particular, we shall show that this is the case if the rig $S$ is a \textit{ring}.  More generally, we focus on rigs $S$ in $\V$ that arise as the positive parts $S = R_+$ of certain preordered rings $R$ in $\V$ \pref{par:preord_ring}; the case where $S$ is a ring is included as a special case, since here we can take $R = S$ and $R_+ = R$.

Hence in the present section our aim is to establish conditions on a preordered ring $R$ in $\V$ that entail that the theory $\Mat_{R_+}^\aff$ of left $R$-convex spaces and the theory $\Mat_{R_+^\op}^*$ of pointed right $R_+$-modules are commutants of one another over $\uV_{R_+}$.  In the case where $\V = \Set$, a result of this type was established in the author's recent paper \cite{Lu:CvxAffCmt}, so we will begin by reviewing that result, and then we will see how this $\Set$-based result can be used to prove a more general result in the $\V$-based setting.

\begin{DefSub}\label{def:fa_preord_dalg}
Let $R$ be a preordered ring in $\Set$.
\begin{enumerate}
\item $R$ is \textbf{archimedean} if $R$ has the property that for any element $r$ of $R$, if the set $\{nr \mid n \in \NN\}$ has an upper bound in $R$, then $r \lt 0$.
\item (\cite[10.18]{Lu:CvxAffCmt}) $R$ is \textbf{firmly archimedean} if $R$ is archimedean and for every element $s$ of $R_+$ there is some $n \in \NN$ such that $s \lt n$ in $R$, where we have written simply $n$ to denote the sum of $n$ instances of the unit element $1$ of $R$.
\item We denote by $\DD$ the ring of \textbf{dyadic rationals}, i.e. the subring of $\QQ$ consisting of all elements of the form $\frac{p}{2^n}$ with $p \in \ZZ$ and $n \in \NN$, equivalently, the localization of $\ZZ$ with respect to the multiplicative subset $\{2^n \mid n \in \NN\} \subseteq \ZZ$.  We regard $\DD$ as a preordered ring, under the usual total order that $\DD$ inherits from $\QQ$.
\item (\cite[10.14]{Lu:CvxAffCmt}) $R$ is said to be a \textbf{preordered algebra over the dyadic rationals}, or a \textbf{preordered $\DD$-algebra} (in $\Set$), if there exists a monotone homomorphism of rings $\DD \rightarrow R$ (which is necessarily unique if it exists, by the universal property of the localization $\DD$ of $\ZZ$).  Equivalently, $R$ is a preordered $\DD$-algebra if the element $2 = 1 + 1$ of $R$ is invertible and its inverse lies in $R_+$ \cite[10.13]{Lu:CvxAffCmt}.
\end{enumerate}
\end{DefSub}

\begin{ThmSub}[{\cite[10.20, 9.3]{Lu:CvxAffCmt}}]\label{thm:cmtnt_th_in_set}\emptybox
\begin{enumerate}
\item If $R$ is a firmly archimedean preordered $\DD$-algebra (in $\Set$), then the theory $\Mat_{R_+}^\aff$ of left $R$-convex spaces and the theory $\Mat_{R_+^\op}^*$ of pointed right $R_+$-modules are commutants of one another over $\Set_{R_+}$.
\item If $R$ is a ring (in $\Set$), then the theory $\Mat_R^\aff$ of left $R$-affine spaces and the theory $\Mat_{R^\op}^*$ of pointed right $R$-modules are commutants of one another over $\Set_R$.
\end{enumerate}
\end{ThmSub}

\begin{ParSub}\label{par:setup_for_cmtnt_lem}
In order to prove related results in the $\V$-based setting, let us now consider a given preordered ring $R$ in $\V$.  Note that the hom-set $\V(V,R)$ carries the structure of a preordered ring in $\Set$.  Indeed, since preordered rings in a category with finite limits are the models of a finite limit sketch \pref{par:preord_ring}, this follows from the fact that $\V(V,-):\V \rightarrow \Set$ preserves limits.  Explicitly, $\V(V,-)$ sends $R$ to a ring $\V(V,R)$ in $\Set$, and the positive part $\V(V,R)_+ \subseteq \V(V,R)$ consists of all $f \in \V(V,R)$ such that $f$ factors through $R_+ \hookrightarrow R$.

By \ref{thm:th_raffsp_as_cmtnt} and \ref{par:cmtn}, we know that $\Mat_{R_+}^\aff$ and $\Mat_{R_+^\op}^*$ commute over $\uV_{R_+}$, i.e. their associated morphisms to $\uV_{R_+}$ commute, so by \ref{par:cmtn} we deduce that the morphism $\Phi^{R_+}:\Mat_{R_+^\op}^* \rightarrow \uV_{R_+}$ of \eqref{eq:can_mor_matropstar_to_vr} factors through the commutant $(\Mat^\aff_{R_+})^\perp \hookrightarrow \uV_{R_+}$ of $\Mat_{R_+}^\aff$ over $\uV_{R_+}$, via a unique morphism that we shall denote also by
\begin{equation}\label{eq:cmtnt_compn}\Phi^{R_+}:\Mat_{R^\op_+}^* \rightarrow (\Mat_{R_+}^\aff)^\perp\;.\end{equation}
Let us say that $R$ \textbf{has the commutant property} if $\Phi^{R_+}$ is an isomorphism $\Mat_{R^\op_+}^* \rightarrow (\Mat_{R_+}^\aff)^\perp$.  Since $\Phi^{R_+}$ is the unique morphism $\Phi^{R_+}:\Mat_{R^\op_+}^* \rightarrow (\Mat_{R_+}^\aff)^\perp$ in $\Th \slash \uV_{R_+}$, we deduce by \ref{thm:th_raffsp_as_cmtnt} that
\begin{equation}\label{eq:val_cmtnt_th}
\begin{minipage}{3.6in}
\textit{$R$ has the commutant property if and only if $\Mat_{R_+}^\aff$ and $\Mat_{R_+^\op}^*$ are commutants of one another over $\uV_{R_+}$.}
\end{minipage}
\end{equation}
\end{ParSub}

Terminology concerning generators and generating families in categories is somewhat variable; herein, a \textit{generating class} for a category $\C$ is, by definition, a class of objects $\G \subseteq \ob\C$ such that for all $f,g:A \rightarrow B$ in $\C$, if $f \cdot x = g \cdot x$ for all $G \in \G$ and all $x:G \rightarrow A$ in $\C$, then $f = g$.

\begin{PropSub}\label{thm:cmtnt_lem}
Let $R$ be a strong preordered ring in $\V$ \pref{par:preord_ring}, and suppose that $\V$ has a generating class $\G$ such that, for each object $V \in \G$, the preordered ring $\V(V,R)$ in $\Set$ has the commutant property.  Then $R$ has the commutant property.
\end{PropSub}
\begin{proof}
With reference to \ref{par:th_pted_right_rmods}, it suffices (by \ref{par:alg_jth}) to show that for each $n \in \NN$ the structural morphism 
$$\Phi^{R_+}_n := \Phi^{R_+}_{n1}:R_+^{1+n} = \Mat_{R_+^\op}^*(n,1) \rightarrow (\Mat_{R_+}^\aff)^\perp(n,1) = \Aff{R_+}(R_+^n,R_+)$$
is an isomorphism in $\V$.  Considering $R^{1+n}$ as a product $\prod_{i = 0}^n R$, we shall define a morphism $\rho:\Aff{R_+}(R_+^n,R_+) \rightarrow R^{1+n}$ by  the equations
$$
\begin{array}{ccll}
(\rho(\varphi))_0 & = & \varphi(0)                & \text{where $\varphi:\Aff{R_+}(R_+^n,R_+)$}\\
(\rho(\varphi))_i & = & \varphi(b_i) - \varphi(0) & (i = 1,...,n)\\
\end{array}
$$
in the notation of \ref{par:notn_fp} and \ref{par:rigs_mods_gen_elts}, where $b_i:1 \rightarrow R_+^n$ denotes the $i$-th standard basis vector; here we have omitted from our notation all instances of the inclusion $\Aff{R_+}(R_+^n,R_+) \hookrightarrow \uV(R_+^n,R_+)$ and the inclusion $\iota:R_+ \hookrightarrow R$.  We know that $\Phi^{R_+}_n$ is given by the formula \eqref{eq:formula_for_phirn}, and an easy computation shows that the diagram
$$
\xymatrix{
R_+^{1+n} \ar@{^(->}[dr]_{\iota^{1+n}} \ar[r]^(.4){\Phi^{R_+}_n} & \Aff{R_+}(R_+^n,R_+) \ar[d]^\rho\\
& R^{1+n}
}
$$
commutes.  Since $\iota$ is a strong monomorphism, it follows that $\iota^{1+n}$ is a strong monomorphism (e.g., by \cite[6.2(2)]{Lu:EnrFactnSys}), so $\Phi^{R_+}_n$ is a strong monomorphism (e.g., by \cite[6.2(1)]{Lu:EnrFactnSys}).  Therefore, it suffices to show that $\Phi^{R_+}_n$ is an epimorphism.  Let $J:\G \hookrightarrow \V$ denote the inclusion of $\G$ as a full subcategory of $\V$, and let $Y:\V \rightarrow [\G^\op,\Set]$ denote the functor given by $YV = \V(J-,V)$.  Since $\G$ is a generating class, we know that $Y$ is faithful, so $Y$ reflects epimorphisms and hence it suffices to show that $Y(\Phi^{R_+}_n)$ is an epimorphism in $[\G^\op,\Set]$.  To this end, it suffices to let $V \in \G$ and show that $\V(V,\Phi^{R_+}_n):\V(V,R_+^{1+n}) \rightarrow \V(V,\Aff{R_+}(R_+^n,R_+))$ is surjective.

It follows from \cite[4.11(2)]{Lu:Cmtnts} and \ref{def:raff_sp} that the subobject $\Aff{R_+}(R_+^n,R_+) \hookrightarrow \uV(R_+^n,R_+)$ is a \textit{pairwise equalizer} (in the sense of \cite[2.1]{Lu:Cmtnts}) of a family of parallel pairs
$$P_j,Q_j\;:\;\uV(R_+^n,R_+) \rightarrow \uV(R_+^{j,\aff} \times (R_+^n)^j,R_+)\;\;\;\;\;\;(j \in \NN),$$
whose transposes $\bar{P_j},\bar{Q_j}:\uV(R_+^n,R_+) \times R_+^{j,\aff} \times (R_+^n)^j \rightarrow R_+$ are given by the equations
$$
\begin{array}{rcll}
\displaystyle{\bar{P_j}(\varphi,a,x)} & \displaystyle{=} & \displaystyle{\varphi(\sum_{i=1}^j a_ix_i)} &\\
\displaystyle{\bar{Q_j}(\varphi,a,x)} & \displaystyle{=} & \displaystyle{\sum_{i=1}^j a_i\varphi(x_i)} & \text{where $(\varphi,a,x):\uV(R_+^n,R_+) \times R_+^{j,\aff} \times (R_+^n)^j$}\\
\end{array}
$$
in the notation of \ref{par:notn_fp} and \ref{par:rigs_mods_gen_elts}, where we have omitted instances of the inclusion $R_+^{j,\aff} = \Mat_{R_+}^\aff(j,1) \hookrightarrow R_+^j$ \pref{exa:th_lraffsp} and considered $R_+^n$ as a left $R_+$-module.

This pairwise equalizer is sent by $\V(V,-):\V \rightarrow \Set$ to a pairwise equalizer in $\Set$, and by exponential transposition we obtain a characterization of $\V(V,\Aff{R_+}(R_+^n,R_+))$ as a pairwise equalizer
$$\V(V,\Aff{R_+}(R_+^n,R_+)) \rightarrow \V(V,\uV(R_+^n,R_+)) \cong \V(R_+^n \times V, R_+)$$
of a family of parallel pairs of the form
$$p_j,q_j:\V(R_+^n \times V,R_+) \rightarrow \V(R_+^{j,\aff} \times (R_+^n)^j \times V,R_+)\;\;\;\;\;\;(j \in \NN)$$
in $\Set$.  Explicitly, given a morphism $\varphi:R_+^n \times V \rightarrow R_+$ in $\V$, we compute that $p_j$ and $q_j$ send $\varphi$ to the morphisms $p_j(\varphi),q_j(\varphi):R_+^{j,\aff} \times (R_+^n)^j \times V \rightarrow R_+$ given by the equations
\begin{equation}\label{eq:pj_qj}
\begin{array}{rcll}
\displaystyle{(p_j(\varphi))(a,x,v)} & = & \displaystyle{\varphi\left(\sum_{i=1}^j a_ix_i,v\right)} & \\
\displaystyle{(q_j(\varphi))(a,x,v)} & = & \displaystyle{\sum_{i=1}^j a_i\varphi(x_i,v)} & \text{where $(a,x,v):R_+^{j,\aff} \times (R_+^n)^j \times V$}\\
\end{array}
\end{equation}
in the notation of \ref{par:notn_fp} and \ref{par:rigs_mods_gen_elts}.

A morphism $\varphi:R_+^n \times V \rightarrow R_+$ lies in the subobject $\V(V,\Aff{R_+}(R_+^n,R_+)) \hookrightarrow \V(R_+^n \times V,R_+)$ if and only if $p_j(\varphi) = q_j(\varphi):R_+^{j,\aff} \times (R_+^n)^j \times V \rightarrow R_+$ for all $j \in \NN$.  Since $\G$ is a generating class, it is equivalent to require that for each object $W \in \G$ and each morphism $\zeta:W \rightarrow R_+^{j,\aff} \times (R_+^n)^j \times V$, the equation $p_j(\varphi) \cdot \zeta = q_j(\varphi) \cdot \zeta:W \rightarrow R_+$ holds.  But morphisms $\zeta:W \rightarrow R_+^{j,\aff} \times (R_+^n)^j \times V$ are in bijective correspondence with triples $\zeta = (a,x,v)$ with $a \in \V(W,R_+^{j,\aff})$, $x \in \V(W,(R_+^n)^j)$, and $v \in \V(W,V)$.  Recall that $\V(W,R)$ is a preordered ring with positive part $\V(W,R)_+ \cong \V(W,R_+)$ \pref{par:setup_for_cmtnt_lem}, so that since $\V(W,-):\V \rightarrow \Set$ preserves limits and $R_+^{j,\aff}$ is defined as a certain equalizer \pref{exa:th_lraffsp}, it follows that such triples $\zeta = (a,x,v)$ are equivalently described as triples $\zeta = (a,x,v)$ with $a \in \V(W,R_+)^{j,\aff}$, $x \in \V(W,R_+^n)^j$, $v \in \V(W,V)$, and by using the equations in \eqref{eq:pj_qj}, we deduce that the equation $p_j(\varphi) \cdot \zeta = q_j(\varphi) \cdot \zeta$ asserts precisely that
$$\varphi \cdot \left(\sum_{i=1}^j a_ix_i,v\right) = \sum_{i=1}^j a_i(\varphi \cdot (x_i,v)),$$
where the left $\V(W,R_+)$-affine combination $\sum_{i=1}^j a_ix_i$ is taken within the left $\V(W,R_+)$-affine space $\V(W,R_+^n)$, and the left $\V(W,R_+)$-affine combination on the right-hand side is taken within the rig $\V(W,R_+)$ itself.

Hence the subobject $\V(V,\Aff{R_+}(R_+^n,R_+)) \hookrightarrow \V(R_+^n \times V,R_+)$ consists of precisely those morphisms $\varphi:R_+^n \times V \rightarrow R_+$ such that for every $W \in \G$ and every morphism $v:W \rightarrow V$ in $\V$, the associated mapping
\begin{equation}\label{eq:assoc_left_vwrplusaffine_map}\varphi_v := \varphi \cdot (-,v):\V(W,R_+)^n \cong \V(W,R_+^n) \rightarrow \V(W,R_+)\end{equation}
is a left $\V(W,R_+)$-affine map.  But by assumption we know that the preordered ring $\V(W,R)$ has the commutant property when $W \in \G$, so the mapping
\begin{equation}\label{eq:bij_witnessing_that_vwr_val_cmtnt_th}\Phi^{\V(W,R_+)}_n:\V(W,R_+)^{1+n} \rightarrow \Aff{\V(W,R_+)}\bigl(\V(W,R_+)^n,\V(W,R_+)\bigr)\end{equation}
is a bijection.  Therefore, if we now fix a morphism $\varphi:R_+^n \times V \rightarrow R_+$ satisfying the given condition, then for each $v:W \rightarrow V$ with $W \in \G$ we know that the left $\V(W,R)_+$-affine map $\varphi_v$ of \eqref{eq:assoc_left_vwrplusaffine_map} is of the form $\Phi^{\V(W,R_+)}(u^{\varphi_v})$ for a unique element $u^{\varphi_v} \in \V(W,R_+)^{1+n}$; we will use the same notation $u^{\varphi_v}$ for the corresponding morphism $u^{\varphi_v}:W \rightarrow R_+^{1+n}$.  In particular, the identity morphism $1:V \rightarrow V$ determines an associated element $u^{\varphi_1} \in \V(V,R_+^{1+n})$, and it now suffices to show that the composite mapping
\begin{equation}\label{eq:eq1}\V(V,R_+^{1+n}) \xrightarrow{\V(V,\Phi^{R_+}_n)} \V(V,\Aff{R_+}(R_+^n,R_+)) \hookrightarrow \V(R_+^n \times V,R_+)\end{equation}
sends $u^{\varphi_1}$ to $\varphi$, so that the needed surjectivity of $\V(V,\Phi^{R_+}_n)$ is thus obtained.

For each morphism $v:W \rightarrow V$ with $W \in \G$, we can consult \cite[10.1; 10.10($2'$)]{Lu:CvxAffCmt} in order to obtain a description of the element $u^{\varphi_v} \in \V(W,R_+^{1+n}) \cong \V(W,R_+)^{1+n}$; explicitly, $u^{\varphi_v}$ is the $(1+n)$-tuple $(u^{\varphi_v}_0,...,u^{\varphi_v}_n)$ with $u^{\varphi_v}_0 = \varphi_v(0)$ and with $u^{\varphi_v}_i = \varphi_v(b_i) - \varphi_v(0)$ for all $i = 1,...,n$, where $b_i \in \V(W,R_+)^n$ is the $i$-th standard basis vector.  Since the elements $0,b_1,...,b_n$ of $\V(W,R_+^n) \cong \V(W,R_+)^n$ factor through the similarly named elements $0,b_1,...,b_n$ of $\V(1,R_+^n)$, it is straightforward to compute that $u^{\varphi_v}$ can be expressed as the composite
\begin{equation}\label{eq:expr_for_wvarphiv_in_terms_of_wvarphi1_and_v}u^{\varphi_v} = u^{\varphi_1} \cdot v = \left(W \xrightarrow{v} V \xrightarrow{u^{\varphi_1}} R_+^{1+n}\right).\end{equation}

In order to show that the composite \eqref{eq:eq1} sends $u^{\varphi_1}$ to $\varphi$, we must show that $\varphi$ is the composite
\begin{equation}\label{eq:composite}\psi = \left(R_+^n \times V \xrightarrow{1 \times u^{\varphi_1}} R_+^n \times R_+^{1+n} \xrightarrow{\bar{\Phi}^{R_+}_n}  R_+\right),\end{equation}
where $\bar{\Phi}^{R_+}_n$ is defined in \eqref{eq:phi_bar}.  Using the formula for $\bar{\Phi}^{R_+}_n$ given in \eqref{eq:formula_for_phirn}, we find that $\psi$ is characterized by the equation
\begin{equation}\label{eq:charn_lambda}\psi(x,v) = (u^{\varphi_1}(v))_0 + \sum_{i=1}^nx_i(u^{\varphi_1}(v))_i\;\;\;\;\;\text{where $(x,v):R_+^n \times V$}\end{equation}
in the notation of \ref{par:notn_fp} and \ref{par:rigs_mods_gen_elts}.

Since $\G$ is a generating class, it suffices to let $\tau:W \rightarrow R_+^n \times V$ be an arbitrary morphism with $W \in \G$ and show that $\varphi \cdot \tau = \psi \cdot \tau:W \rightarrow R_+$.  But $\tau = (x,v)$ for some $x = (x_1,...,x_n):W \rightarrow R_+^n$ and some $v:W \rightarrow V$.  Hence we can use \eqref{eq:charn_lambda} and \eqref{eq:expr_for_wvarphiv_in_terms_of_wvarphi1_and_v} to deduce that
$$\psi \cdot \tau = u^{\varphi_v}_0 + \sum_{i=1}^nx_iu^{\varphi_v}_i\;:\;W \rightarrow R_+$$
where $x_1,...,x_n,u^{\varphi_v}_0,...,u^{\varphi_v}_n:W \rightarrow R_+$ are considered as elements of the rig $\V(W,R_+)$ and the right-hand side is evaluated therein.  Identifying $u^{\varphi_v} \in \V(W,R_+^{1+n})$ and $x \in \V(W,R_+^n)$ with their corresponding elements of $\V(W,R_+)^{1+n}$ and $\V(W,R_+)^n$, respectively, we thus deduce by \eqref{eq:formula_for_phirn} that the map $\Phi^{\V(W,R_+)}(u^{\varphi_v}):\V(W,R_+)^n \rightarrow \V(W,R_+)$ sends $x$ to $\psi \cdot \tau$, i.e.
$$\psi \cdot \tau = (\Phi^{\V(W,R_+)}(u^{\varphi_v}))(x)\;.$$
But we know that $(\Phi^{\V(W,R_+)}(u^{\varphi_v})) = \varphi_v = \varphi \cdot (-,v)$, so $\psi \cdot \tau = \varphi \cdot (x,v) = \varphi \cdot \tau$ as needed.
\end{proof}

As a first corollary, we obtain the desired result for affine spaces over a ring in $\V$:

\begin{ThmSub}\label{thm:cmtnt_th_raff_sp}
Let $R$ be a ring in $\V$.  Then the theory $\Mat_R^\aff$ of left $R$-affine spaces and the theory $\Mat_{R^\op}^*$ of pointed right $R$-modules are commutants of one another over $\uV_R$.
\end{ThmSub}
\begin{proof}
Consider $R$ as a strong preordered ring in $\V$ with $R_+ = R$, and take $\G = \ob\V$.  Then for each $V \in \G$, the associated preordered ring $\V(V,R)$ has $\V(V,R)_+ = \V(V,R)$ and hence has the commutant property, by \ref{thm:cmtnt_th_in_set}(2) and \eqref{eq:val_cmtnt_th}.  Hence we can invoke \ref{thm:cmtnt_lem} to deduce that $R$ has the commutant property, and the result follows by \eqref{eq:val_cmtnt_th}.
\end{proof}

As another corollary to \ref{thm:cmtnt_lem} and \ref{thm:cmtnt_th_in_set}, we obtain a general result for convex spaces over a preordered ring in $\V$:

\begin{ThmSub}\label{thm:enr_commutant_theorem}
Let $R$ be a strong preordered ring in $\V$, and suppose that $\V$ has a generating class $\G$ such that for each object $V \in \G$, the preordered ring $\V(V,R)$ in $\Set$ is a firmly archimedean preordered $\DD$-algebra.  Then the theory $\Mat_{R_+}^\aff$ of left $R$-convex spaces and the theory $\Mat_{R_+^\op}^*$ of pointed right $R_+$-modules are commutants of each other over $\uV_{R_+}$.
\end{ThmSub}
\begin{proof}
This follows immediately from \ref{thm:cmtnt_lem}, \ref{thm:cmtnt_th_in_set}(1), and \eqref{eq:val_cmtnt_th}.
\end{proof}

\begin{CorSub}\label{thm:pos_rcvx_distn_charn_thm}
Let $R$ be a commutative preordered ring in $\V$, and suppose that the hypotheses of \ref{thm:enr_commutant_theorem} are satisfied.  Then for each object $V$ of $\V$, functional distributions on $V$ in the positive $R$-convex context $(\V,\Mat_{R_+}^\aff,R_+)$ \pref{exa:cls_exa_ffa_ctxts} are the same as homomorphisms of pointed $R_+$-modules $\mu:[V,R_+] \rightarrow R_+$.  Moreover,
$$D_{\scriptscriptstyle(\Mat_{\scalebox{.8}{$\scriptscriptstyle R_+$}}^{\scalebox{.9}{$\scriptscriptstyle\aff$}}\kern-0.3ex,\,R_+)}(V) = \Mod{R_+}^*([V,R_+],R_+)\;.$$
\end{CorSub}

As a corollary to \ref{thm:enr_commutant_theorem}, we obtain the following result in the case where $\V$ is a concrete category over $\Set$:

\begin{CorSub}\label{thm:commutant_thm_for_concrete_cats}
Suppose that the `underlying set' functor $\V(1,-):\V \rightarrow \Set$ is faithful, and let $R$ be a strong preordered ring in $\V$ whose underlying preordered ring $\V(1,R)$ in $\Set$ is a firmly archimedean preordered $\DD$-algebra.  Then the theory $\Mat_{R_+}^\aff$ of left $R$-convex spaces and the theory $\Mat_{R_+^\op}^*$ of pointed right $R_+$-modules are commutants of each other over $\uV_{R_+}$.
\end{CorSub}

\begin{ExaSub}[\textbf{Convergence convex spaces}]\label{par:exa_conv_convex_spaces_cmtnt}
Take $\V = \Conv$ to be the category of convergence spaces \pref{par:conv_sp_rad_meas}, and take $R = \RR$.  The inclusion $\RR_+ \hookrightarrow \RR$ is an embedding, equivalently, a strong monomorphism in $\V$ \cite[11.9]{Wy:QTop}.  The forgetful functor $U \cong \V(1,-):\V \rightarrow \Set$ is faithful and sends $R$ to the firmly archimedean preordered $\DD$-algebra $\RR$, so we can apply \ref{thm:commutant_thm_for_concrete_cats} to deduce that the theory of convergence convex spaces \pref{exa:conv_cvx_sp} and the theory of pointed $\RR_+$-modules in $\V$ are commutants of one another over $\uV_{\RR_+}$.  Hence we can now employ the reasoning in \ref{exa:rad_prob_meas_cpct_supp} to deduce the following:
\end{ExaSub}

\begin{ThmSub}
Let $\V$ be the category $\Conv$ of convergence spaces.  For a locally compact Hausdorff topological space $V$, functional distributions on $V$ in the positive $\RR$-convex context $(\V,\Mat_{\RR_+}^\aff,\RR_+)$ are in bijective correspondence with compactly supported Radon probability measures on $V$.
\end{ThmSub}

Now let us consider an example of a category $\V$ that is not concrete over $\Set$ but where Theorem \ref{thm:enr_commutant_theorem} still applies.

\begin{ExaSub}[\textbf{The line object in the Cahiers topos}]\label{exa:line_obj_cah_top}
Let $\V = \Shv(\C^\op_\text{C})$ be the Cahiers topos, recalling that $\C_\text{C}$ is the full subcategory $\iota:\C_\text{C} \hookrightarrow \CinftyRing$ consisting of the $C^\infty$-rings of the form $C^\infty(M) \otimes_\infty W$ where $M$ is a smooth manifold and $W$ is a Weil algebra \pref{par:cinfty_rings}.  The embedding $i:\Mf \rightarrow \V$ sends each manifold $M$ to the representable presheaf determined by $C^\infty(M)$, so since binary products in $\C_\text{C}^\op$ are given by $\otimes_\infty$ \pref{par:cinfty_rings} it follows that the image of the Yoneda embedding $y:\C_\text{C}^\op \rightarrow \V$ consists of the products $i(M) \times y(W)$ in $\V$, where $M$ is a manifold and $W$ is a Weil algebra.  There is a fully faithful functor $i':\Mf' \rightarrow \V$ from the category $\Mf'$ of smooth manifolds-with-boundary to $\V$, given by sending a manifold-with-boundary $K$ to the sheaf $i'(K) = \CinftyRing(C^\infty(K),\iota-):\C_\text{C} \rightarrow \Set$ on $\C_\text{C}^\op$ \cite[III.9.6]{Kock:SDG}.  The line object $R = i(\RR)$ is a preordered ring in $\V$ when we take $R_+ = i'(\RR_+)$ \cite[III.11.3]{Kock:SDG}, and we deduce the following:
\end{ExaSub}

\begin{PropSub}\label{thm:cah}
The preordered ring $R = i(\RR)$ in the Cahiers topos $\V$ satisfies the hypotheses of Theorem \ref{thm:enr_commutant_theorem} when we take
$$\G = \{i'(K_{nr}) \times y(W) \mid \textnormal{$n\in \NN$, $r \in (0,\infty)$, $W$ any Weil algebra}\} \hookrightarrow \V$$
where for every natural number $n$ and every positive real number $r$ we write $K_{nr}$ to denote the closed ball of radius $r$ about the origin in $\RR^n$, considered as a manifold-with-boundary.  In particular, $\G$ is a generating class for $\V$.
\end{PropSub}
\begin{proof}
First note that we have an isomorphism of preordered rings $\V(1,R) \cong \RR$.  For each $C \in \C_\text{C}$, it follows\footnote{Here we use the fact that every $C \in \C_\text{C}$ has a presentation $C \cong C^\infty(\RR^m)\slash J$ as required in \cite[III.11.4]{Kock:SDG}, and we note also that $\V(y(C),R) = \V(y(C),y(C^\infty(\RR)) \cong \CinftyRing(C^\infty(\RR),C) \cong \ca{C}$ (since $C^\infty(\RR)$ is the free $C^\infty$-ring on one generator) and that $\V(1,y(C)) = \V(y(\RR),y(C)) \cong \CinftyRing(C,\RR) \cong Z(J)$, where $Z(J) \subseteq \RR^m$ is the set of all common zeroes of $J$.} from \cite[III.11.4]{Kock:SDG} that the preorder carried by $\V(y(C),R)$ is \textit{pointwise with respect to global points}, in the sense that
$$\text{$f \lt g$ in $\V(y(C),R)$} \;\;\;\;\;\;\Leftrightarrow\;\;\;\;\;\; \forall x \in \V(1,y(C)) \;\;:\;\; \text{$f \cdot x \lt g \cdot x$ in $\V(1,R) \cong \RR$}.$$
From this it follows that for \textit{any} object $V$ of $\V$, the preorder carried by $\V(V,R)$ is also pointwise w.r.t. global points, since the following statements are equivalent\footnote{Indeed, we can show that successive pairs of statements are equivalent, using the Yoneda lemma and the fact that $1 = y(\RR)$.}: (1) $0 \lt f$ in $\V(V,R)$; (2) $f:V \rightarrow R$ factors through $R_+ \hookrightarrow R$; (3) for each morphism $v:y(C) \rightarrow V$, with $C \in \C_\text{C}$, the composite $f \cdot v:y(C) \rightarrow R$ factors through $R_+ \hookrightarrow R$; (4) for each morphism $v:y(C) \rightarrow V$ with $C \in \C_\text{C}$, $0 \lt f \cdot v$ in $\V(y(C),R)$; (5) for all $C \in \C_\text{C}$ and all $1 \xrightarrow{x} y(C) \xrightarrow{v} V$ in $\V$, $0 \lt f \cdot v \cdot x$ in $\V(1,R)$; (6) for all $x \in \V(1,V)$, $0 \lt f \cdot x$ in $\V(1,R) \cong \RR$.

Hence, letting $K = i'(K_{nr})$ with $n \in \NN$ and $r \in (0,\infty)$, the preordered ring $\V(K,R)$ is isomorphic to the partially ordered ring $C^\infty(K_{nr})$ of all smooth real-valued functions on $K_{nr}$, under the pointwise order.  Hence since every such function is bounded, we deduce by \cite[10.21(3)]{Lu:CvxAffCmt} that $\V(K,R) \cong C^\infty(K_{nr})$ is a firmly archimedean preordered $\DD$-algebra.

Now given also a Weil algebra $W$, we know that $T = y(W)$ has exactly one global point $t:1 \rightarrow T$ in $\V$; indeed since $W$ is a \textit{local} $C^\infty$-ring \cite[I.3, I.3.13]{MoeRey}, there is a unique homomorphism of $C^\infty$-rings $W \rightarrow \RR$ \cite[I.3.8, I.3.18]{MoeRey}.  By pullback, the point $t$ induces a morphism $t' = (1,t \cdot !):K \rightarrow K \times T$, and we find that every global point $1 \rightarrow K \times T$ is of the form $(x,t) = t' \cdot x$ for a unique point $x:1 \rightarrow K$.  Hence, since the preorders on $\V(K,R)$ and $\V(K \times T,R)$ are pointwise w.r.t. global points, we deduce that for all $f,g \in \V(K \times T,R)$,
$$f \lt g\;\text{in $\V(K \times T,R)$}\;\Leftrightarrow\; (\forall x \in \V(1,K)\::\:f \cdot t' \cdot x \lt g \cdot t' \cdot x) \;\Leftrightarrow\; f \cdot t' \lt g \cdot t'\;\text{in $\V(K,R)$}\;.$$
Hence since $\V(K,R)$ is a firmly archimedean preordered $\DD$-algebra, it follows that $\V(K\times T,R)$ is a firmly archimedean preordered $\DD$-algebra.

Let us say that an object $W$ of the topos $\V$ is \textit{covered} by a given class of objects $\sH$ in $\V$ if there exists a jointly epimorphic family $(f_\gamma:V_\gamma \rightarrow W)_{\gamma \in \Gamma}$ whose domains $V_\gamma$ lie in $\sH$.  By \ref{exa:line_obj_cah_top}, $\V$ has a generating class consisting of the objects $i(M) \times T$ where $M$ is a manifold and $T = y(W)$ for a Weil algebra $W$.  Hence, in order to show that $\G$ is a generating class, it suffices to show that each such object $i(M) \times T$ is covered by $\G$.  But since jointly epimorphic families are preserved by the left adjoint functor $(-) \times T:\V \rightarrow \V$, it suffices to show that $i(M)$ is covered by $\K = \{i'(K_{nr}) \mid n \in \NN, r \in (0,\infty)\}$.  But since $i:\Mf \rightarrow \V$ sends open covers to jointly epimorphic families \cite[III.9.4, III.4.C]{Kock:SDG}, it follows that $i(M)$ is covered by the singleton class $\{i(\RR^n)\}$ where $n$ is the dimension of $M$.  Hence it suffices to show that $i(\RR^n)$ is covered by $\K$, so it suffices to show that the inclusions $K_{nr} \hookrightarrow \RR^n$ $(r>0)$ are sent by $i':\Mf' \rightarrow \V$ to a jointly epimorphic family in $\V$.  But the family consisting of the interiors $\mathring{K}_{nr} \hookrightarrow \RR^n$ of these subsets $K_{nr}$ is sent by $i:\Mf \rightarrow \V$ to a jointly epimorphic family, and the result follows.
\end{proof}

\bibliographystyle{amsplain}
\bibliography{bib}

\end{document}

%% file: pre.tex
\usepackage{amsmath, amsthm, amssymb, stmaryrd}
\usepackage[all]{xy}
\usepackage{fullpage}
\usepackage{titlesec}
\usepackage{mathrsfs}
\usepackage{amsbsy}
\usepackage{turnstile}
\usepackage{bbm}
\usepackage{yhmath}
\usepackage{tensor}
\usepackage{graphics}
\usepackage{enumitem}
\usepackage{calligra}
\usepackage{verbatim}
\usepackage{setspace}
\usepackage{hhline}
\usepackage{array}

\usepackage[pdftex,bookmarks=true]{hyperref}
\usepackage[svgnames]{xcolor}
\hypersetup{
    linkbordercolor = {PaleTurquoise},
    citebordercolor = {PaleGreen},
}

\usepackage{titletoc}

\titlecontents{section}
[0pt]                                               
{}%
{\contentsmargin{0pt}                               
    \thecontentslabel\enspace%
    }
{\contentsmargin{0pt}}                        
{\titlerule*[.5pc]{.}\contentspage}                 
[]                                                  

\makeatletter
\newcommand*{\@old@slash}{}\let\@old@slash\slash
\def\slash{\relax\ifmmode\delimiter"502F30E\mathopen{}\else\@old@slash\fi}
\makeatother

\titleformat{\section}{\normalsize\bfseries}{\thesection}{1em}{}
\titleformat{\subsection}{\normalsize\bfseries}{\thesubsection}{1em}{}

\numberwithin{equation}{subsection}

\theoremstyle{plain}

\newtheorem{PropSub}[subsection]{Proposition}
\newtheorem{LemSub}[subsection]{Lemma}
\newtheorem{CorSub}[subsection]{Corollary}
\newtheorem{ThmSub}[subsection]{Theorem}

\newtheorem{PropSubSub}[subsubsection]{Proposition}

\newtheorem{CorSubSub}[subsubsection]{Corollary}
\newtheorem{ThmSubSub}[subsubsection]{Theorem}

\theoremstyle{definition}

\newtheorem{DefSub}[subsection]{Definition}
\newtheorem{ExaSub}[subsection]{Example}
\newtheorem{RemSub}[subsection]{Remark}
\newtheorem{ParSub}[subsection]{}

\newtheorem{DefSubSub}[subsubsection]{Definition}
\newtheorem{RemSubSub}[subsubsection]{Remark}
\newtheorem{ExaSubSub}[subsubsection]{Example}
\newtheorem{ParSubSub}[subsubsection]{}

\newtheorem*{Acknowledgement}{Acknowledgement}

\newcommand*{\emptybox}{\leavevmode\hbox{}}

\newdir{ >}{{}*!/-10pt/@{>}}

\DeclareMathAlphabet{\mathpzc}{OT1}{pzc}{m}{it}
\DeclareMathAlphabet{\mathcalligra}{T1}{calligra}{m}{n}

\newcommand{\pref}[1]{\textnormal{(\ref{#1})}}

\newcommand{\A}{\ensuremath{\mathscr{A}}}
\newcommand{\B}{\ensuremath{\mathscr{B}}}
\newcommand{\C}{\ensuremath{\mathscr{C}}}

\newcommand{\F}{\ensuremath{\mathscr{F}}}
\newcommand{\G}{\ensuremath{\mathscr{G}}}
\newcommand{\sH}{\ensuremath{\mathscr{H}}}

\newcommand{\normalJ}{\ensuremath{\mathscr{J}}}
\newcommand{\J}{\ensuremath{{\kern -0.4ex \mathscr{J}}}}

\newcommand{\JF}{\ensuremath{{\kern -0.4ex \mathscr{J}_{\kern -0.05ex F}}}}

\newcommand{\K}{\ensuremath{\mathscr{K}}}

\newcommand{\sP}{\ensuremath{\mathscr{P}}}

\newcommand{\sS}{\ensuremath{\mathscr{S}}}
\newcommand{\T}{\ensuremath{\mathscr{T}}}
\newcommand{\U}{\ensuremath{\mathscr{U}}}
\newcommand{\V}{\ensuremath{\mathscr{V}}}
\newcommand{\W}{\ensuremath{\mathscr{W}}}
\newcommand{\X}{\ensuremath{\mathscr{X}}}
\newcommand{\Y}{\ensuremath{\mathscr{Y}}}

\newcommand{\uV}{\ensuremath{\mkern2mu\underline{\mkern-2mu\mathscr{V}\mkern-6mu}\mkern5mu}}

\newcommand{\CAT}{\ensuremath{\operatorname{\textnormal{\text{CAT}}}}}

\newcommand{\VCAT}{\ensuremath{\V\textnormal{-\text{CAT}}}}

\newcommand{\CC}{\ensuremath{\mathbb{C}}}
\newcommand{\DD}{\ensuremath{\mathbb{D}}}

\newcommand{\MM}{\ensuremath{\mathbb{M}}}
\newcommand{\NN}{\ensuremath{\mathbb{N}}}
\newcommand{\PP}{\ensuremath{\mathbb{P}}}
\newcommand{\QQ}{\ensuremath{\mathbb{Q}}}

\newcommand{\RR}{\ensuremath{\mathbb{R}}}
\newcommand{\SSS}{\ensuremath{\mathbb{S}}}
\newcommand{\TT}{\ensuremath{\mathbb{T}}}

\newcommand{\UU}{\ensuremath{\mathbb{U}}}

\newcommand{\ZZ}{\ensuremath{\mathbb{Z}}}

\newcommand{\ob}{\ensuremath{\operatorname{\textnormal{\textsf{ob}}}}}
\newcommand{\mor}{\ensuremath{\operatorname{\textnormal{\textsf{mor}}}}}

\newcommand{\ca}[1]{\scalebox{0.85}{\raisebox{0.3ex}{$|$}} #1\scalebox{0.85}{\raisebox{0.3ex}{$|$}}}

\newcommand{\Shv}{\ensuremath{\textnormal{\textnormal{Shv}}}}

\newcommand{\Th}{\ensuremath{\textnormal{Th}}}

\newcommand{\ThJ}{\ensuremath{\Th_{\kern -0.5ex \normalJ}}}
\newcommand{\SubThJ}{\ensuremath{\textnormal{SubTh}_{\kern -0.5ex \normalJ}}}
\newcommand{\SubTh}{\ensuremath{\textnormal{SubTh}}}
\newcommand{\Vfp}{\ensuremath{\V_{\kern -0.5ex fp}}}
\newcommand{\TTperpJ}{\ensuremath{\TT^\perp_{\kern-.1ex\scriptscriptstyle\J}}}
\newcommand{\TTperpJwrt}[1]{\ensuremath{\TT^\perp_{\kern-.1ex\scriptscriptstyle\J#1}}}
\newcommand{\PPperpJ}{\ensuremath{\PP^\perp_{\kern-.1ex\scriptscriptstyle\J}}}
\newcommand{\PPperpJwrt}[1]{\ensuremath{\PP^\perp_{\kern-.1ex\scriptscriptstyle\J#1}}}

\newcommand{\Ev}{\ensuremath{\textnormal{\textsf{Ev}}}}

\newcommand{\inJ}[1]{#1 \in \normalJ\kern -1.2ex}

\newcommand{\VCATJ}{\VCAT_{\kern -0.7ex \normalJ}}
\newcommand{\PhiJ}{\Phi_{\kern -0.8ex \normalJ}}
\newcommand{\MndJ}{\Mnd_{\kern -0.8ex \normalJ}}

\newcommand{\Set}{\ensuremath{\operatorname{\textnormal{\text{Set}}}}}
\newcommand{\FinCard}{\ensuremath{\operatorname{\textnormal{\text{FinCard}}}}}
\newcommand{\DFin}{\ensuremath{\operatorname{\textnormal{\text{DFin}}}}}

\newcommand{\Bool}{\ensuremath{\operatorname{\textnormal{\text{Bool}}}}}
\newcommand{\SLat}{\ensuremath{\operatorname{\textnormal{\text{SLat}}}}}

\newcommand{\Mod}[1]{\ensuremath{#1\textnormal{-\text{Mod}}}}
\newcommand{\Aff}[1]{\ensuremath{#1\textnormal{-\text{Aff}}}}
\newcommand{\Cvx}[1]{\ensuremath{#1\textnormal{-\text{Cvx}}}}

\newcommand{\Mat}{\ensuremath{\textnormal{\text{Mat}}}}

\newcommand{\Top}{\ensuremath{\operatorname{\textnormal{\text{Top}}}}}

\newcommand{\LCHaus}{\ensuremath{\operatorname{\textnormal{\text{LCHaus}}}}}
\newcommand{\Conv}{\ensuremath{\operatorname{\textnormal{\text{Conv}}}}}

\newcommand{\OCart}{\ensuremath{\operatorname{\textnormal{\text{OCart}}}}}
\newcommand{\Mf}{\ensuremath{\operatorname{\textnormal{\text{Mf}}}}}
\newcommand{\Fro}{\ensuremath{\operatorname{\textnormal{\text{Fr\"o}}}}}
\newcommand{\Diff}{\ensuremath{\operatorname{\textnormal{\text{Diff}}}}}
\newcommand{\Cah}{\ensuremath{\operatorname{\textnormal{\text{Cah}}}}}

\newcommand{\Mnd}{\ensuremath{\operatorname{\textnormal{\text{Mnd}}}}}

\newcommand{\CMon}{\ensuremath{\operatorname{\textnormal{\text{CMon}}}}}

\newcommand{\Alg}[1]{\ensuremath{#1\kern -.5ex\operatorname{\textnormal{-\text{Alg}}}}}
\newcommand{\CAlg}[1]{\ensuremath{#1\kern -.5ex\operatorname{\textnormal{-\text{CAlg}}}}}
\newcommand{\CinftyRing}{\ensuremath{C^\infty\kern -.5ex\operatorname{\textnormal{-\text{Ring}}}}}

\newcommand{\CRProf}{\ensuremath{\operatorname{\textnormal{\text{CRProf}}}}}
\newcommand{\CRProfJ}{\ensuremath{\CRProf_{\kern -0.6ex \normalJ}}}

\newcommand{\Adjn}[6]{\xymatrix {#1 \ar@/_0.5pc/[rr]_{#2}^(0.4){#4}^(0.6){#5}^{\top} & & #6 \ar@/_0.5pc/[ll]_{#3}}}
\newcommand{\Equiv}[6]{\xymatrix {#1 \ar@/_0.5pc/[rr]_{#2}^(0.4){#4}^(0.6){#5}^{\sim} & & #6 \ar@/_0.5pc/[ll]_{#3}}}

\newcommand{\lt}{\leqslant}
\newcommand{\gt}{\geqslant}

\newcommand{\op}{\ensuremath{\textnormal{op}}}

\newcommand{\aff}{\ensuremath{\textnormal{aff}}}

\newcommand{\pushoutcorner}{\ar@{}[dr]|(.3)\ulcorner}
\newcommand{\pullbackcorner}{\ar@{}[dr]|(.3)\lrcorner}

\newcommand{\cmt}[1]{}